\documentclass[a4paper,reqno]{amsart}
\usepackage{amssymb, amsmath, amsthm, mathrsfs, ascmac,color}
\usepackage[pdftex]{graphicx} 
\usepackage[subrefformat=parens]{subcaption}
\usepackage[top=3cm, bottom=3cm, left=3cm,right=3cm]{geometry}
\usepackage{comment}
\usepackage{caption}
\usepackage{here}

\usepackage[foot]{amsaddr}
\newtheoremstyle{mystyle}
  {}
  {}
  {\normalfont}
  { }
  {\bfseries}
  {}
  {10pt}
  { }
\theoremstyle{mystyle}
\newtheorem{thm}{Theorem}[section]
\newtheorem{prop}[thm]{Proposition}
\newtheorem{lem}[thm]{Lemma}

\newtheorem{rmk}{Remark}
\makeatletter

\@addtoreset{equation}{section}
\makeatother

\newcommand{\ind}{\boldsymbol 1}
\newcommand{\dd}{\mathrm d}
\newcommand{\pd}{\partial}
\newcommand{\ee}{\mathrm e}

\newcommand{\EE}{\mathbb E}
\newcommand{\VV}{\mathrm{Var}}
\newcommand{\cov}{\mathrm{Cov}}

\newcommand{\pto}{\stackrel{p}{\to}}
\newcommand{\dto}{\stackrel{d}{\to}}

\newcommand{\TT}{\mathsf T}

\allowdisplaybreaks
\begin{document}
\bibliographystyle{plain}

\title
[Parametric estimation for SPDEs based on temporal and spatial increments ] 
{
Parametric estimation for linear parabolic SPDEs
in two space dimensions based on temporal and spatial increments
}
\date{}
\author{Yozo Tonaki$^1$}
\author{Yusuke Kaino$^2$}
\author{Masayuki Uchida$^{1,3}$}
\address{
$^1$Graduate School of Engineering Science, Osaka University
}
\address{
$^2$Graduate School of Maritime Sciences, Kobe University
}
\address{
$^3$Center for Mathematical Modeling and Date Science (MMDS), Osaka University
and JST CREST
}


\maketitle

\begin{abstract}
We deal with parameter estimation for a linear parabolic second-order 
stochastic partial differential equation 
in two space dimensions driven by two types of $Q$-Wiener processes 
based on high frequency data with respect to time and space.
We propose minimum contrast estimators of the coefficient parameters
based on temporal and spatial squared increments, 
and provide adaptive estimators of the coefficient parameters 
based on an approximate coordinate process.
We also give an example and simulation results of the proposed estimators.
\end{abstract}

\section{Introduction}\label{sec1}
We consider the following linear parabolic stochastic partial differential equation
(SPDE) in two space dimensions:
\begin{align}
\dd X_t(y,z)
&=\biggl\{
\theta_2
\biggl(\frac{\pd^2}{\pd y^2}+\frac{\pd^2}{\pd z^2}\biggr)
+\theta_1\frac{\pd}{\pd y} 
+\eta_1\frac{\pd}{\pd z} 
+\theta_0 
\biggr\} X_t(y,z) \dd t
\nonumber
\\
&\qquad
+\sigma\dd W_t^{Q}(y,z),
\quad (t,y,z) \in [0,1] \times D
\label{2d_spde}
\end{align}
with the initial condition $X_0 = \xi$ and the Dirichlet boundary condition
\begin{equation*}
X_t(y,z) = 0, \quad (t,y,z) \in [0,1] \times \pd D,
\end{equation*}
where $D=(0,1)^2$, 
$W_t^Q$ is a $Q$-Wiener process in a Sobolev space on $D$,
the initial value $\xi$ is independent of $W_t^Q$,
$\theta=(\theta_0, \theta_1, \eta_1, \theta_2)$ and $\sigma$ are unknown parameters,
$(\theta_0, \theta_1, \eta_1, \theta_2,\sigma) \in \Theta$, 
the parameter space $\Theta$ is a compact convex subset of 
$\mathbb R^3 \times (0,\infty)^2$, 
$\theta^*=(\theta_0^*, \theta_1^*, \eta_1^*, \theta_2^*)$ and $\sigma^*$ 
are true values of the parameters and we assume that 
$(\theta^*,\sigma^*)$ belongs to the interior of $\Theta$.

SPDEs are, roughly speaking, partial differential equations with a random term 
and they are used in various fields 
such as engineering, economics, and biology.
In particular, the linear parabolic SPDEs are basic and important equations,
including the stochastic heat equation, and appear in many situations.
For applications of SPDEs to 
geophysical fluid dynamics and biology,
see Piterbarg and Ostrovskii \cite{Piterbarg_Ostrovskii1997}, 
Tuckwell \cite{Tuckwell2013} and Altmeyer et al.\,\cite{Altmeyer_etal2022}.

Statistical inference for SPDEs 
has been developed by many researchers,
see for instance, Cialenco et al.\,\cite{Cialenco_etal2018},
Cialenco and Glatt-Holtz \cite{Cialenco_Glatt-Holtz2011},
H{\"u}bner et al.\,\cite{Hubner_etal1993},
H{\"u}bner and Rozovskii \cite{Hubner_Rozovskii1995},
Lototsky \cite{Lototsky2003},
and Mahdi Khalil and Tudor \cite{MahdiKhalil_Tudor2019}.
As for discrete observations, 
see Markussen \cite{Markussen2003},
Bibinger and Trabs \cite{Bibinger_Trabs2020},
Chong  \cite{Chong2020,Chong2019arXiv},
Cialenco et al.\,\cite{Cialenco_etal2020},
Cialenco and Huang \cite{Cialenco_Huang2020},
Kaino and Uchida \cite{Kaino_Uchida2020,Kaino_Uchida2021},
Hildebrandt and Trabs 
\cite{Hildebrandt_Trabs2021arXiv},
Tonaki et al.\,\cite{TKU2022arXiv2}
and references therein.
For an overview of existing theories, see 
Lototsky \cite{Lototsky2009} and Cialenco \cite{Cialenco2018}.
Recently, Tonaki et al.\,\cite{TKU2022arXiv1} proposed the estimators 
of the coefficient parameters of SPDE \eqref{2d_spde} based on temporal squared increments. 
They first constructed the estimators of the parameters appearing 
in the eigenfunctions of the differential operator of SPDE \eqref{2d_spde} 
and utilized them to approximate the coordinate processes defined 
by \eqref{cor-pro} below. 
Since the coordinate process was an Ornstein-Uhlenbeck process, 
they next estimated each coefficient parameter 
based on statistical inference for stochastic differential equations.  
The results showed that the estimators of the coefficient parameters 
had asymptotic normality with the convergence rate $\sqrt{n}$ for some $n \le N$,
where $N$ is the number of temporal observations.
Hildebrandt and Trabs \cite{Hildebrandt_Trabs2021} 
treated the following linear parabolic SPDE 
in one space dimension 
\begin{equation*}
\dd X_t(y)
=\biggl(
\theta_2 \frac{\pd^2}{\pd y^2}
+\theta_1\frac{\pd}{\pd y} 
+\theta_0 
\biggr) X_t(y) \dd t
+\sigma\dd B_t(y),
\quad t \ge 0, y \in (0,1)
\end{equation*}
with the initial condition $X_0 \in L^2((0,1))$
and the Dirichlet boundary condition
$X_t(0) = X_t(1) = 0$,
where $B_t$ is a cylindrical Brownian motion in a Sobolev space on $(0,1)$.
They proposed minimum contrast estimators based on double increments, 
i.e., temporal and spatial increments, 
and showed that the estimators have asymptotic normality with the convergence rate
$\sqrt{MN}$, where $M$ and $N$ are the number of 
spatial and temporal observations, respectively.

In this paper, we apply the method of Hildebrandt and Trabs \cite{Hildebrandt_Trabs2021}
to SPDE \eqref{2d_spde} in two space dimensions, 
and propose minimum contrast estimators of the coefficient parameters 
$(\theta_1,\eta_1,\theta_2,\sigma^2)$ of the SPDE
driven by two types of $Q$-Wiener processes 
based on temporal and spatial increments utilizing thinned data with respect to space.
Moreover, using thinned data with respect to time,
we derive parametric adaptive estimators of the coefficient parameters 
of the SPDE for each driving noise 
based on approximate coordinate processes, 
where coordinate processes are  Ornstein-Uhlenbeck processes.
For details of parametric adaptive estimators for diffusion processes, see for example, 
Uchida and Yoshida \cite{Uchida_Yoshida2012} and references therein.
For simplicity, parametric adaptive estimators are referred to as adaptive estimators throughout this paper.
Note that our method in this paper addresses both time and space increments 
and is inherently different from a method
in Tonaki et al.\,\cite{TKU2022arXiv1}
because our method cannot specify the method based on only time increments 
in Tonaki et al.\,\cite{TKU2022arXiv1}.
The main result of this paper is 
that the minimum contrast estimators are bounded 
in probability at the rate $\sqrt{m N}=O(N)$,
where $N$ is the number of temporal observations and 
$m$ is the number of spatially thinned data. 
Furthermore, the adaptive estimators have asymptotic normality 
with the convergence rate $\sqrt{n}$ for some $n \le N$.
It is also revealed that  
the convergence rates of both estimators are faster 
than those of Tonaki et al.\,\cite{TKU2022arXiv1}.

This paper is organized as follows.
In Section \ref{sec2}, we state main results.
We propose minimum contrast estimators and adaptive estimators
of the coefficient parameters 
in SPDE \eqref{2d_spde} driven by two types of $Q$-Wiener processes,
and show the asymptotic properties of these estimators.
Section \ref{sec3} gives an example and simulation studies of the proposed estimators.
In Section \ref{sec4}, we provide the proofs.

\section{Main results}\label{sec2}

Let $(\Omega,\mathscr F, \{\mathscr F_t\}_{t\ge0}, P)$ be a stochastic basis
with usual conditions, and let $\{w_{\ell_1,\ell_2}\}_{\ell_1,\ell_2 \ge 1}$ 
be independent real valued standard Brownian motions on this basis.

We set the differential operator $A_\theta$ by 
\begin{equation*}
-A_\theta = 
\theta_2\biggl(\frac{\pd^2}{\pd y^2} + \frac{\pd^2}{\pd z^2} \biggr)
+ \theta_1\frac{\pd}{\pd y} + \eta_1\frac{\pd}{\pd z} + \theta_0.
\end{equation*}
Notice that SPDE \eqref{2d_spde} can be expressed as
\begin{equation*}
\dd X_t(y,z) = -A_\theta X_t(y,z) \dd t + \sigma \dd W_t^Q(y,z).
\end{equation*}
The domain of $A_\theta$ is $\mathscr D(A_\theta) = H^2(D) \cap H_0^1(D)$,
where $H^p(D)$ is the $L^2$-Sobolev space of order $p \in \mathbb N$ and
$H_0^1(D)$ is the closure of $C_{\mathrm{c}}^{\infty}(D)$ in $H^1(D)$.
The eigenfunctions $e_{\ell_1,\ell_2}$ of the differential operator $A_\theta$ 
and the corresponding eigenvalues $\lambda_{\ell_1,\ell_2}$ are given by 
\begin{equation*}
e_{\ell_1,\ell_2}(y,z)=e_{\ell_1}^{(1)}(y)e_{\ell_2}^{(2)}(z),
\quad
\lambda_{\ell_1,\ell_2}=\theta_2(\pi^2(\ell_1^2+\ell_2^2)+\Gamma),
\end{equation*}
for $\ell_1,\ell_2 \ge 1$ and $y,z\in [0,1]$, where
\begin{equation*}
e_{\ell_1}^{(1)}(y)
=\sqrt 2 \sin(\pi\ell_1 y) \ee^{-\kappa y/2},
\quad
e_{\ell_2}^{(2)}(z)
=\sqrt 2 \sin(\pi\ell_2 z) \ee^{-\eta z/2},
\end{equation*}
\begin{equation*}
\kappa=\frac{\theta_1}{\theta_2}, 
\quad
\eta=\frac{\eta_1}{\theta_2},
\quad
\Gamma=-\frac{\theta_0}{\theta_2} +\frac{\kappa^2+\eta^2}{4}.
\end{equation*}
The eigenfunctions $\{e_{\ell_1,\ell_2}\}_{\ell_1,\ell_2\ge1}$ 
are orthonormal with respect to the weighted $L^2$-inner product
\begin{equation*}
\langle u, v \rangle
:= \langle u, v \rangle_\theta 
= \int_0^1 \int_0^1 u(y,z)v(y,z)\ee^{\kappa y} \ee^{\eta z} \dd y \dd z, 
\quad 
\| u \| = \sqrt{\langle u, u \rangle}
\end{equation*}
for $u,v\in L^2(D)$. 

Although it is possible to consider the general initial condition
as in Tonaki et al.\,\cite{TKU2022arXiv1}, for simplicity, 
we impose the following condition on the initial value $\xi$ 
of SPDE \eqref{2d_spde}.
\begin{enumerate}
\item[\textbf{[A1]}]
The initial value $\xi \in L^2(D)$ is non-random and
$\| A_\theta^{(1+\alpha)/2}\xi \| < \infty$.
\end{enumerate}
Moreover, we assume $\lambda_{1,1}^* > 0$ 
so that $A_\theta$ is a positive self-adjoint operator.

We first consider the $Q_1$-Wiener process $W_t^{Q_1}$ defined as follows.
\begin{equation}\label{QW}
\langle W_t^{Q_1}, u \rangle = 
\sum_{\ell_1,\ell_2\ge1} \lambda_{\ell_1,\ell_2}^{-\alpha/2} 
\langle u, e_{\ell_1,\ell_2}\rangle  w_{\ell_1,\ell_2}(t)
\end{equation}
for $u \in L^2(D)$, $t \ge 0$ and $\alpha \in (0,2)$.
$\alpha$ is known and its restriction guarantees 
the estimability of the coefficient parameters 
$(\theta_1,\eta_1,\theta_2,\sigma^2)$.
We leave the study of the estimability of $\alpha$ for future work.
There exists a unique mild solution of SPDE \eqref{2d_spde} driven 
by $Q_1$-Wiener process \eqref{QW}, 
which is given by 
\begin{equation*}
X_t=\ee^{-t A_\theta}\xi+\sigma\int_0^t \ee^{-(t-s)A_\theta}\dd W_s^{Q_1},
\end{equation*}
where $\ee^{-t A_\theta} u 
= \sum_{\ell_1,\ell_2 \ge 1} \ee^{-\lambda_{\ell_1,\ell_2}t}
\langle u, e_{\ell_1,\ell_2}\rangle e_{\ell_1,\ell_2}$ for $u \in L^2(D)$.
By setting the coordinate process 
\begin{equation}\label{cor-pro}
x_{\ell_1,\ell_2}(t) = \langle X_t, e_{\ell_1,\ell_2} \rangle =
\ee^{-\lambda_{\ell_1,\ell_2} t} \langle \xi, e_{\ell_1,\ell_2} \rangle 
+\sigma \lambda_{\ell_1,\ell_2}^{-\alpha/2} 
\int_0^t \ee^{-\lambda_{\ell_1,\ell_2}(t-s)} \dd w_{\ell_1,\ell_2}(s),
\end{equation}
the random field $X_t(y,z)$ is spectrally decomposed as
\begin{equation*}
X_t(y,z)=\sum_{\ell_1,\ell_2\ge1} 
x_{\ell_1,\ell_2}(t)
e_{\ell_1}^{(1)}(y)e_{\ell_2}^{(2)}(z),
\quad t\ge0,\ y,z\in[0,1],
\end{equation*}
where $\{x_{\ell_1,\ell_2}\}_{\ell_1,\ell_2 \ge 1}$ are
one dimensional independent processes satisfying the Ornstein-Uhlenbeck dynamics
\begin{equation}\label{OU}
\dd x_{\ell_1,\ell_2}(t) =
-\lambda_{\ell_1,\ell_2} x_{\ell_1,\ell_2}(t)\dd t
+\sigma \lambda_{\ell_1,\ell_2}^{-\alpha/2} \dd w_{\ell_1,\ell_2}(t),
\quad
x_{\ell_1,\ell_2}(0) = \langle \xi, e_{\ell_1,\ell_2} \rangle.
\end{equation}

We assume that a mild solution $X$ is discretely observed on 
the grid $(t_i,y_j,z_k) \in [0,1]^3$ with 
\begin{equation*}
t_i = i\Delta, \quad 
y_j = j/M_1, \quad
z_k = k/M_2
\end{equation*}
for $i=0,\ldots,N$, $j=0,\ldots,M_1$ and $k=0,\ldots,M_2$, where $\Delta=1/N$.
That is, the data are discrete observations 
$\mathbb X_{M,N}=\{X_{t_i}(y_j,z_k);
i=0,\ldots,N, j=0,\ldots,M_1, k=0,\ldots, M_2 \}$,
where we set $M = M_1 M_2$.
For a sequence $\{a_n\}$, 
we write $a_n \equiv a$ if $a_n = a$ for some $a \in \mathbb R$ and all $n$.

\subsection{Minimum contrast estimation of the coefficient parameters}
\label{sec2-1}
In this subsection, we give the second-order moment of the triple increments 
and consider the minimum contrast estimation 
of the coefficient parameters $(\theta_1,\eta_1,\theta_2, \sigma^2)$.

Fix $b \in (0,1/2)$.
We consider the spatially thinned data 
$\mathbb X_{m,N}^{(1)}=\{X_{t_i}(\widetilde y_j,\widetilde z_k); 
i=0,\ldots,N, j=0,\ldots,m_1, k=0,\ldots, m_2\}$ 
of $\mathbb X_{M,N}$ such that
\begin{equation*}
b \le \widetilde y_0 < \widetilde y_1 < \cdots < \widetilde y_{m_1} \le 1-b,
\quad
b \le \widetilde z_0 < \widetilde z_1 < \cdots < \widetilde z_{m_2} \le 1-b,
\end{equation*}
$m=m_1 m_2$, $m=O(N)$, $N=O(m)$. 
For simplicity, 
we set $\widetilde y_0 = \widetilde z_0 = b$, 
$\widetilde y_{m_1} = \widetilde z_{m_2} = 1-b$ and $m_1=m_2$,
that is, 
$\widetilde y_j = b + j\delta$ and $\widetilde z_k = b + k\delta$,
where $\delta = \frac{1-2b}{m_1} = \frac{1-2b}{\sqrt{m}}$.

We introduce the following temporal and spatial increments
\begin{align*}
\Delta_i X (y,z) & = X_{t_{i}}(y,z)-X_{t_{i-1}}(y,z), 
\\
D_{j,k} X(t) & = X_t(\widetilde y_{j},\widetilde z_{k})
- X_t(\widetilde y_{j-1},\widetilde z_{k})
- X_t(\widetilde y_{j},\widetilde z_{k-1})
+ X_t(\widetilde y_{j-1},\widetilde z_{k-1})
\end{align*}
and consider triple increments
\begin{align*}
T_{i,j,k}X & = (\Delta_i \circ D_{j,k}) X
= (D_{j,k} \circ \Delta_i) X 
\\
&= X_{t_{i}}(\widetilde y_{j},\widetilde z_{k})
-X_{t_{i}}(\widetilde y_{j-1},\widetilde z_{k})
-X_{t_{i}}(\widetilde y_{j},\widetilde z_{k-1})
+X_{t_{i}}(\widetilde y_{j-1},\widetilde z_{k-1})
\\
&\qquad
-X_{t_{i-1}}(\widetilde y_{j},\widetilde z_{k})
+X_{t_{i-1}}(\widetilde y_{j-1},\widetilde z_{k})
+X_{t_{i-1}}(\widetilde y_{j},\widetilde z_{k-1})
-X_{t_{i-1}}(\widetilde y_{j-1},\widetilde z_{k-1}).
\end{align*}

Let $J_0$ be the Bessel function of the first kind of order 0:
\begin{equation*}
J_0(x) = 1 + \sum_{k\ge1} \frac{(-1)^k}{(k!)^2} \biggl(\frac{x}{2}\biggr)^{2k}.
\end{equation*}
For $\delta/\sqrt{\Delta} \equiv r \in(0,\infty)$, define
\begin{equation}\label{psi}
\psi_{r,\alpha}(\theta_2)
=\frac{2}{\theta_2\pi}
\int_0^\infty 
\frac{1-\ee^{-x^2}}{x^{1+2\alpha}}
\biggl(
J_0\Bigl(\frac{\sqrt{2}r x}{\sqrt{\theta_2}}\Bigr)
-2J_0\Bigl(\frac{r x}{\sqrt{\theta_2}}\Bigr)+1
\biggr) \dd x.
\end{equation}

\begin{prop}\label{prop1}
Let $\alpha \in (0,2)$ and $\delta/\sqrt{\Delta} \equiv r \in(0,\infty)$.
Under [A1], it holds that
\begin{equation*}
\EE[(T_{i,j,k}X)^2]
=\Delta^\alpha \sigma^2
\ee^{-\kappa(\widetilde y_{j-1}+\widetilde y_j)/2}
\ee^{-\eta (\widetilde z_{k-1}+\widetilde z_k)/2}
\psi_{r,\alpha}(\theta_2)
+ R_{i,j,k} + O(\Delta^{1+\alpha}),
\end{equation*}
where $\sum_{i=1}^N R_{i,j,k}=O(\Delta^{\alpha})$ uniformly in $j,k$.
\end{prop}

Let $\widetilde T_{i,j,k}X=T_{i,j,k}X+T_{i+1,j,k}X$ and 
$\nu=(\kappa, \eta, \theta_2, \sigma^2)$.
For $\alpha \in (0,2)$ and 
$\delta/\sqrt{\Delta} \equiv r \in(0,\infty)$, 
we introduce the contrast function 
\begin{align*}
K_{m,N}(\nu) &= 
\frac{1}{m}\sum_{k=1}^{m_2}\sum_{j=1}^{m_1} 
\Biggl\{
\frac{1}{N\Delta^{\alpha}}\sum_{i=1}^{N} (T_{i,j,k}X)^2
-f_{r,\alpha}\Bigl(\frac{\widetilde y_{j-1}+\widetilde y_j}{2},
\frac{\widetilde z_{k-1}+\widetilde z_k}{2}:\nu\Bigr)
\Biggr\}^2
\\
&\qquad+\frac{1}{m}\sum_{k=1}^{m_2}\sum_{j=1}^{m_1} 
\Biggl\{
\frac{1}{N(2\Delta)^{\alpha}}\sum_{i=1}^{N-1} (\widetilde T_{i,j,k}X)^2
-f_{r/\sqrt{2},\alpha}\Bigl(\frac{\widetilde y_{j-1}+\widetilde y_j}{2},
\frac{\widetilde z_{k-1}+\widetilde z_k}{2}:\nu\Bigr)
\Biggr\}^2,
\end{align*}
where 
$f_{r,\alpha}(y,z:\nu)=\sigma^2\ee^{-\kappa y}\ee^{-\eta z}\psi_{r,\alpha}(\theta_2)$.
Set
\begin{equation*}
\underline{\theta_2}(r,\alpha)=
\begin{cases}
-\frac{r^2}{8\log(2^{\alpha/2}-1)}, & \alpha \in (0,1),
\\
0, & \alpha \in [1,2).
\end{cases}
\end{equation*}
Let $\Xi$ be a compact convex subset of 
$\mathbb R^2 \times (\underline{\theta_2}(r,\alpha),\infty) \times(0,\infty)$
and define the minimum contrast estimator of $\nu$ by
\begin{equation*}
\hat\nu 
= \underset{\nu \in \Xi}{\mathrm{argmin}}\,K_{m,N}(\nu).
\end{equation*}

\begin{thm}\label{th1}
Under [A1], as $m \to \infty$ and $N \to \infty$, 
\begin{equation*}
\sqrt{m N}(\hat\nu-\nu^*) = O_p(1).
\end{equation*}
\end{thm}

\begin{rmk}\label{rmk1}
\ 
\begin{enumerate}
\item[(i)]
By using the estimators 
$\hat \nu = (\hat \kappa, \hat \eta, \hat \theta_2, \hat \sigma^2)$, 
the coefficient parameters $\theta_1$ and $\eta_1$ can be estimated by 
$\hat \theta_1 = \hat \kappa \hat \theta_2$ and
$\hat \eta_1 = \hat \eta \hat \theta_2$, respectively. 
Therefore, the estimators of the coefficient parameters 
$(\theta_2, \theta_1, \eta_1, \sigma^2)$ are bounded in probability 
at the rate $\sqrt{mN}$.
Note that $\sqrt{mN}=O(N)$.
Since the estimators proposed by Tonaki et al.\,\cite{TKU2022arXiv1}
have asymptotic normality with the convergence rate $\sqrt{n}$ for some $n \le N$, 
the estimators we proposed have a faster convergence rate 
than that of Tonaki et al.\,\cite{TKU2022arXiv1}.

\item[(ii)]
According to Step 3 in the proof of Theorem \ref{th1}, 
we find that our estimators do not have asymptotic normality unlike
Hildebrandt and Trabs \cite{Hildebrandt_Trabs2021}
because Proposition \ref{prop1} implies that
\begin{equation*}
\sum_{i=1}^{N} 
\biggl\{
\EE[(T_{i,j,k}X)^2]
-\Delta^{\alpha} f_{r,\alpha}\Bigl(\frac{\widetilde y_{j-1}+\widetilde y_j}{2},
\frac{\widetilde z_{k-1}+\widetilde z_k}{2}:\nu\Bigr)
\biggr\}
\end{equation*}
is bounded at the rate $\Delta^\alpha$ uniformly in $j,k$, 
and it does not converge to $0$ at the rate $\Delta^\alpha$.
This is caused by the fact that the approximation error cannot be neglected 
due to the increase in spatial dimension.

\end{enumerate}
\end{rmk}

\subsection{Adaptive estimation of the coefficient parameters}
\label{sec2-2}
In this subsection, we construct adaptive estimators of the coefficient parameters 
$(\theta_0, \theta_1,\eta_1, \theta_2, \sigma^2)$ with asymptotic normality 
based on the minimum contrast estimators and the coordinate processes.

Let $n \le N$ and $\widetilde t_i = \lfloor\frac{N}{n}\rfloor\frac{i}{N}$
for $i = 0,\ldots,n$.
We set the thinned data $\mathbb X_{M,n}^{(2)}
=\{X_{\widetilde t_i}(y_j,z_k); i=0,\ldots,n, j=0,\ldots,M_1, k=0,\ldots, M_2\}$
with respect to time of $\mathbb X_{M,N}$.

Since 
\begin{equation*}
x_{\ell_1,\ell_2}(t)
=2\int_0^1 \int_0^1 X_t(y,z)\sin(\pi \ell_1 y) \sin(\pi \ell_2 z)
\ee^{\kappa y/2}\ee^{\eta z/2} \dd y \dd z,
\end{equation*}
we set the approximate coordinate process by
\begin{equation*}
\hat x_{\ell_1,\ell_2}(t) = 
\frac{2}{M} \sum_{j=1}^{M_1}\sum_{k=1}^{M_2}
X_t(y_j,z_k)\sin(\pi \ell_1 y_j)\sin(\pi \ell_2 z_k)
\ee^{\hat \kappa y_j/2}\ee^{\hat \eta z_k/2}, 
\end{equation*}
and define the estimator of 
$\sigma_{\ell_1,\ell_2}:=\sigma\lambda_{\ell_1,\ell_2}^{-\alpha/2}$ in \eqref{OU}
as 
\begin{equation*}
\hat\sigma_{\ell_1,\ell_2}^2 = 
\sum_{i=1}^n (\hat x_{\ell_1,\ell_2}(\widetilde t_i)
-\hat x_{\ell_1,\ell_2}(\widetilde t_{i-1}))^2.
\end{equation*}
We propose the adaptive estimators of $\lambda_{\ell_1,\ell_2}$ and
$(\theta_0, \theta_1, \eta_1, \theta_2, \sigma^2)$ as follows.
\begin{equation*}
\check \lambda_{\ell_1,\ell_2} 
= \biggl(\frac{\hat\sigma^2}{\hat\sigma_{\ell_1,\ell_2}^2}\biggr)^{1/\alpha},
\quad
\check \theta_0 = -\hat\lambda_{1,1}+
\biggl(\frac{\hat\kappa^2+\hat\eta^2}{4}+2\pi^2\biggr)\hat\theta_2,
\end{equation*}
\begin{equation*}
\check \theta_2 = \frac{\check \lambda_{1,2}-\check \lambda_{1,1}}{3\pi^2},
\quad
\check \theta_1 = \hat \kappa \check \theta_2,
\quad
\check \eta_1 = \hat \eta \check \theta_2,
\quad
\check \sigma^2 = \frac{\hat \sigma^2}{\hat \theta_2} \check \theta_2.
\end{equation*}
Let
\begin{equation*}
\mathcal J = 2
\begin{pmatrix}
J_{1,1} & J_{1,2}
\\
J_{1,2}^\TT & J_{2,2}
\end{pmatrix},
\end{equation*}
where
\begin{equation*}
J_{1,1} = (\lambda_{1,1}^*)^2, 
\quad
J_{1,2} 
=\frac{-(\lambda_{1,1}^*)^2}{3\pi^2\theta_2^*} 
(\theta_1^*, \eta_1^*, \theta_2^*, (\sigma^*)^2),
\end{equation*}
\begin{equation*}
J_{2,2} 
=
\frac{(\lambda_{1,1}^*)^2 + (\lambda_{1,2}^*)^2}{9\pi^4 (\theta_2^*)^2 \alpha^2}
\begin{pmatrix}
(\theta_1^*)^2 & \theta_1^* \eta_1^* & \theta_1^* \theta_2^* & \theta_1^* (\sigma^*)^2
\\
\theta_1^* \eta_1^* & (\eta_1^*)^2 & \eta_1^* \theta_2^* & \eta_1^* (\sigma^*)^2
\\
\theta_1^* \theta_2^* & \eta_1^* \theta_2^* & (\theta_2^*)^2 & \theta_2^* (\sigma^*)^2
\\
\theta_1^* (\sigma^*)^2 & \eta_1^* (\sigma^*)^2 & \theta_2^* (\sigma^*)^2 & (\sigma^*)^4
\\
\end{pmatrix}
\end{equation*}
and $\TT$ denotes the transpose.

\begin{thm}\label{th2}
Let $\alpha\in(0,2)$ and $\delta/\sqrt{\Delta} \equiv r \in(0,\infty)$.
Under [A1], the followings hold.
\begin{itemize}
\item[(i)]
If $\frac{n^{1-\alpha+\tau}}{(M_1 \land M_2)^{2\tau}}\to0$
for some $0 < \tau \le 1$ such that $\tau <\alpha$, then
$\check \theta_0 \pto \theta_0^*$ as 
$M_1 \to \infty$, $M_2 \to \infty$ and $n \to \infty$.

\item[(ii)]
If $\frac{n^{2-\alpha+\tau}}{(M_1 \land M_2)^{2\tau}}\to0$
for some $0 < \tau \le 1$ such that $\tau <\alpha$, then
as $M_1 \to \infty$, $M_2 \to \infty$ and $n \to \infty$,
\begin{equation}\label{eq-Th2.3}
\sqrt{n}
\begin{pmatrix}
\check \theta_0 - \theta_0^*
\\
\check \theta_1 - \theta_1^*
\\
\check \eta_1 - \eta_1^*
\\
\check \theta_2 - \theta_2^*
\\
\check \sigma^2 - (\sigma^*)^2
\end{pmatrix}
\dto N (0, \mathcal J).
\end{equation}

\end{itemize}
\end{thm}

\begin{rmk}\label{rmk2}
\ 
\begin{enumerate}
\item[(i)]
The condition on $n \le N$ is more relaxed 
than Theorem 3.3 in Tonaki et al.\,\cite{TKU2022arXiv1}. 
In other words, 
if $n_1$ and $n_2$ $(n_1, n_2  \le N)$ satisfy
the conditions of Theorem 3.3 in Tonaki et al.\,\cite{TKU2022arXiv1} and Theorem \ref{th2},
respectively, then $n_1 \le n_2$,
which implies that 
the convergence rate of our estimators is faster than or equal to that of
the estimators of Tonaki et al.\,\cite{TKU2022arXiv1}.

\item[(ii)]
Set $2\widetilde J_{1,1}$ as the asymptotic variance 
of the estimator of $\theta_0$
in Theorem 3.3 in Tonaki et al.\,\cite{TKU2022arXiv1}. 
We obtain from
\begin{align*}
\widetilde J_{1,1} - J_{1,1}
&=
\frac{(\lambda_{1,1}^*)^2+(\lambda_{1,2}^*)^2}{9\pi^4 (\theta_2^*)^2}
\Biggl\{
\biggl(
\Bigl(\frac{\lambda_{1,1}^*}{\alpha}-\frac{\theta_0^*}{1-\alpha}\Bigr)^2
+\frac{3\pi^2\theta_2^*(\lambda_{1,1}^*)^2}
{\alpha((\lambda_{1,1}^*)^2+(\lambda_{1,2}^*)^2)}
\biggr)^2
\\
&\qquad+
\frac{(\lambda_{1,1}^*)^2}
{((\lambda_{1,1}^*)^2+(\lambda_{1,2}^*)^2)^2}
\biggl(
\Bigl(\frac{1}{\alpha^2}-1\Bigr)(\lambda_{1,2}^*)^2
-(\lambda_{1,1}^*)^2
\biggr)
\Biggr\}
\\
&\ge
(\lambda_{1,1}^*)^2 \Bigl(\frac{1}{\alpha^2}-1\Bigr)
\end{align*}
that $\widetilde J_{1,1} \ge J_{1,1}$ for $\alpha \in (0,1)$.

\item[(iii)]
Let $2\widetilde J_{2,2}$ be the asymptotic variance 
of the estimator of $(\theta_1, \eta_1, \theta_2, \sigma^2)$ 
in Theorem 3.3 in Tonaki et al.\,\cite{TKU2022arXiv1}. 
We have
\begin{equation*}
\widetilde J_{2,2}-J_{2,2}
=\frac{1}{9\pi^4(\theta_2^*)^2}\Bigl(\frac{1}{(1-\alpha)^2}-\frac{1}{\alpha^2}\Bigr)
(\theta_1^*, \eta_1^*, \theta_2^*, (\sigma^*)^2)^\TT
(\theta_1^*, \eta_1^*, \theta_2^*, (\sigma^*)^2),
\end{equation*}
which implies that 
$\widetilde J_{2,2}-J_{2,2}$ is non-negative definite if $\alpha \in [1/2,1)$.

\item[(iv)]
$\theta_0$ can be estimated when the driving noise is defined by \eqref{QW}. 
On the other hand,  $\theta_0$ cannot be estimated if the driving noise is a $Q$-Wiener process
such as \eqref{QW2} below.
\end{enumerate}
\end{rmk}

\subsection{
Estimation of the coefficient parameters in SPDE driven by $Q_2$-Wiener process}
\label{sec2-3}
We next consider SPDE \eqref{2d_spde} driven by the $Q_2$-Wiener process $W_t^{Q_2}$ 
given by
\begin{equation}\label{QW2}
\langle W_t^{Q_2}, u \rangle = 
\sum_{\ell_1,\ell_2\ge1} \mu_{\ell_1,\ell_2}^{-\alpha/2} 
\langle u, e_{\ell_1,\ell_2}\rangle  w_{\ell_1,\ell_2}(t)
\end{equation}
for $u \in L^2(D)$, $t \ge 0$ and $\alpha \in (0,2)$,
where $\mu_{\ell_1, \ell_2} = \pi^2(\ell_1^2 +\ell_2^2) + \mu_0$
and $\mu_0 \in (-2\pi^2, \infty)$. 
$\mu_0$ may or may not be known, 
the parameter space of $\mu_0$ is a compact convex subset of $(-2\pi^2, \infty)$
and the true value $\mu_0^*$ belongs to its interior.
In the case that the driving noise is the $Q_2$-Wiener process,
$\theta_0$ cannot be estimated, but
$(\theta_1, \eta_1, \theta_2, \sigma^2, \mu_0)$ 
can be estimated.

When the driving noise is the $Q_2$-Wiener process, 
under the same observation scheme as above, 
the second order moment of triple increments $\EE[(T_{i,j,k}X)^2]$
is achieved by replacing $\psi_{r,\alpha}(\theta_2)$ in Proposition \ref{prop1} with
$\widetilde \psi_{r,\alpha}(\theta_2) = \theta_2^\alpha \psi_{r,\alpha}(\theta_2)$,
see \eqref{eq-999-0} below. Therefore, we define the contrast function
\begin{align*}
\widetilde K_{m,N}(\nu) &= 
\frac{1}{m}\sum_{k=1}^{m_2}\sum_{j=1}^{m_1} 
\Biggl\{
\frac{1}{N\Delta^{\alpha}}\sum_{i=1}^{N} (T_{i,j,k}X)^2
-\widetilde f_{r,\alpha}\Bigl(\frac{\widetilde y_{j-1}+\widetilde y_j}{2},
\frac{\widetilde z_{k-1}+\widetilde z_k}{2}:\nu\Bigr)
\Biggr\}^2
\\
&\qquad+\frac{1}{m}\sum_{k=1}^{m_2}\sum_{j=1}^{m_1} 
\Biggl\{
\frac{1}{N(2\Delta)^{\alpha}}\sum_{i=1}^{N-1} (\widetilde T_{i,j,k}X)^2
-\widetilde f_{r/\sqrt{2},\alpha}\Bigl(\frac{\widetilde y_{j-1}+\widetilde y_j}{2},
\frac{\widetilde z_{k-1}+\widetilde z_k}{2}:\nu\Bigr)
\Biggr\}^2,
\end{align*}
where $\widetilde f_{r,\alpha}(y,z:\nu) 
=\sigma^2\ee^{-\kappa y}\ee^{-\eta z} \widetilde \psi_{r,\alpha}(\theta_2)$,
and set the minimum contrast estimator
$\tilde \nu = (\tilde \kappa, \tilde \eta, \tilde \theta_2, \tilde \sigma^2 )$ 
as follows. 
\begin{equation*}
\tilde \nu = \underset{\nu \in \Xi}{\mathrm{argmin}}\, \widetilde K_{m,N}(\nu).
\end{equation*}
\begin{thm}\label{th3}
Under [A1], $\sqrt{m N} (\tilde \nu -\nu^*) = O_p(1)$ 
as $m \to \infty$ and $N \to \infty$.
\end{thm}

Since the coordinate process of SPDE \eqref{2d_spde} driven by 
the $Q_2$-Wiener process is given by
\begin{equation*}
\dd x_{\ell_1,\ell_2}(t) =
-\lambda_{\ell_1,\ell_2} x_{\ell_1,\ell_2}(t)\dd t
+\sigma \mu_{\ell_1,\ell_2}^{-\alpha/2} \dd w_{\ell_1,\ell_2}(t),
\quad
x_{\ell_1,\ell_2}(0) = \langle \xi, e_{\ell_1,\ell_2} \rangle,
\end{equation*}
we use the approximate coordinate process 
\begin{equation*}
\tilde x_{\ell_1,\ell_2}(t) = 
\frac{2}{M} \sum_{j=1}^{M_1} \sum_{k=1}^{M_2} 
X_t(y_j,z_k) \sin(\pi \ell_1 y_j) \sin(\pi \ell_2 z_k) 
\ee^{\tilde \kappa y_j/2} \ee^{\tilde \eta z_k/2}
\end{equation*}
and estimate $\varsigma_{\ell_1,\ell_2} := \sigma \mu_{\ell_1,\ell_2}^{-\alpha/2}$ by
\begin{equation*}
\tilde \sigma_{\ell_1,\ell_2}^2 = 
\sum_{i=1}^n (\tilde x_{\ell_1,\ell_2}(\widetilde t_i)
-\tilde x_{\ell_1,\ell_2}(\widetilde t_{i-1}))^2.
\end{equation*}
When $\mu_0$ is known, the estimators of the coefficient parameters with asymptotic 
normality are constructed by 
\begin{equation*}
\bar \sigma^2 = \mu_{1,1}^{\alpha} \tilde \sigma_{1,1}^2, 
\quad 
\bar \theta_2 = \frac{\tilde \theta_2}{\tilde \sigma^2} \bar \sigma^2,
\quad 
\bar \theta_1 = \tilde \kappa \bar \theta_2,
\quad 
\bar \eta_1 = \tilde \eta \bar \theta_2,
\end{equation*}
and when $\mu_0$ is unknown, the estimators of the coefficient parameters with 
asymptotic normality and $\mu_0$ are defined as follows.
\begin{equation*}
\breve \sigma^2 = 
\Biggl\{
3\pi^2
\biggl(
\frac{1}{(\tilde \sigma_{1,2}^2)^{1/\alpha}} 
-\frac{1}{(\tilde \sigma_{1,1}^2)^{1/\alpha}}
\biggr)^{-1}
\Biggr\}^\alpha, 
\quad
\breve \mu_0 = 
\biggl( \frac{\breve \sigma ^2}{\tilde \sigma_{1,1}^2} \biggr)^{1/\alpha} - 2\pi^2,
\end{equation*}
\begin{equation*}
\breve \theta_2 = \frac{\tilde \theta_2}{\tilde \sigma^2} \breve \sigma^2,
\quad 
\breve \theta_1 = \tilde \kappa \breve \theta_2,
\quad 
\breve \eta_1 = \tilde \eta \breve \theta_2.
\end{equation*}
Let
\begin{equation*}
\mathcal K = 2
\begin{pmatrix}
(\theta_1^*)^2 & \theta_1^* \eta_1^* & \theta_1^* \theta_2^* & \theta_1^* (\sigma^*)^2
\\
\theta_1^* \eta_1^* & (\eta_1^*)^2 & \eta_1^* \theta_2^* & \eta_1^* (\sigma^*)^2
\\
\theta_1^* \theta_2^* & \eta_1^* \theta_2^* & (\theta_2^*)^2 & \theta_2^* (\sigma^*)^2
\\
\theta_1^* (\sigma^*)^2 & \eta_1^* (\sigma^*)^2 & \theta_2^* (\sigma^*)^2 & (\sigma^*)^4
\\
\end{pmatrix},
\quad
\mathcal L = \frac{2}{9\pi^4}
\begin{pmatrix}
L_{1,1} & L_{1,2}
\\
L_{1,2}^\TT & L_{2,2}
\end{pmatrix},
\end{equation*}
where
\begin{equation*}
L_{1,1} = \frac{2(\mu_{1,1}^*)^2(\mu_{1,2}^*)^2}{\alpha^2},
\quad
L_{1,2} =\frac{\mu_{1,1}^*\mu_{1,2}^* (\mu_{1,1}^* +\mu_{1,2}^*)}{\alpha}
(\theta_1^*, \eta_1^*, \theta_2^*, (\sigma^*)^2),
\end{equation*}
\begin{equation*}
L_{2,2} = ((\mu_{1,1}^*)^2 +(\mu_{1,2}^*)^2)
\begin{pmatrix}
(\theta_1^*)^2 & \theta_1^* \eta_1^* & \theta_1^* \theta_2^* & \theta_1^* (\sigma^*)^2
\\
\theta_1^* \eta_1^* & (\eta_1^*)^2 & \eta_1^* \theta_2^* & \eta_1^* (\sigma^*)^2
\\
\theta_1^* \theta_2^* & \eta_1^* \theta_2^* & (\theta_2^*)^2 & \theta_2^* (\sigma^*)^2
\\
\theta_1^* (\sigma^*)^2 & \eta_1^* (\sigma^*)^2 & \theta_2^* (\sigma^*)^2 & (\sigma^*)^4
\\
\end{pmatrix}.
\end{equation*}

\begin{thm}\label{th4}
Let $\alpha\in(0,2)$ and $\delta/\sqrt{\Delta} \equiv r \in(0,\infty)$.
Under [A1], the followings hold.

\begin{itemize}
\item[(a)] 
Suppose that $\mu_0$ is known.
If $\frac{n^{2-\alpha+\tau}}{(M_1 \land M_2)^{2\tau}}\to0$
for some $0 < \tau \le 1$ such that $\tau <\alpha$, then
as $M_1 \to \infty$, $M_2 \to \infty$ and $n \to \infty$,
\begin{equation*}
\sqrt{n}
\begin{pmatrix}
\bar \theta_1 - \theta_1^*
\\
\bar \eta_1 - \eta_1^*
\\
\bar \theta_2 - \theta_2^*
\\
\bar \sigma^2 - (\sigma^*)^2
\end{pmatrix}
\dto N (0, \mathcal K).
\end{equation*}

\item[(b)]
Suppose that $\mu_0$ is unknown.

\begin{itemize}
\item[(i)]
If $\frac{n^{1-\alpha+\tau}}{(M_1 \land M_2)^{2\tau}}\to0$
for some $0 < \tau \le 1$ such that $\tau <\alpha$, then
$\breve \mu_0 \pto \mu_0^*$ as $M_1 \to \infty$, $M_2 \to \infty$ and $n \to \infty$.

\item[(ii)]
If $\frac{n^{2-\alpha+\tau}}{(M_1 \land M_2)^{2\tau}}\to0$
for some $0 < \tau \le 1$ such that $\tau <\alpha$, then
as $M_1 \to \infty$, $M_2 \to \infty$ and $n \to \infty$,
\begin{equation*}
\sqrt{n}
\begin{pmatrix}
\breve \mu_0 - \mu_0^*
\\
\breve \theta_1 - \theta_1^*
\\
\breve \eta_1 - \eta_1^*
\\
\breve \theta_2 - \theta_2^*
\\
\breve \sigma^2 - (\sigma^*)^2
\end{pmatrix}
\dto N(0,\mathcal L).
\end{equation*}

\end{itemize}
\end{itemize}
\end{thm}

\begin{rmk}
\ 
\begin{enumerate}
\item[(i)]
Comparing the asymptotic variances $\mathcal K$ and $\mathcal L$ 
with those in Theorem 3.6 of Tonaki et al.\,\cite{TKU2022arXiv1}, 
we can see that the estimators of the coefficient parameters 
$(\theta_1, \eta_1, \theta_2, \sigma^2)$ 
are improved for both cases where $\mu_0$ is known and unknown.

\item[(ii)]
Theorem \ref{th4} shows that 
the estimators of the coefficient parameters when $\mu_0$ is unknown 
cost $\frac{2}{9}(\frac{\mu_0}{\pi^2} +2)(\frac{\mu_0}{\pi^2} +5)$
times as much variance as those when $\mu_0$ is known.

\end{enumerate}

\end{rmk}

\section{Simulations}\label{sec3}
The numerical solution of SPDE \eqref{2d_spde} 
is generated by
\begin{eqnarray}
\tilde{X}_{t_{i}}(y_j, z_k)
= \sum^{K}_{\ell_1 = 1}\sum^{L}_{\ell_2 = 1}x_{\ell_1,\ell_2}(t_{i})
e_{\ell_1}^{(1)}(y_j) e_{\ell_2}^{(2)}(z_k), 
\quad i = 1, ..., N, j = 1, ..., M_1, k = 1, ..., M_2.
\label{appro-SPDE}
\end{eqnarray}
In this simulation, the true values of parameters
$(\theta_0^*, \theta_1^*,\eta_1^*, \theta_2^*, \sigma^*) = (0,0.2,0.2,0.2,1)$.
We set that $N = 10^3$, $M_1 = M_2 = 200$, $K = L = 10^4$,  
$\xi = 0$,
$\alpha = 0.5$,
$\lambda_{1,1}^* \approx 4.05$.
When $N=10^3, M_1=M_2=200$, the size of data is about 10 GB. 
We used R language to compute the estimators of Theorems \ref{th1} and \ref{th2}.
The computation time of \eqref{appro-SPDE} is directly proportional to $K \times L$.
Therefore, the computation time for the numerical solution of 
SPDE \eqref{2d_spde} 
is directly proportional to $N \times M_1 \times  M_2 \times K \times L$.
In the setting of this simulation, 
$N \times M_1 \times  M_2 \times K \times L = 4 \times 10^{15}$.
Three personal computers with Intel Gold 6128 (3.4 GHz) 
were used for this simulation,
and it takes about 100h to generate one sample path of SPDE \eqref{2d_spde}.
The number of iterations is $250$.

First, we estimated $\theta_1$, $\eta_1$, $\theta_2$ and 
$\sigma^2$.
Table \ref{tab1} is the simulation results of 
the means and the standard deviations (s.d.s) of $\hat{\theta}_1$, 
$\hat{\eta}_1$, $\hat{\theta}_2$  and $\hat{\sigma}^2$, where $\hat{\theta}_1 = \hat{\kappa}\hat{\theta_2}$ and $\hat{\eta}_1 = \hat{\eta}\hat{\theta}_2$.
We set that  $(N, m_1,m_2,\alpha) = (10^3, 30,30,0.5)$, $r \approx 0.98382$ and $\underline{\theta_2} \approx 0.07267 $.
\begin{table}[h]
\caption{Simulation results of $\hat{\theta}_1$, 
$\hat{\eta}_1$, $\hat{\theta}_2$  and $\hat{\sigma}^2$ \label{tab1}}
\begin{center}
\begin{tabular}{c|cccc} \hline
		&$\hat{\theta}_1$&$\hat{\eta}_1$&$\hat{\theta}_2$&$\hat{\sigma}^2$
\\ \hline
true value & 0.2& 0.2 & 0.2 & 1
\\ \hline
mean & 0.197 & 0.197 & 0.197 & 1.000
 \\
s.d. & (0.0027) & (0.0029) & (0.0023) & (0.0161)
 \\   \hline
\end{tabular}
\end{center}
\end{table}
It seems from Table \ref{tab1} that the biases of $\hat{\theta}_1$, 
$\hat{\eta}_1$, $\hat{\theta}_2$  and $\hat{\sigma}^2$ are very small.
The estimators of $\theta_1$, $\eta_1$, $\theta_2$ and $\sigma^2$ have good performances.

Next, we estimated $\theta_0$.
Table \ref{tab2} is the simulation results of 
the mean and the standard deviation (s.d.) of $\check{\theta}_0$
with $(N, m_1,m_2,n,\alpha,\tau)$ 
$ = (10^3, 30,30, 50, 0.5. 0.49)$.
In this case, $\frac{n^{1-\alpha+\tau}}{(M_1 \land M_2)^{2\tau}} \approx 0.267$.
\begin{table}[h]
\caption{Simulation results of $\check{\theta}_0$ \label{tab2}}
\begin{center}
\begin{tabular}{c|c} \hline
		&$\check{\theta}_{0}$
\\ \hline
true value &0  
\\ \hline
mean & -0.359
 \\
s.d. & (1.371)
 \\   \hline
\end{tabular}
\end{center}
\end{table}
We see from Table \ref{tab2} that 
the biases of $\check{\theta}_0$ is a little bit small.

Table \ref{tab3} is the simulation results of 
the means and the standard deviations (s.d.s) of $\check{\theta}_0$
with $(N, m_1,m_2,n,\alpha,\tau)$ 
$ = (10^3, 30,30, 100, 0.5. 0.49)$.
In this case, $\frac{n^{1-\alpha+\tau}}{(M_1 \land M_2)^{2\tau}} \approx 0.531$.
\begin{table}[h]
\caption{Simulation results of $\check{\theta}_0$ \label{tab3}}
\begin{center}
\begin{tabular}{c|c} \hline
		&$\check{\theta}_{0}$
\\ \hline
true value &0
\\ \hline
mean & -0.377
 \\
s.d. & (0.946)
 \\   \hline
\end{tabular}
\end{center}
\end{table}
Table \ref{tab3} shows that $\hat{\theta}_0$ has a small bias.  

Table \ref{tab4} is the simulation results of 
the means and the standard deviations (s.d.s) of $\check{\theta}_0$
with $(N, m_1,m_2,n,\alpha,\tau)$ 
$ = (10^3, 30,30, 150, 0.5.0. 49)$.
In this case, $\frac{n^{1-\alpha+\tau}}{(M_1 \land M_2)^{2\tau}} \approx 0.793$.
\begin{table}[h]
\caption{Simulation results of $\check{\theta}_0$ \label{tab4}}
\begin{center}
\begin{tabular}{c|c} \hline
		&$\check{\theta}_{0}$
\\ \hline
true value &0  
\\ \hline
mean & -0.918
 \\
s.d. & (0.828)
 \\   \hline
\end{tabular}
\end{center}
\end{table}
It seems from Table \ref{tab4} that the bias of $\check{\theta}_0$ is small. 

We calculated  $\hat{\theta}^{S} = (\hat{\theta}^{S}_0,\hat{\theta}^{S}_1,\hat{\eta}^{S}_1,\hat{\theta}^{S}_2,(\hat{\sigma}^{S})^2)$ using Theorem 3.3 
in Tonaki et al.\,\cite{TKU2022arXiv1}.
Table \ref{tab6} is the simulation results of 
the means and the standard deviations (s.d.s) of $\hat{\theta}^{S}_0$, 
$\hat{\theta}^{S}_1$, $\hat{\eta}^{S}_1$, $\hat{\theta}^{S}_2$ and $(\hat{\sigma}^{S})^2$
with $(N, m_1,m_2,n,\alpha, 
\epsilon)$ 
$ = (10^3, 30,30,50,0.5,$  $ 
0.499)$.
In this case, 
$\frac{n^{1-\alpha+\epsilon}}{M_1^{2\epsilon} \land M_2^{2\epsilon}} \approx 0.25$.
\begin{table}[h]
\caption{Simulation results of $\hat{\theta}^{S}_0$, $\hat{\theta}^{S}_1$, $\hat{\eta}^{S}_1$, 
$\hat{\theta}^{S}_2$ and $(\hat{\sigma}^{S})^2$ \label{tab6}}
\begin{center}
\begin{tabular}{c|ccccc} \hline
		&$\hat{\theta}^{S}_{0}$&$\hat{\theta}^{S}_{1}$&$\hat{\eta}^{S}_{1}$&$\hat{\theta}^{S}_{2}$&$(\hat{\sigma}^{S})^2$
\\ \hline
true value &0 & 0.2 & 0.2 &0.2 &1 
\\ \hline
mean &-1.382&  0.161& 0.161 & 0.175& 0.758
 \\
s.d. & (5.576)& (0.111)& (0.111) &(0.121) & (0.524)
 \\   \hline
\end{tabular}
\end{center}
\end{table}
It seems from Table \ref{tab6} that 
the biases of $\hat{\theta}^{S}_1$, $\hat{\eta}^{S}_1$, 
$\hat{\theta}^{S}_2$ and $(\hat{\sigma}^{S})^2$ 
are not too large, 
but
$\hat{\theta}^{S}_0$ has a bias.

Table \ref{tab7} is the simulation results of 
the means and the standard deviations (s.d.s) of $\hat{\theta}^{S}_0$, 
$\hat{\theta}^{S}_1$, $\hat{\eta}^{S}_1$, $\hat{\theta}^{S}_2$ and $(\hat{\sigma}^{S})^2$
with $(N, m_1,m_2,n,\alpha, 
\epsilon) 
= (10^3, 30,30,100,0.5, 
0.499)$.
In this case, 
$\frac{n^{1-\alpha+\epsilon}}{M_1^{2\epsilon} \land M_2^{2\epsilon}} \approx 0.50$.
\begin{table}[h]
\caption{Simulation results of $\hat{\theta}^{S}_0$, $\hat{\theta}^{S}_1$, $\hat{\eta}^{S}_1$, 
$\hat{\theta}^{S}_2$ and $(\hat{\sigma}^{S})^2$ \label{tab7}}
\begin{center}
\begin{tabular}{c|ccccc} \hline
		&$\hat{\theta}^{S}_{0}$&$\hat{\theta}^{S}_{1}$&$\hat{\eta}^{S}_{1}$&$\hat{\theta}^{S}_{2}$&$(\hat{\sigma}^{S})^2$
\\ \hline
true value &0 & 0.2 & 0.2 &0.2 &1 
\\ \hline
mean &-1.245& 0.187& 0.187 & 0.204& 0.882
 \\
s.d. & (4.594)& (0.092)& (0.091) &(0.100) & (0.432)
 \\   \hline
\end{tabular}
\end{center}
\end{table}
Table \ref{tab7} shows that 
 $\hat{\theta}^{S}_0$  has a bias
 and $\hat{\theta}^{S}_1$, $\hat{\eta}^{S}_1$ and $(\hat{\sigma}^{S})^2$
have small biases
 , but
$\hat{\theta}^{S}_2$ has  a good performance.

Table \ref{tab8} is the simulation results of 
the means and the standard deviations (s.d.s) of $\hat{\theta}^{S}_0$, 
$\hat{\theta}^{S}_1$, $\hat{\eta}^{S}_1$, $\hat{\theta}^{S}_2$ and $(\hat{\sigma}^{S})^2$
with $(N, m_1,m_2,n,\alpha, 
\epsilon) 
= (10^3, 30,30,150,0.5, 
0.499)$.
In this case, 
$\frac{n^{1-\alpha+\epsilon}}{M_1^{2\epsilon} \land M_2^{2\epsilon}} \approx 0.75$.
\begin{table}[h]
\caption{Simulation results of $\hat{\theta}^{S}_0$, $\hat{\theta}^{S}_1$, $\hat{\eta}^{S}_1$, 
$\hat{\theta}^{S}_2$ and $(\hat{\sigma}^{S})^2$ \label{tab8}}
\begin{center}
\begin{tabular}{c|ccccc} \hline
		&$\hat{\theta}^{S}_{0}$&$\hat{\theta}^{S}_{1}$&$\hat{\eta}^{S}_{1}$&$\hat{\theta}^{S}_{2}$&$(\hat{\sigma}^{S})^2$
\\ \hline
true value &0 & 0.2 & 0.2 &0.2 &1 
\\ \hline
mean &-0.967& 0.163& 0.163 & 0.177& 0.767
 \\
s.d. & (3.333)& (0.070)& (0.069) &(0.076) & (0.329)
 \\   \hline
\end{tabular}
\end{center}
\end{table}
By Table \ref{tab8}, it seems that
the biases of $\hat{\theta}^{S}_1$, $\hat{\eta}^{S}_1$ 
and $\hat{\theta}^{S}_2$ 
are not too large
and that $\hat{\theta}^{S}_0$ and $(\hat{\sigma}^{S})^2$ have biases.

From Tables \ref{tab6}-\ref{tab8}, we can see that 
$\hat{\theta}^{S}_1$, $\hat{\eta}^{S}_1$, $\hat{\theta}^{S}_2$ and $(\hat{\sigma}^{S})^2$ 
have good performances when $n=100$
and $\hat{\theta}^{S}_0$ has a smallest bias when $n=150$.

Table \ref{tab9} is summary of Tables \ref{tab1}-\ref{tab4} and \ref{tab6}-\ref{tab8}.
Rows 2, 3, and 4 of Table \ref{tab9} show 
the results of $\hat{\theta} = (\check{\theta}_0, \hat{\theta}_1, \hat{\eta}_1, \hat{\theta}_2 , \hat{\sigma}^2)$ with 
$(N, m_1,m_2,n,\alpha,\tau) =$ 
$(10^3, 30,30, 50, 0.5. 0.49)$, $(10^3, 30,30, 100, 0.5. 0.49)$ and $(10^3, 30,30, 150, 0.5. 0.49)$, respectively. 
Rows 6, 7, and 8 of Table \ref{tab9} show 
the results of $\hat{\theta}^{S} = (\hat{\theta}^{S}_0, \hat{\theta}^{S}_1, \hat{\eta}^{S}_1, \hat{\theta}^{S}_2 , (\hat{\sigma}^{S})^2)$ with 
$(N, m_1,m_2,n,\alpha, \epsilon) = $
$(10^3, 30,30, 50, 0.5. 0.49)$, $(10^3, 30,30, 100, 0.5. 0.49)$ and $(10^3, 30,30, 150, 0.5. 0.49)$, respectively. 
Table \ref{tab9} shows that 
for all cases n=50, 100, and 150, the biases of $\check{\theta}_0$, $\hat{\theta}_1$, $\hat{\eta}_1$, $\hat{\theta}_2$ and $\hat{\sigma}^2$ are smaller than those of $\hat{\theta}^{S}_0$, $\hat{\theta}^{S}_1$, $\hat{\eta}^{S}_1$, $\hat{\theta}^{S}_2$ and $(\hat{\sigma}^{S})^2$, respectively.
This indicates that $\hat{\theta} = (\check{\theta}_0, \hat{\theta}_1, \hat{\eta}_1, \hat{\theta}_2 , \hat{\sigma}^2)$ is superior to $\hat{\theta}^{S} = (\hat{\theta}^{S}_0, \hat{\theta}^{S}_1, \hat{\eta}^{S}_1, \hat{\theta}^{S}_2 , (\hat{\sigma}^{S})^2)$.
\begin{table}[h]
\caption{Summary of Tables \ref{tab1}-\ref{tab4} and \ref{tab6}-\ref{tab8} \label{tab9}}
\begin{center}
\begin{tabular}{l|c|c|c|cccc} \hline
& & 		&$\check{\theta}_{0}$&$\hat{\theta}_{1}$&$\hat{\eta}_{1}$&$\hat{\theta}_{2}$&$\hat{\sigma}^2$
\\ 
& & true value &0 & 0.2 & 0.2 &0.2 &1 
\\ \cline{2-8}
Simulation result of & $n=50$ & mean & -0.359 &  &  &  & 
 \\ 
$\hat{\theta} = (\check{\theta}_0, \hat{\theta}_1, \hat{\eta}_1, \hat{\theta}_2 , \hat{\sigma}^2)$& & s.d. & (1.371)  &  &  &  & 
 \\   \cline{2-4}
& $n=100$ & mean & -0.377  & 0.197 & 0.197 & 0.197 & 1.000
 \\
& & s.d. & (0.946)  & (0.0027) & (0.0029) & (0.0023) & (0.0161)
 \\   \cline{2-4}
& $n=150$ & mean & -0.918 &  &  &  & 
 \\
& & s.d. & (0.828) &  &  &  & 
 \\   \hline
& & 		&$\hat{\theta}^{S}_{0}$&$\hat{\theta}^{S}_{1}$&$\hat{\eta}^{S}_{1}$&$\hat{\theta}^{S}_{2}$&$(\hat{\sigma}^{S})^2$
\\ 
& & true value &0 & 0.2 & 0.2 &0.2 &1 
\\ \cline{2-8}
Simulation result of & $n=50$ & mean &-1.382&  0.161& 0.161 & 0.175& 0.758
 \\
$\hat{\theta}^{S} = (\hat{\theta}^{S}_0, \hat{\theta}^{S}_1, \hat{\eta}^{S}_1, \hat{\theta}^{S}_2 , (\hat{\sigma}^{S})^2)$& & s.d. & (5.576)& (0.111)& (0.111) &(0.121) & (0.524)
 \\   \cline{2-8}
& $n=100$ & mean &-1.245&  0.187& 0.187 & 0.204& 0.882
 \\
& & s.d. & (4.594)& (0.092)& (0.091) &(0.100) & (0.432)
 \\   \cline{2-8}
& $n=150$ & mean &-0.967& 0.163& 0.163 & 0.177& 0.767
 \\
& & s.d. & (3.333)& (0.070)& (0.069) &(0.076) & (0.329)
 \\   \hline
\end{tabular}
\end{center}
\end{table}

\section{Proofs}\label{sec4}
We set the following notation.
\begin{enumerate}
\item[1.]
Let $\mathbb R_{+}=(0,\infty)$.

\item[2.]
For $a,b \in \mathbb R$, 
we write $a \lesssim b$ if $|a| \le C|b|$ for some constant $C>0$.

\item[3.]
For two functions $f, g : \mathbb R^d \to \mathbb R$,
we write $f(x) \lesssim g(x)$ $(x \to a)$ 
if $f(x) \lesssim g(x)$ in a neighborhood of $x = a$. 

\item[4.]
For $x=(x_1,\ldots, x_d) \in \mathbb R^d$ and $f:\mathbb R^d \to \mathbb R$,
we write $\pd_{x_i} f(x) = \frac{\pd}{\pd x_i}f(x)$,
$\pd_x f(x) = (\pd_{x_1}f(x), \ldots, \pd_{x_d}f(x))$ and 
$\pd_x^2 f(x) = (\pd_{x_j}\pd_{x_i}f(x))_{i,j=1}^d$.

\item[5.]
Let $\ind_A$ be the indicator function of $A$.

\item[6.]
For a function $f:\mathbb R \to \mathbb R$ and a positive number $\epsilon$,
we set
\begin{equation*}
D_\epsilon f(x) = f(x+\epsilon) - f(x).
\end{equation*}
Note that
$D_\epsilon^2 f(x) =D_\epsilon[D_\epsilon f](x) 
=  f(x+2\epsilon) - 2f(x+\epsilon) +f(x)$.

\item[7.]
For a function $f:\mathbb R^2 \to \mathbb R$ and a positive number $\epsilon$,
we set 
\begin{equation*}
D_{1,\epsilon} f(x,y) = f(x+\epsilon,y) - f(x,y),
\quad 
D_{2,\epsilon} f(x) = f(x,y+\epsilon) - f(x,y).
\end{equation*}

\end{enumerate}

\subsection{Proofs of the results in Subsections 
\ref{sec2-1} and \ref{sec2-2}}
Let 
\begin{align*}
A_{i,\ell_1,\ell_2} &= 
-\langle \xi, e_{\ell_1,\ell_2} \rangle (1-\ee^{-\lambda_{\ell_1,\ell_2}\Delta})
\ee^{-(i-1)\lambda_{\ell_1,\ell_2}\Delta},
\\
B_{i,\ell_1,\ell_2}^{(1)} &= 
-\frac{\sigma(1-\ee^{-\lambda_{\ell_1,\ell_2}\Delta})}
{\lambda_{\ell_1,\ell_2}^{\alpha/2}}
\int_0^{(i-1)\Delta} \ee^{-\lambda_{\ell_1,\ell_2}((i-1)\Delta-s)}
\dd w_{\ell_1,\ell_2}(s),
\\
B_{i,\ell_1,\ell_2}^{(2)} &= 
\frac{\sigma}{\lambda_{\ell_1,\ell_2}^{\alpha/2}}
\int_{(i-1)\Delta}^{i\Delta} \ee^{-\lambda_{\ell_1,\ell_2}(i\Delta-s)}
\dd w_{\ell_1,\ell_2}(s),
\end{align*}
and $B_{i,\ell_1,\ell_2}=B_{i,\ell_1,\ell_2}^{(1)}+B_{i,\ell_1,\ell_2}^{(2)}$.
The increment of the coordinate process of SPDE \eqref{2d_spde} driven by 
the $Q_1$-Wiener process is decomposed as
\begin{equation*}
x_{\ell_1,\ell_2}(t_{i})-x_{\ell_1,\ell_2}(t_{i-1})=
A_{i,\ell_1,\ell_2}+B_{i,\ell_1,\ell_2}.
\end{equation*}
\begin{proof}[\bf Proof of Proposition \ref{prop1}]
For $y,z \in [0,2)$, define
\begin{equation}\label{eq-1-1}
F_\alpha(y,z)=
\sum_{\ell_1,\ell_2\ge1}
\frac{1-\ee^{-\lambda_{\ell_1,\ell_2}\Delta}}{\lambda_{\ell_1,\ell_2}^{1+\alpha}}
\cos(\pi \ell_1 y)\cos(\pi \ell_2 z).
\end{equation}
Let
\begin{equation}\label{eq-1-2}
F_{j,k}=
\sum_{\ell_{1},\ell_{2}\ge1}
\frac{1-\ee^{-\lambda_{\ell_1,\ell_2}\Delta}}{\lambda_{\ell_1,\ell_2}^{1+\alpha}}
(e_{\ell_{1}}^{(1)}(\widetilde y_{j})-e_{\ell_{1}}^{(1)}(\widetilde y_{j-1}))^2
(e_{\ell_{2}}^{(2)}(\widetilde z_{k})-e_{\ell_{2}}^{(2)}(\widetilde z_{k-1}))^2.
\end{equation}

We first show that
\begin{equation}\label{eq-1-3}
\EE[(T_{i,j,k}X)^2] = \sigma^2 F_{j,k} + R_{i,j,k},
\end{equation}
where $R_{i,j,k}$ satisfies $\sum_{i=1}^N R_{i,j,k} \lesssim F_{j,k}$.

Notice that
\begin{equation*}
T_{i,j,k}X
= \sum_{\ell_1,\ell_2\ge1}
(x_{\ell_1,\ell_2}(t_{i})-x_{\ell_1,\ell_2}(t_{i-1}))
(e_{\ell_1}^{(1)}(\widetilde y_{j})-e_{\ell_1}^{(1)}(\widetilde y_{j-1}))
(e_{\ell_2}^{(2)}(\widetilde z_{k})-e_{\ell_2}^{(2)}(\widetilde z_{k-1})).
\end{equation*}
It holds from
the independence of $\{B_{i,\ell_1,\ell_2}\}_{\ell_1,\ell_2 \ge 1}$
and $\EE[B_{i,\ell_1,\ell_2}]=0$
that
\begin{align*}
\EE[(T_{i,j,k}X)^2]
&= \sum_{\ell_1,\ell_1',\ell_2,\ell_2'\ge1}
A_{i,\ell_1,\ell_2}A_{i,\ell_1',\ell_2'}
\\
&\qquad\qquad\times
(e_{\ell_1}^{(1)}(\widetilde y_{j})-e_{\ell_1}^{(1)}(\widetilde y_{j-1}))
(e_{\ell_2}^{(2)}(\widetilde z_{k})-e_{\ell_2}^{(2)}(\widetilde z_{k-1}))
\\
&\qquad\qquad\times
(e_{\ell_1'}^{(1)}(\widetilde y_{j})-e_{\ell_1'}^{(1)}(\widetilde y_{j-1}))
(e_{\ell_2'}^{(2)}(\widetilde z_{k})-e_{\ell_2'}^{(2)}(\widetilde z_{k-1}))
\\
&\qquad
+\sum_{\ell_1,\ell_2\ge1}
\EE[B_{i,\ell_1,\ell_2}^2]
(e_{\ell_1}^{(1)}(\widetilde y_{j})-e_{\ell_1}^{(1)}(\widetilde y_{j-1}))^2
(e_{\ell_2}^{(2)}(\widetilde z_{k})-e_{\ell_2}^{(2)}(\widetilde z_{k-1}))^2
\\
&= 
\Biggl(
\sum_{\ell_{1},\ell_{2}\ge1}
A_{i,\ell_{1},\ell_{2}}
(e_{\ell_{1}}^{(1)}(\widetilde y_{j})-e_{\ell_{1}}^{(1)}(\widetilde y_{j-1}))
(e_{\ell_{2}}^{(2)}(\widetilde z_{k})-e_{\ell_{2}}^{(2)}(\widetilde z_{k-1}))
\Biggr)^2
\\
&\qquad
+\sum_{\ell_1,\ell_2\ge1}
\EE[B_{i,\ell_1,\ell_2}^2]
(e_{\ell_1}^{(1)}(\widetilde y_{j})-e_{\ell_1}^{(1)}(\widetilde y_{j-1}))^2
(e_{\ell_2}^{(2)}(\widetilde z_{k})-e_{\ell_2}^{(2)}(\widetilde z_{k-1}))^2
\\
&=: S_1 + S_2.
\end{align*}
Since $\|A_{\theta}^{(1+\alpha)/2} \xi \|^2
=\sum_{\ell_1, \ell_2 \ge 1}\lambda_{\ell_1, \ell_2}^{1+\alpha}
\langle \xi, e_{\ell_1,\ell_2} \rangle^2$, 
it follows from the Schwarz inequality that
\begin{align*}
S_1
&=\Biggl(
\sum_{\ell_{1},\ell_{2}\ge1}
\lambda_{\ell_1,\ell_2}^{(1+\alpha)/2} 
\langle \xi, e_{\ell_1,\ell_2} \rangle
\\
&\qquad\qquad
\times
\frac{1-\ee^{-\lambda_{\ell_1,\ell_2}\Delta}}{\lambda_{\ell_1,\ell_2}^{(1+\alpha)/2}}
\ee^{-(i-1)\lambda_{\ell_1,\ell_2}\Delta}
(e_{\ell_{1}}^{(1)}(\widetilde y_{j})-e_{\ell_{1}}^{(1)}(\widetilde y_{j-1}))
(e_{\ell_{2}}^{(2)}(\widetilde z_{k})-e_{\ell_{2}}^{(2)}(\widetilde z_{k-1}))
\Biggr)^2
\\
&\le
\sum_{\ell_{1},\ell_{2}\ge1}
\lambda_{\ell_1,\ell_2}^{1+\alpha}
\langle \xi, e_{\ell_1,\ell_2} \rangle^2 
\\
&\qquad\times
\sum_{\ell_{1},\ell_{2}\ge1}
\frac{(1-\ee^{-\lambda_{\ell_1,\ell_2}\Delta})^2}{\lambda_{\ell_1,\ell_2}^{1+\alpha}}
\ee^{-2(i-1)\lambda_{\ell_1,\ell_2}\Delta}
(e_{\ell_{1}}^{(1)}(\widetilde y_{j})-e_{\ell_{1}}^{(1)}(\widetilde y_{j-1}))^2
(e_{\ell_{2}}^{(2)}(\widetilde z_{k})-e_{\ell_{2}}^{(2)}(\widetilde z_{k-1}))^2
\\
&=:
\|A_{\theta}^{(1+\alpha)/2}\xi\|^2 G_{i,j,k},
\end{align*}
where $G_{i,j,k}$ satisfies
\begin{equation}\label{eq-1-4}
\sum_{i=1}^N G_{i,j,k} 
\le
\sum_{\ell_{1},\ell_{2}\ge1}
\frac{1-\ee^{-\lambda_{\ell_1,\ell_2}\Delta}}{\lambda_{\ell_1,\ell_2}^{1+\alpha}}
(e_{\ell_{1}}^{(1)}(\widetilde y_{j})-e_{\ell_{1}}^{(1)}(\widetilde y_{j-1}))^2
(e_{\ell_{2}}^{(2)}(\widetilde z_{k})-e_{\ell_{2}}^{(2)}(\widetilde z_{k-1}))^2
=F_{j,k}.
\end{equation}
Meanwhile, 
one has
\begin{equation*}
\EE[B_{i,\ell_1,\ell_2}^2]
=\frac{\sigma^2(1-\ee^{-\lambda_{\ell_1,\ell_2}\Delta})}
{\lambda_{\ell_1,\ell_2}^{1+\alpha}}
\biggl(
1-\frac{1-\ee^{-\lambda_{\ell_1,\ell_2}\Delta}}{2}
\ee^{-2(i-1)\lambda_{\ell_1,\ell_2}\Delta}
\biggr),
\end{equation*}
and thus it follows that
\begin{equation*}
S_2 = \sigma^2 \biggl(F_{j,k}-\frac{G_{i,j,k}}{2}\biggr).
\end{equation*}
Consequently, from [A1] and \eqref{eq-1-4}, we obtain \eqref{eq-1-3}.

We next show that 
\begin{equation}\label{eq-1-5}
F_{j,k} = \Delta^\alpha
\ee^{-\kappa(\widetilde y_{j-1}+\widetilde y_j)/2}
\ee^{-\eta(\widetilde z_{k-1}+\widetilde z_k)/2} \psi_{r,\alpha}(\theta_2)
+O(\Delta^{1+\alpha}),
\end{equation}
where $\psi_{r,\alpha}(\theta_2)$ is given in \eqref{psi}.

Since it holds from Lemma \ref{lem9} that
\begin{equation*}
F_{j,k} = 
4\ee^{-\kappa(\widetilde y_{j-1}+\widetilde y_j)/2}
\ee^{-\eta(\widetilde z_{k-1}+\widetilde z_k)/2}
D_{1,\delta}D_{2,\delta}F_\alpha(0,0)
+O(\Delta^{1+\alpha}),
\end{equation*}
the calculation of $D_{1,\delta}D_{2,\delta}F_\alpha(0,0)$ remains.
Let 
\begin{equation*}
f_\alpha(s) = \frac{1-\ee^{-s}}{s^{1+\alpha}},
\quad
G_{r,\alpha}(x,y:\theta_2)=
f_\alpha(\theta_2\pi^2(x^2+y^2))(\cos(\pi r x)-1)(\cos(\pi r y)-1).
\end{equation*}
It follows from Lemma \ref{lem1} that for $\alpha \in (0,2)$ 
and $\delta = r \sqrt{\Delta}$,
\begin{align}
D_{1,\delta}D_{2,\delta}F_\alpha(0,0)
&=
\Delta^{1+\alpha}
\sum_{\ell_1,\ell_2\ge1}
f_\alpha(\lambda_{\ell_1,\ell_2}\Delta)
(\cos(\pi r\ell_1 \sqrt{\Delta})-1)(\cos(\pi r\ell_2 \sqrt{\Delta})-1)
\nonumber
\\
&=
\Delta^{\alpha}
\iint_{\mathbb R_{+}^2} G_{r,\alpha}(x,y:\theta_2) \dd x \dd y
+O(\Delta^{1+\alpha}).
\label{eq-1-6}
\end{align}
Since
\begin{equation*}
\int_0^{\pi/2}\cos(x\cos (t))\dd t = 
\int_0^{\pi/2}\cos(x\sin (t))\dd t = 
\frac{\pi}{2}J_0(x),
\end{equation*}
\begin{equation*}
\int_0^{\pi/2} \cos(x\cos (t))\cos(x\sin (t)) \dd t 
= \frac{\pi}{2}J_0(\sqrt{2} x),
\end{equation*}
the integral in \eqref{eq-1-6} can be computed as
\begin{align*}
&\iint_{\mathbb R_{+}^2} G_{r,\alpha}(x,y:\theta_2) \dd x \dd y
\\
&=\frac{1}{\theta_2\pi^2}
\int_0^\infty z f_\alpha (z^2)
\int_0^{\pi/2} 
\Bigl(\cos\Bigl(\frac{r z}{\sqrt{\theta_2}}\cos (t)\Bigr)-1\Bigr)
\Bigl(\cos\Bigl(\frac{r z}{\sqrt{\theta_2}}\sin (t)\Bigr)-1\Bigr) \dd t \dd z
\\
&=\frac{1}{2\theta_2\pi}
\int_0^\infty 
\frac{1-\ee^{-z^2}}{z^{1+2\alpha}}
\biggl(
J_0\Bigl(\frac{\sqrt{2}r z}{\sqrt{\theta_2}}\Bigr)
-2J_0\Bigl(\frac{r z}{\sqrt{\theta_2}}\Bigr)+1
\biggr) \dd z
\\
&=\frac{\psi_{r,\alpha}(\theta_2)}{4},
\end{align*}
and thus we get \eqref{eq-1-5}. This concludes the proof.
\end{proof}

For $b \in (0,1/2)$ and $u,v \in L^2((b,1-b)^2)$, we write
\begin{equation*}
\langle u,v \rangle_b = 
\frac{1}{(1-2b)^2} \int_b^{1-b}\int_b^{1-b} u(x,y) v(x,y)\dd x \dd y,
\quad
\| u \|_b = \sqrt{\langle u,u \rangle_b}.
\end{equation*}
We define $V(\nu) = 2(V^{(1)}(\nu)+V^{(2)}(\nu))$, where
\begin{equation*}
V^{(j)}(\nu) = 
\Bigl(
\bigl\langle 
\pd_{\nu_p} f_{r/\sqrt{j},\alpha}(\cdot,\cdot:\nu), 
\pd_{\nu_q} f_{r/\sqrt{j},\alpha}(\cdot,\cdot:\nu)
\bigr\rangle_b
\Bigr)_{1\le p,q \le 4}.
\end{equation*}
Set $\bar y_j = (\widetilde y_{j-1}+\widetilde y_j)/2$, 
$\bar z_k = (\widetilde z_{k-1}+\widetilde z_k)/2$.
The contrast function $K_{m,N}$ can be expressed as 
$K_{m,N}(\nu) = K_{m,N}^{(1)}(\nu) + K_{m,N}^{(2)}(\nu)$, where
\begin{align*}
K_{m,N}^{(1)}(\nu)
&=
\frac{1}{m}\sum_{k=1}^{m_2} \sum_{j=1}^{m_1}
\Biggl(
\frac{1}{N \Delta^\alpha}\sum_{i=1}^N (T_{i,j,k}X)^2 
- f_{r,\alpha}(\bar y_j, \bar z_k:\nu)
\Biggr)^2,
\\
K_{m,N}^{(2)}(\nu)
&=
\frac{1}{m}\sum_{k=1}^{m_2} \sum_{j=1}^{m_1}
\Biggl(
\frac{1}{N (2\Delta)^\alpha}\sum_{i=1}^N (\widetilde T_{i,j,k}X)^2 
- f_{r/\sqrt{2},\alpha}(\bar y_j, \bar z_k:\nu)
\Biggr)^2.
\end{align*}
Let $K(\nu, \nu^*) = K^{(1)}(\nu,\nu^*)+K^{(2)}(\nu,\nu^*)$, where
\begin{equation*}
K^{(p)}(\nu,\nu^*)
=\|f_{r/\sqrt{p},\alpha}(\cdot,\cdot:\nu)-f_{r/\sqrt{p},\alpha}(\cdot,\cdot:\nu^*)
\|_b^2.
\end{equation*}

\begin{proof}[\bf Proof of Theorem \ref{th1}]
It holds from the mean value theorem that
\begin{equation*}
-\sqrt{m N}\pd_\nu K_{m,N}(\nu^*)^\TT
=\int_0^1 \pd_\nu^2 K_{m,N}(\nu^*+u(\hat\nu-\nu^*))\dd u 
\sqrt{m N}(\hat\nu-\nu^*).
\end{equation*}
To complete the proof, it suffices to prove that
\begin{equation}\label{eq-2-1}
\hat\nu \pto \nu^*, 
\end{equation}
\begin{equation}\label{eq-2-2}
-\sqrt{MN}\pd_\nu K_{m,N}(\nu^*)^\TT = O_p(1),
\end{equation}
\begin{equation}\label{eq-2-3}
\int_0^1 \pd_\nu^2 K_{m,N}(\nu^*+u(\hat\nu-\nu^*))\dd u 
\pto V(\nu^*) \in \mathrm{GL}_4(\mathbb R).
\end{equation}

We will prove \eqref{eq-2-1}-\eqref{eq-2-3} by the following five steps.
We verify that
$K(\nu,\nu^*)$ takes its unique minimum in $\nu = \nu^*$ and 
$K_{m,N}(\nu)$ converges to $K(\nu,\nu^*)$ 
in probability uniformly with respect to $\nu \in \Xi$
in order to show \eqref{eq-2-1} in Steps 1 and 2, 
Step 3 is devoted to the proof of \eqref{eq-2-2}.
We show the convergence in probability of \eqref{eq-2-3} in Step 4
and confirm that $V(\nu^*)$ is non-singular in Step 5.

\textbf{Step 1:}
We show that $K(\nu,\nu^*)=0$ if and only if $\nu=\nu^*$.
If $K(\nu,\nu^*)=0$, then
\begin{equation*}
\|f_{r,\alpha}(\cdot,\cdot:\nu)-f_{r,\alpha}(\cdot,\cdot:\nu^*)\|_b
= \| f_{r/\sqrt{2},\alpha}(\cdot,\cdot:\nu)
-f_{r/\sqrt{2},\alpha}(\cdot,\cdot:\nu^*) \|_b
= 0,
\end{equation*}
that is, 
$f_{r,\alpha}(y,z:\nu) = f_{r,\alpha}(y,z:\nu^*)$
and $f_{r/\sqrt{2},\alpha}(y,z:\nu) = f_{r/\sqrt{2},\alpha}(y,z:\nu^*)$
for any $y,z\in[b,1-b]$.
Therefore, it holds that for any $y,z\in[b,1-b]$,
\begin{align*}
\sigma^2\psi_{r,\alpha}(\theta_2)\ee^{-\kappa y}\ee^{-\eta z} &=
(\sigma^*)^2\psi_{r,\alpha}(\theta_2^*)\ee^{-\kappa^* y}\ee^{-\eta^* z},
\\
\sigma^2\psi_{r/\sqrt{2},\alpha}(\theta_2)\ee^{-\kappa y}\ee^{-\eta z} &=
(\sigma^*)^2\psi_{r/\sqrt{2},\alpha}(\theta_2^*)\ee^{-\kappa^* y}\ee^{-\eta^* z},
\end{align*}
which yield $\kappa = \kappa^*$, $\eta = \eta^*$,
$\sigma^2\psi_{r,\alpha}(\theta_2) 
= (\sigma^*)^2\psi_{r,\alpha}(\theta_2^*)$
and $\sigma^2\psi_{r/\sqrt{2},\alpha}(\theta_2) 
= (\sigma^*)^2\psi_{r/\sqrt{2},\alpha}(\theta_2^*)$. 
Noting that 
$\psi_{r,\alpha}(\theta_2)/\psi_{r/\sqrt{2},\alpha}(\theta_2)
=\psi_{r,\alpha}(\theta_2^*)/\psi_{r/\sqrt{2},\alpha}(\theta_2^*)$
and the function
\begin{equation*}
\theta_2\mapsto \psi_{r,\alpha}(\theta_2)/\psi_{r/\sqrt{2},\alpha}(\theta_2)
\end{equation*}
is injective for $\theta_2 > \underline{\theta_2}(r,\alpha)$, 
which follows from Lemmas \ref{lem11} and \ref{lem12},
we obtain from $\psi_{r,\alpha}(\theta_2)>0$ that 
$\theta_2=\theta_2^*$ and $\sigma^2=(\sigma^*)^2$.
Moreover, if $\nu=\nu^*$, then $K(\nu,\nu^*)=0$. Therefore, we have the desired result.

\textbf{Step 2:}
We show that
\begin{equation*}
\sup_{\nu \in \Xi}|K_{m,N}(\nu)-K(\nu,\nu^*)| \pto 0, 
\quad
m,N\to\infty.
\end{equation*}
Setting
\begin{equation*}
\zeta_{i,j,k} =
\frac{(T_{i,j,k} X)^2}{\Delta^\alpha}-f_{r,\alpha}(\bar y_j, \bar z_k:\nu^*),
\quad
Z_{j,k} = \frac{1}{N}\sum_{i=1}^N \zeta_{i,j,k},
\end{equation*}
and $\phi_{r,\alpha}(y,z:\nu) = f_{r,\alpha}(y,z:\nu^*) - f_{r,\alpha}(y,z:\nu)$,
we can write
\begin{equation*}
K_{m,N}^{(1)}(\nu)
=\frac{1}{m}\sum_{k=1}^{m_2} \sum_{j=1}^{m_1} Z_{j,k}^2
+\frac{2}{m}\sum_{k=1}^{m_2} \sum_{j=1}^{m_1} Z_{j,k} 
\phi_{r,\alpha}(\bar y_j, \bar z_k:\nu)
+\frac{1}{m}\sum_{k=1}^{m_2} \sum_{j=1}^{m_1}
\phi_{r,\alpha}^2(\bar y_j, \bar z_k:\nu),
\end{equation*}
and thus have
\begin{align}
\sup_{\nu \in \Xi}|K_{m,N}^{(1)}(\nu)-K^{(1)}(\nu,\nu^*)|
&\le
\frac{1}{m}\sum_{k=1}^{m_2} \sum_{j=1}^{m_1} Z_{j,k}^2
+2\sup_{\nu \in \Xi}
\Biggl|
\frac{1}{m}\sum_{k=1}^{m_2} \sum_{j=1}^{m_1} Z_{j,k} 
\phi_{r,\alpha}(\bar y_j, \bar z_k:\nu)
\Biggr|
\nonumber
\\
&\qquad
+\sup_{\nu \in \Xi}
\Biggl|
\frac{1}{m}\sum_{k=1}^{m_2} \sum_{j=1}^{m_1}
\phi_{r,\alpha}^2(\bar y_j, \bar z_k:\nu)
-K^{(1)}(\nu,\nu^*)
\Biggr|.
\label{eq-2-4}
\end{align}
Since the function $(y,z,\nu) \mapsto \phi_{r,\alpha}(y,z:\nu)$ is
continuous on the compact set $[b,1-b]^2 \times \Xi$, 
the last term of \eqref{eq-2-4} converges to 0.
Furthermore, we obtain from Proposition \ref{prop1} that 
\begin{equation*}
\zeta_{i,j,k} =
\Delta^{-\alpha}
\bigl((T_{i,j,k} X)^2-\EE[(T_{i,j,k} X)^2] \bigr) + \mathcal R_{i,j,k},
\end{equation*}
where $\sum_{i=1}^N \mathcal R_{i,j,k}=O(1)$ uniformly in $j,k$,
and it follows from Isserlis's theorem and Lemma \ref{lem10} that 
\begin{align*}
\EE[\zeta_{i,j,k}\zeta_{i',j,k}]
&=\Delta^{-2\alpha}\cov[(T_{i,j,k} X)^2,(T_{i',j,k} X)^2]
+\mathcal R_{i,j,k} \mathcal R_{i',j,k}
\nonumber
\\
&=2\Delta^{-2\alpha}
\cov[T_{i,j,k} X, T_{i',j,k} X]^2
+\mathcal R_{i,j,k} \mathcal R_{i',j,k}
\nonumber
\\
&=
O\biggl(\frac{1}{(|i-i'|+1)^2}\biggr) 
+\mathcal R_{i,j,k} \mathcal R_{i',j,k}.
\end{align*}
Therefore, it holds from the Schwarz inequality and the boundedness of 
$\phi_{r,\alpha}$ that
\begin{equation*}
\EE[Z_{j,k}^2]
=\frac{1}{N^2}\sum_{i,i'=1}^N \EE[\zeta_{i,j,k}\zeta_{i',j,k}]
=O\biggl(\frac{1}{N}\biggr)
\text{ uniformly in } j,k, 
\end{equation*}
and
\begin{align*}
\EE\Biggl[
\sup_{\nu \in \Xi}
\biggl|\frac{1}{m}
\sum_{k=1}^{m_2} \sum_{j=1}^{m_1} 
Z_{j,k} \phi_{r,\alpha}(\bar y_j, \bar z_k:\nu)\biggr|^2
\Biggr]
&\lesssim
\frac{1}{m^2}
\sum_{k,k'=1}^{m_2} \sum_{j,j'=1}^{m_1}
\EE[Z_{j,k}^2]^{1/2} \EE[Z_{j',k'}^2]^{1/2}
\\
&=
\Biggl(
\frac{1}{m} \sum_{k=1}^{m_2} \sum_{j=1}^{m_1} \EE[Z_{j,k}^2]^{1/2}
\Biggr)^2
=O\biggl(\frac{1}{N}\biggr),
\end{align*}
which imply the first and second terms of \eqref{eq-2-4} converge to 0 in probability.
Noting that $\widetilde T_{i,j,k} X = T_{i,j,k} X +T_{i+1,j,k} X$, 
we can show that
\begin{equation*}
\sup_{\nu \in \Xi}|K_{m,N}^{(2)}(\nu)-K^{(2)}(\nu,\nu^*)| \pto 0,
\quad m,N \to \infty
\end{equation*}
in the same way, and therefore we obtain the desired result.

\textbf{Step 3:}
We prove \eqref{eq-2-2}.
It follows from Lemma \ref{lem10} and Proposition \ref{prop1} that
\begin{align*}
&\VV \Biggl[
\frac{1}{\sqrt{m N} \Delta^\alpha}
\sum_{k=1}^{m_2} \sum_{j=1}^{m_1} \sum_{i=1}^N 
(T_{i,j,k}X)^2 \pd_\nu f_{r,\alpha}(\bar y_j, \bar z_k:\nu^*)^\TT
\Biggr]
\\
&=
\frac{1}{m N \Delta^{2\alpha}}
\sum_{k,k'=1}^{m_2} \sum_{j,j'=1}^{m_1} \sum_{i,i'=1}^N 
\cov[(T_{i,j,k}X)^2,(T_{i',j',k'}X)^2] 
\pd_\nu f_{r,\alpha}(\bar y_j, \bar z_k:\nu^*)^\TT
\pd_\nu f_{r,\alpha}(\bar y_{j'}, \bar z_{k'}:\nu^*)
\\
&=
\frac{2}{m N \Delta^{2\alpha}}
\sum_{k,k'=1}^{m_2} \sum_{j,j'=1}^{m_1} \sum_{i,i'=1}^N 
\cov[T_{i,j,k}X, T_{i',j',k'}X]^2 
\pd_\nu f_{r,\alpha}(\bar y_j, \bar z_k:\nu^*)^\TT
\pd_\nu f_{r,\alpha}(\bar y_{j'}, \bar z_{k'}:\nu^*)
\\
&=O(1)
\end{align*}
and 
\begin{equation*}
\sum_{i=1}^N 
\Bigl( \EE[(T_{i,j,k}X)^2] 
- \Delta^\alpha f_{r,\alpha}(\bar y_j, \bar z_k:\nu^*) \Bigr) 
= O(\Delta^\alpha)
\end{equation*}
uniformly in $j,k$. Therefore, we obtain 
\begin{align*}
&\sqrt{m N}\pd_\nu K_{m,N}^{(1)}(\nu^*)^\TT
\\
&=
\sqrt{m N} \times \frac{-2}{m}\sum_{k=1}^{m_2} \sum_{j=1}^{m_1}
\Biggl(
\frac{1}{N \Delta^\alpha}\sum_{i=1}^N (T_{i,j,k}X)^2 
- f_{r,\alpha}(\bar y_j, \bar z_k:\nu^*)
\Biggr)
\pd_\nu f_{r,\alpha}(\bar y_j, \bar z_k:\nu^*)^\TT
\\
&=
\frac{-2}{\sqrt{m N} \Delta^\alpha}
\sum_{k=1}^{m_2} \sum_{j=1}^{m_1} 
\sum_{i=1}^N 
\Bigl( (T_{i,j,k}X)^2 - \EE[(T_{i,j,k}X)^2] \Bigr)
\pd_\nu f_{r,\alpha}(\bar y_j, \bar z_k:\nu^*)^\TT
\\
&\qquad
-\frac{2}{\sqrt{m N} \Delta^\alpha}
\sum_{k=1}^{m_2} \sum_{j=1}^{m_1} 
\sum_{i=1}^N 
\Bigl( \EE[(T_{i,j,k}X)^2] 
- \Delta^\alpha f_{r,\alpha}(\bar y_j, \bar z_k:\nu^*) \Bigr)
\pd_\nu f_{r,\alpha}(\bar y_j, \bar z_k:\nu^*)^\TT
\\
&=O_p(1).
\end{align*}
Similarly, it can be shown that 
$\sqrt{m N}\pd_\nu K_{m,N}^{(2)}(\nu^*)^\TT$ is uniformly tight,
and thus the desired result can be obtained.

\textbf{Step 4:}
We show that for $\epsilon_{m,N} \to 0$,  
\begin{equation}\label{eq-2-5}
\sup_{|\nu-\nu^*|\le \epsilon_{m,N}}
|\pd_\nu^2 K_{m,N}(\nu) - V(\nu^*)| \pto 0,
\quad m,N \to \infty.
\end{equation}
Note that 
\begin{align*}
\pd_\nu^2 K_{m,N}^{(1)}(\nu) &=
\frac{2}{m}\sum_{k=1}^{m_2}\sum_{j=1}^{m_1}
\pd_\nu f_{r,\alpha}(\bar y_j, \bar z_k:\nu)^\TT
\pd_\nu f_{r,\alpha}(\bar y_j, \bar z_k:\nu)
\nonumber
\\
&\qquad
-\frac{2}{m}\sum_{k=1}^{m_2}\sum_{j=1}^{m_1}\Biggl(
\frac{1}{N \Delta^\alpha}\sum_{i=1}^N (T_{i,j,k}X)^2 
-f_{r,\alpha}(\bar y_j, \bar z_k:\nu)
\Biggr)\pd_\nu^2 f_{r,\alpha}(\bar y_j, \bar z_k:\nu)
\nonumber
\\
&=:
\frac{1}{m}\sum_{k=1}^{m_2}\sum_{j=1}^{m_1}
\biggl\{
g_1(\bar y_j, \bar z_k:\nu)
+\biggl(
\frac{1}{N \Delta^\alpha}\sum_{i=1}^N (T_{i,j,k}X)^2 
\biggr)g_2(\bar y_j, \bar z_k:\nu)
\biggr\}.
\end{align*}
Since the function $(y,z,\nu) \mapsto g_\ell(y,z:\nu)$ is continuous 
on the compact set $[b,1-b]^2 \times \Xi$ and 
\begin{equation*}
\sup_{m,N}
\EE\Biggl[
\frac{1}{m N\Delta^\alpha}
\sum_{k=1}^{m_2}\sum_{j=1}^{m_1}\sum_{i=1}^N (T_{i,j,k}X)^2
\Biggr]<\infty,
\end{equation*}
it follows that as $m,N \to \infty$,
\begin{align*}
&\sup_{|\nu-\nu^*|\le \epsilon_{m,N}}
|\pd_\nu^2 K_{m,N}^{(1)}(\nu)-\pd_\nu^2 K_{m,N}^{(1)}(\nu^*)|
\\
&\le
\sup_{\substack{y,z \in [b,1-b],\\ |\nu-\nu^*|\le \epsilon_{m,N}}}
|g_1(y,z:\nu) - g_1(y,z:\nu^*)|
\\
&\qquad
+
\Biggl(
\frac{1}{m N \Delta^\alpha}
\sum_{k=1}^{m_2}\sum_{j=1}^{m_1}\sum_{i=1}^N (T_{i,j,k}X)^2
\Biggr)
\sup_{\substack{y,z \in [b,1-b],\\ |\nu-\nu^*|\le \epsilon_{m,N}}}
|g_2(y,z:\nu) - g_2(y,z:\nu^*)|
\\
&\pto 0.
\end{align*}
By the same reasoning, it holds that
$\sup_{|\nu-\nu^*|\le \epsilon_{m,N}}
|\pd_\nu^2 K_{m,N}^{(2)}(\nu)-\pd_\nu^2 K_{m,N}^{(2)}(\nu^*)| \pto 0$
as $m,N \to \infty$, and thus
\begin{equation}\label{eq-2-6}
\sup_{|\nu-\nu^*|\le \epsilon_{m,N}}
|\pd_\nu^2 K_{m,N}(\nu)-\pd_\nu^2 K_{m,N}(\nu^*)| \pto 0.
\end{equation}
Meanwhile, 
in the same way as Step 2, it holds that as $m, N \to \infty$,
\begin{equation*}
\frac{1}{m}
\sum_{k=1}^{m_2}\sum_{j=1}^{m_1}
\pd_\nu f_{r,\alpha}(\bar y_j, \bar z_k:\nu^*)^\TT
\pd_\nu f_{r,\alpha}(\bar y_j, \bar z_k:\nu^*)
\to V^{(1)}(\nu^*),
\end{equation*}
\begin{align*}
\frac{1}{m}\sum_{k=1}^{m_2}\sum_{j=1}^{m_1}\Biggl(
\frac{1}{N \Delta^\alpha}\sum_{i=1}^N (T_{i,j,k}X)^2 
-f_{r,\alpha}(\bar y_j, \bar z_k:\nu^*)
\Biggr)\pd_\nu^2 f_{r,\alpha}(\bar y_j, \bar z_k:\nu^*)
\pto 0,
\end{align*}
and hence we have  
\begin{equation*}
\pd_\nu^2 K_{m,N}^{(1)}(\nu^*) \pto 2V^{(1)}(\nu^*),
\quad
m, N \to \infty.
\end{equation*}
Since 
it can be similarly shown that
$\pd_\nu^2 K_{m,N}^{(2)}(\nu^*)$ converges to $2V^{(2)}(\nu^*)$ in probability,
we obtain
\begin{equation*}
\pd_\nu^2 K_{m,N}(\nu^*) \pto V(\nu^*),
\quad 
m, N \to \infty,
\end{equation*}
which together with \eqref{eq-2-6} yields \eqref{eq-2-5}.

\textbf{Step 5:} 
We show that $V=V(\nu^*)$ is strictly positive definite.
Let 
$g_p = \pd_{\nu_p} f_{r,\alpha}(\cdot,\cdot:\nu^*)$ and
$h_p = \pd_{\nu_p} f_{r/\sqrt{2},\alpha}(\cdot,\cdot:\nu^*)$
for $p=1,\ldots,4$.
It follows that for any $u = (u_p)_{p=1}^4 \in \mathbb R^4 \setminus \{0\}$, 
\begin{equation*}
u^\TT V u 
=\sum_{p,q=1}^4 u_p u_q \langle g_p, g_q \rangle_b
+\sum_{p,q=1}^4 u_p u_q \langle h_p, h_q \rangle_b
=
\Biggl\| \sum_{p=1}^4 u_p g_p \Biggr\|_b^2
+\Biggl\| \sum_{p=1}^4 u_p h_p \Biggr\|_b^2
\ge0.
\end{equation*}
Now, we assume that there exists $u \in \mathbb R^4\setminus\{0\}$ such that
\begin{equation}\label{eq-2-7}
\Biggl\| \sum_{p=1}^4 u_p g_p \Biggr\|_b^2
+\Biggl\| \sum_{p=1}^4 u_p h_p \Biggr\|_b^2
=0.
\end{equation}
Since
\begin{equation*}
\sum_{p=1}^4 u_p g_p(y,z) = 0,
\quad
\sum_{p=1}^4 u_p h_p(y,z) = 0,
\quad (y,z) \in [b,1-b]^2  
\end{equation*}
and $g=(g_p)_{p=1}^4$ is given by 
\begin{equation*}
g(y,z) = \ee^{-\kappa y} \ee^{-\eta z} 
(-y \sigma^2 \psi_{r,\alpha}(\theta_2), -z \sigma^2 \psi_{r,\alpha}(\theta_2), 
\sigma^2 \psi_{r,\alpha}'(\theta_2), \psi_{r,\alpha}(\theta_2)),
\end{equation*}
we have
\begin{align}
u_1 \sigma^2 \psi_{r,\alpha}(\theta_2) &= 0,
\label{eq-2-8}
\\
u_2 \sigma^2 \psi_{r,\alpha}(\theta_2) &= 0,
\label{eq-2-9}
\\
u_3 \sigma^2 \psi_{r,\alpha}'(\theta_2) + u_4 \psi_{r,\alpha}(\theta_2)
&= 0,
\label{eq-2-10}
\\
u_3 \sigma^2 \psi_{r/\sqrt{2},\alpha}'(\theta_2) 
+ u_4 \psi_{r/\sqrt{2},\alpha}(\theta_2) &= 0.
\label{eq-2-11}
\end{align}

Note that $\psi_{r,\alpha}(\theta_2) \neq 0$.
\eqref{eq-2-8} and \eqref{eq-2-9} yield $u_1 = u_2 = 0$, and 
\eqref{eq-2-10} and \eqref{eq-2-11} imply
\begin{equation*}
0=
u_4(-\psi_{r,\alpha}(\theta_2)\psi_{r/\sqrt{2},\alpha}'(\theta_2)
+\psi_{r/\sqrt{2},\alpha}(\theta_2)\psi_{r,\alpha}'(\theta_2))
=u_4 \psi_{r/\sqrt{2},\alpha}^2(\theta_2) 
\biggl(\frac{\psi_{r,\alpha}(\theta_2)}{\psi_{r/\sqrt{2},\alpha}(\theta_2)}\biggr)'.
\end{equation*}
Since $(\psi_{r,\alpha}(\theta_2)/\psi_{r/\sqrt{2},\alpha}(\theta_2))' \neq 0$
for $\alpha \in (0,2)$, which follows from 
Lemmas \ref{lem11} and \ref{lem12}, 
we get $u_4=0$ and $u_3=0$. 
This contradicts the assumption that $u \in \mathbb R^4 \setminus \{0\}$.
Therefore, 
there does not exist $u \in \mathbb R^4 \setminus \{0\}$ 
satisfying \eqref{eq-2-7},
which implies that $u^\TT V u > 0$.
\end{proof}

\begin{proof}[\bf{Proof of Theorem \ref{th2}}]
(ii) For a detailed proof, we refer to Theorem 3.3 
in Tonaki et al.\,\cite{TKU2022arXiv1}.
We first show that
\begin{equation}\label{eq-3-1}
\sqrt{n}
\begin{pmatrix}
\hat\sigma_{1,1}^2-(\sigma_{1,1}^*)^2
\\
\hat\sigma_{1,2}^2-(\sigma_{1,2}^*)^2
\end{pmatrix}
\dto 
N
\Biggl(0,2
\begin{pmatrix}
(\sigma_{1,1}^*)^4 & 0
\\
0 & (\sigma_{1,2}^*)^4 
\end{pmatrix}
\Biggr).
\end{equation}
In order to show \eqref{eq-3-1}, it is enough to verify that
under $\frac{n^{2-\alpha+\tau}}{(M_1 \land M_2)^{2\tau}}\to0$ 
for some $0<\tau \le 1$ such that $\tau < \alpha$, 
\begin{align*}
\sqrt{n}\Biggl(
\hat\sigma_{\ell_1,\ell_2}^2
-\sum_{i=1}^{n}
(x_{\ell_1,\ell_2}(\widetilde t_i)-x_{\ell_1,\ell_2}(\widetilde t_{i-1}))^2
\Biggr)
=o_p(1),
\end{align*}
that is, for $k=1,2,3$,
\begin{equation*}
n \sum_{i=1}^n \mathcal B_{k,i}^2 = o_p(1),
\end{equation*}
where
$\widetilde \Delta_i X(y,z)=X_{\widetilde t_i}(y,z)-X_{\widetilde t_{i-1}}(y,z)$,
\begin{align*}
\mathcal B_{1,i} 
&= \frac{1}{M}\sum_{j=1}^{M_1}\sum_{k=1}^{M_2}
\widetilde \Delta_i X(y_j,z_k)
\sin(\pi \ell_1 y_j)\sin(\pi \ell_2 z_k)
(\ee^{\hat\kappa y_j/2}\ee^{\hat\eta z_k/2}
-\ee^{\kappa^* y_j/2}\ee^{\eta^* z_k/2}),
\\
\mathcal B_{2,i} &= 
\sum_{j=1}^{M_1}\sum_{k=1}^{M_2}
\widetilde \Delta_i X(y_j,z_k)
\\
&\qquad \times
\int_{z_{k-1}}^{z_k} \int_{y_{j-1}}^{y_j}
(\sin(\pi \ell_1 y_j)\sin(\pi \ell_2 z_k)
\ee^{\kappa^* y_j/2}\ee^{\eta^* z_k/2}
-\sin(\pi \ell_1 y)\sin(\pi \ell_2 z)
\ee^{\kappa^* y/2}\ee^{\eta^* z/2}) \dd y\dd z,
\\
\mathcal B_{3,i} &= 
\sum_{j=1}^{M_1}\sum_{k=1}^{M_2}
\int_{z_{k-1}}^{z_k} \int_{y_{j-1}}^{y_j}
(\widetilde \Delta_i X(y_j,z_k)-\widetilde \Delta_i X(y,z))
\sin(\pi \ell_1 y)\sin(\pi \ell_2 z)
\ee^{\kappa^* y/2}\ee^{\eta^* z/2}
\dd y\dd z.
\end{align*}

We obtain from the Taylor expansion that 
\begin{equation*}
n \sum_{i=1}^n \mathcal B_{1,i}^2
\le\mathcal C_n
\times m N(|\hat\kappa-\kappa^*|^2+|\hat\eta-\eta^*|^2),
\end{equation*}
where $\mathcal C_n$ satisfies that
for $\epsilon_1, \epsilon_2>0$, 
\begin{align*}
\mathcal C_n
\lesssim
\frac{n}{mN}\sum_{i=1}^{n}
\frac{1}{M}
\sum_{j=1}^{M_1}\sum_{k=1}^{M_2}
(\widetilde \Delta_i X)^2(y_j,z_k)
\end{align*}
on $\Omega_1=\{|\hat\kappa-\kappa^*|+|\hat\eta-\eta^*|<\epsilon_1\}$ and 
\begin{align*}
P(|\mathcal C_n|>\epsilon_2)
\lesssim
\frac{n^{2-\alpha}}{\epsilon_2 m N}+P(\Omega_1^{\mathrm c}),
\end{align*}
which together with $\frac{n^{2-\alpha}}{m N}\to0$ and Theorem \ref{th1}
establishes $n\sum_{i=1}^{n} \mathcal B_{1,i}^2 =o_p(1)$.
Noting that $\EE[\mathcal B_{2,i}^2] \lesssim \frac{1}{n^\alpha (M_1 \land M_2)^2}$,
we see that under $\frac{n^{2-\alpha+\tau}}{(M_1 \land M_2)^{2\tau}}\to0$, 
\begin{equation*}
n\sum_{i=1}^{n}
\mathcal B_{2,i}^2
=O_p\biggl(
\frac{n^{2-\alpha}}{(M_1 \land M_2)^2}
\biggr)
=o_p(1).
\end{equation*}
Since it follows that for $0< \tau \le 1$ with $\tau<\alpha$,
\begin{align*}
\EE\bigl[(\widetilde \Delta_i X(y_j,z_k)-\widetilde \Delta_i X(y,z))^2\bigr]
\lesssim
\frac{1}{(M_1 \land M_2)^{2\tau}}
\sum_{\ell_1, \ell_2 \ge 1}
\frac{1-\ee^{-\lambda_{\ell_1, \ell_2}/n}}
{\lambda_{\ell_1, \ell_2}^{1+\alpha-\tau}}
=O\biggl(
\frac{1}{n^{\alpha-\tau}(M_1 \land M_2)^{2\tau}}
\biggr)
\end{align*}
uniformly in $i,j,k$ and $(y,z) \in (y_{j-1},y_j]\times (z_{k-1},z_k]$,
we obtain that under $\frac{n^{2-\alpha+\tau}}{(M_1 \land M_2)^{2\tau}}\to0$, 
\begin{align*}
n\sum_{i=1}^{n}
\mathcal B_{3,i}^2
&\lesssim
n\sum_{i=1}^{n}
\sum_{j=1}^{M_1}\sum_{k=1}^{M_2}
\int_{z_{k-1}}^{z_k} \int_{y_{j-1}}^{y_j}
\bigl(\widetilde \Delta_i X(y_j,z_k)-\widetilde \Delta_i X(y,z)\bigr)^2
\dd y\dd z
\\
&=O_p\biggl(
\frac{n^{2-\alpha+\tau}}{(M_1 \land M_2)^{2\tau}}
\biggr)
=o_p(1).
\end{align*}

We next prove \eqref{eq-Th2.3}. It follows from
\begin{equation*}
\check \theta_2 = 
\frac{\check \lambda_{1,2}-\check \lambda_{1,1}}{3\pi^2}
= \frac{(\hat \sigma^2)^{1/\alpha}}{3\pi^2}
( (\hat \sigma_{1,2}^2)^{-1/\alpha} - (\hat \sigma_{1,1}^2)^{-1/\alpha} )
\end{equation*}
that
\begin{align*}
\sqrt{n} (\check \theta_2 - \theta_2^*) 
&= 
\frac{\sqrt{n}}{3\pi^2}
\Bigl[
(\hat \sigma^2)^{1/\alpha} 
\{ (\hat \sigma_{1,2}^2)^{-1/\alpha} - (\hat \sigma_{1,1}^2)^{-1/\alpha} \}
-(\sigma^*)^{2/\alpha} 
\{ (\sigma_{1,2}^*)^{-2/\alpha} - (\sigma_{1,1}^*)^{-2/\alpha} \}
\Bigr]
\\
&=
\frac{\sqrt{n}}{3\pi^2}
\{ (\hat \sigma^2)^{1/\alpha} - (\sigma^*)^{2/\alpha} \}
\{ (\hat \sigma_{1,2}^2)^{-1/\alpha} - (\hat \sigma_{1,1}^2)^{-1/\alpha} \}
\\
&\qquad+ 
\frac{\sqrt{n}(\sigma^*)^{2/\alpha}}{3\pi^2}
\Bigl[ \{(\hat \sigma_{1,2}^2)^{-1/\alpha} - (\hat \sigma_{1,1}^2)^{-1/\alpha} \}
- \{ (\sigma_{1,2}^*)^{-2/\alpha} - (\sigma_{1,1}^*)^{-2/\alpha} \}
\Bigr].
\end{align*}
Let $F(x,y) = y^{-1/\alpha} - x^{-1/\alpha}$ for $x,y>0$. 
Since it holds 
that
\begin{align*}
&\sqrt{n}
\Bigl[ \{(\hat \sigma_{1,2}^2)^{-1/\alpha} - (\hat \sigma_{1,1}^2)^{-1/\alpha} \}
- \{ (\sigma_{1,2}^*)^{-2/\alpha} - (\sigma_{1,1}^*)^{-2/\alpha} \}
\Bigr]
\\
&=
\sqrt{n} \{ F(\hat \sigma_{1,1}^2, \hat \sigma_{1,2}^2) 
-F((\sigma_{1,1}^*)^2, (\sigma_{1,2}^*)^2) \}
\\
&= 
\pd F( (\sigma_{1,1}^*)^2, (\sigma_{1,2}^*)^2 )
\sqrt{n} 
\begin{pmatrix}
\hat \sigma_{1,1}^2 - (\sigma_{1,1}^*)^2 
\\
\hat \sigma_{1,2}^2 - (\sigma_{1,2}^*)^2
\end{pmatrix}
+o_p(1)
\\
&=
\frac{1}{\alpha} ((\sigma_{1,1}^*)^{-2/\alpha-2}, -(\sigma_{1,2}^*)^{-2/\alpha-2} )
\sqrt{n} 
\begin{pmatrix}
\hat \sigma_{1,1}^2 - (\sigma_{1,1}^*)^2
\\
\hat \sigma_{1,2}^2 - (\sigma_{1,2}^*)^2
\end{pmatrix}
+o_p(1),
\end{align*}
we obtain from
$\sqrt{m N}\{(\hat \sigma^2)^{1/\alpha} - (\sigma^*)^{2/\alpha}\}=O_p(1)$ 
that
\begin{align*}
\sqrt{n} (\check \theta_2 - \theta_2^*) 
&= 
\frac{(\sigma^*)^{2/\alpha}}{3\pi^2\alpha}
((\sigma_{1,1}^*)^{-2/\alpha-2}, -(\sigma_{1,2}^*)^{-2/\alpha-2} )
\sqrt{n} 
\begin{pmatrix}
\hat \sigma_{1,1}^2 - (\sigma_{1,1}^*)^2
\\
\hat \sigma_{1,2}^2 - (\sigma_{1,2}^*)^2
\end{pmatrix}
+o_p(1)
\\
&= 
\frac{(\sigma^*)^{-2}}{3\pi^2\alpha}
((\lambda_{1,1}^*)^{1+\alpha}, -(\lambda_{1,2}^*)^{1+\alpha} )
\sqrt{n} 
\begin{pmatrix}
\hat \sigma_{1,1}^2 - (\sigma_{1,1}^*)^2
\\
\hat \sigma_{1,2}^2 - (\sigma_{1,2}^*)^2
\end{pmatrix}
+o_p(1).
\end{align*}
Similarly, it follows that
\begin{align*}
\sqrt{n} (\check \theta_0 - \theta_0^*) 
&=
-\sqrt{n}(\check \lambda_{1,1}-\lambda_{1,1}^*) + o_p(1)
\\
&= 
-(\sigma^*)^{2/\alpha}
\sqrt{n}
\{ (\hat \sigma_{1,1}^2)^{-1/\alpha} - (\sigma_{1,1}^*)^{-2/\alpha} \} + o_p(1)
\\
&=
-\frac{(\sigma^*)^{2/\alpha}}{\alpha} (\sigma_{1,1}^*)^{-2/\alpha-2}
\sqrt{n} \{ \hat \sigma_{1,1}^2 - (\sigma_{1,1}^*)^2 \}
+o_p(1)
\\
&= 
-\frac{(\sigma^*)^{-2}}{\alpha}
((\lambda_{1,1}^*)^{1+\alpha}, 0 )
\sqrt{n} 
\begin{pmatrix}
\hat \sigma_{1,1}^2 - (\sigma_{1,1}^*)^2
\\
\hat \sigma_{1,2}^2 - (\sigma_{1,2}^*)^2
\end{pmatrix}
+o_p(1).
\end{align*}
Moreover, one has   
\begin{align*}
\sqrt{n} (\check \theta_1 - \theta_1^*) 
&= \sqrt{n}(\hat \kappa - \kappa^*) \check \theta_2
+\kappa^* \sqrt{n}(\check \theta_2 - \theta_2^*)
\\
&= 
\frac{\kappa^* (\sigma^*)^{-2}}{3\pi^2\alpha}
((\lambda_{1,1}^*)^{1+\alpha}, -(\lambda_{1,2}^*)^{1+\alpha} )
\sqrt{n} 
\begin{pmatrix}
\hat \sigma_{1,1}^2 - (\sigma_{1,1}^*)^2
\\
\hat \sigma_{1,2}^2 - (\sigma_{1,2}^*)^2
\end{pmatrix}
+o_p(1),
\end{align*}
\begin{equation*}
\sqrt{n} (\check \eta_1 - \eta_1^*) 
= \frac{\eta^* (\sigma^*)^{-2}}{3\pi^2\alpha}
((\lambda_{1,1}^*)^{1+\alpha}, -(\lambda_{1,2}^*)^{1+\alpha} )
\sqrt{n} 
\begin{pmatrix}
\hat \sigma_{1,1}^2 - (\sigma_{1,1}^*)^2
\\
\hat \sigma_{1,2}^2 - (\sigma_{1,2}^*)^2
\end{pmatrix}
+o_p(1),
\end{equation*}
\begin{equation*}
\sqrt{n} (\check \sigma^2 - (\sigma^*)^2) 
= \frac{(\theta_2^*)^{-1}}{3\pi^2\alpha}
((\lambda_{1,1}^*)^{1+\alpha}, -(\lambda_{1,2}^*)^{1+\alpha} )
\sqrt{n} 
\begin{pmatrix}
\hat \sigma_{1,1}^2 - (\sigma_{1,1}^*)^2
\\
\hat \sigma_{1,2}^2 - (\sigma_{1,2}^*)^2
\end{pmatrix}
+o_p(1).
\end{equation*}
Therefore, setting
\begin{equation*}
\boldsymbol v_1 = (-3\pi^2, \kappa^*, \eta^*, 1, (\sigma^*)^2/\theta_2^*)^\TT,
\quad
\boldsymbol v_2 = (0, \kappa^*, \eta^*, 1, (\sigma^*)^2/\theta_2^*)^\TT,
\end{equation*}
we obtain from \eqref{eq-3-1} that
\begin{align*}
\sqrt{n}
\begin{pmatrix}
\check \theta_0 - \theta_0^*
\\
\check \theta_1 - \theta_1^*
\\
\check \eta_1 - \eta_1^*
\\
\check \theta_2 - \theta_2^*
\\
\check \sigma^2 - (\sigma^*)^2
\end{pmatrix}
&=
\frac{(\sigma^*)^{-2}}{3 \pi^2 \alpha}
((\lambda_{1,1}^*)^{1+\alpha} \boldsymbol v_1, 
-(\lambda_{1,2}^*)^{1+\alpha} \boldsymbol v_2)
\sqrt{n} 
\begin{pmatrix}
\hat \sigma_{1,1}^2 - (\sigma_{1,1}^*)^2
\\
\hat \sigma_{1,2}^2 - (\sigma_{1,2}^*)^2
\end{pmatrix}
+o_p(1)
\\
&\dto N(0,\mathcal J),
\end{align*}
where $\mathcal J$ is expressed by using 
$\boldsymbol\theta_{-1}^* = (\theta_1^*, \eta_1^*, \theta_2^*, (\sigma^*)^2)$
as follows.
\begin{align*}
\mathcal J &= 
\frac{2(\sigma^*)^{-4}}{9 \pi^4 \alpha^2}
((\lambda_{1,1}^*)^{1+\alpha} \boldsymbol v_1, 
-(\lambda_{1,2}^*)^{1+\alpha} \boldsymbol v_2)
\begin{pmatrix}
(\sigma_{1,1}^*)^4 & 0
\\
0 & (\sigma_{1,2}^*)^4 
\end{pmatrix}
((\lambda_{1,1}^*)^{1+\alpha} \boldsymbol v_1, 
-(\lambda_{1,2}^*)^{1+\alpha} \boldsymbol v_2)^\TT
\\
&=
\frac{2}{9\pi^4 \alpha^2} 
((\lambda_{1,1}^*)^2 \boldsymbol v_1 \boldsymbol v_1^\TT 
+ (\lambda_{1,2}^*)^2 \boldsymbol v_2 \boldsymbol v_2^\TT )
\\
&= 2
\begin{pmatrix}
(\lambda_{1,1}^*)^2 & 
\frac{-(\lambda_{1,1}^*)^2}{3\pi^2\theta_2^*}\boldsymbol\theta_{-1}^*
\\
\frac{-(\lambda_{1,1}^*)^2}{3\pi^2\theta_2^*} (\boldsymbol\theta_{-1}^*)^\TT &
\frac{(\lambda_{1,1}^*)^2 + (\lambda_{1,1}^*)^2}{9\pi^4 (\theta_2^*)^2 \alpha^2}
(\boldsymbol\theta_{-1}^*)^\TT \boldsymbol\theta_{-1}^*
\end{pmatrix}
\\
&= 2
\begin{pmatrix}
J_{1,1} & J_{1,2}
\\
J_{1,2}^\TT & J_{2,2}
\end{pmatrix}
.
\end{align*}

(i) Similarly, it can be shown that under 
$\frac{n^{1-\alpha+\tau}}{(M_1 \land M_2)^{2\tau}}\to0$ 
for some $0<\tau \le 1$ such that $\tau < \alpha$, 
\begin{equation*}
\hat\sigma_{\ell_1,\ell_2}^2
-\sum_{i=1}^{n}
(x_{\ell_1,\ell_2}(\widetilde t_i)-x_{\ell_1,\ell_2}(\widetilde t_{i-1}))^2
=o_p(1),
\end{equation*}
and thus we have $\check \theta_0 \pto \theta_0^*$.
\end{proof}

\subsection{Proofs of the results in Subsection \ref{sec2-3}}
\begin{proof}[\bf{Proof of Theorem \ref{th3}}]
Since it holds from Proposition \ref{prop1} and Lemma \ref{lem13} that under [A1], 
\begin{equation}\label{eq-999-0}
\EE[(T_{i,j,k}X)^2]
=\Delta^\alpha \sigma^2
\ee^{-\kappa(\widetilde y_{j-1}+\widetilde y_j)/2}
\ee^{-\eta (\widetilde z_{k-1}+\widetilde z_k)/2}
\widetilde \psi_{r,\alpha}(\theta_2)
+ R_{i,j,k} + O(\Delta^{1+\alpha}),
\end{equation}
where $\sum_{i=1}^N R_{i,j,k}=O(\Delta^{\alpha})$ uniformly in $j,k$, 
and it follows from Lemmas \ref{lem11} and \ref{lem12} that the function
$\theta_2 \mapsto \widetilde \psi_{r,\alpha}(\theta_2)
/\widetilde \psi_{r/\sqrt{2},\alpha}(\theta_2)$
is injective on $(\underline{\theta_2}(r,\alpha), \infty)$ 
and the parameters $(\theta_2, \kappa, \eta, \sigma^2)$ are identifiable, 
it can be shown in the same way as Theorem \ref{th1}.
\end{proof}

\begin{proof}[\bf{Proof of Theorem \ref{th4}}]
Since the proof is similar to that of Theorem \ref{th2}, 
we only verify that the asymptotic variances are given by $\mathcal K$ and $\mathcal L$ 
when $\mu_0$ is known and unknown, respectively.
Note that
\begin{equation}\label{eq-999-1}
\sqrt{n}
\begin{pmatrix}
\tilde \sigma_{1,1}^2-(\varsigma_{1,1}^*)^2
\\
\tilde \sigma_{1,2}^2-(\varsigma_{1,2}^*)^2
\end{pmatrix}
\dto 
N
\Biggl(0,2
\begin{pmatrix}
(\varsigma_{1,1}^*)^4 & 0
\\
0 & (\varsigma_{1,2}^*)^4 
\end{pmatrix}
\Biggr)
\end{equation}
is established as well as \eqref{eq-3-1}.

(a) Consider the case where $\mu_0$ is known.
Since it follows from \eqref{eq-999-1} that
$\sqrt{n}(\bar \sigma^2 -(\sigma^*)^2) \dto N(0,2(\sigma^*)^4)$, 
we see from Theorem \ref{th3} that
\begin{equation*}
\sqrt{n}
\begin{pmatrix}
\bar \theta_1 - \theta_1^*
\\
\bar \eta_1 - \eta_1^*
\\
\bar \theta_2 - \theta_2^*
\\
\bar \sigma^2 - (\sigma^*)^2
\end{pmatrix}
= (\sigma^*)^{-2} 
\begin{pmatrix}
\theta_1^*
\\
\eta_1^*
\\
\theta_2^*
\\
(\sigma^*)^2
\end{pmatrix}
\sqrt{n} (\bar \sigma^2 - (\sigma^*)^2) +o_p(1)
\dto N(0,\mathcal K).
\end{equation*}

(b) Consider the case where $\mu_0$ is unknown.
Let
$\Phi(x,y) = 3\pi^2 (x^{1/\alpha} y^{-1/\alpha} -1)^{-1}$, 
$\Psi(x,y) = \{ 3\pi^2 (y^{-1/\alpha} -x^{1/\alpha})^{-1}\}^\alpha$,
$\boldsymbol p_1 = \partial \Phi((\varsigma_{1,1}^*)^2,(\varsigma_{1,2}^*)^2)$
and $\boldsymbol p_2 = \partial \Psi((\varsigma_{1,1}^*)^2,(\varsigma_{1,2}^*)^2)$.
A simple calculation yields
\begin{align*}
\boldsymbol p_1 
\begin{pmatrix}
(\varsigma_{1,1}^*)^4 & 0
\\
0 & (\varsigma_{1,2}^*)^4 
\end{pmatrix}
\boldsymbol p_1^\TT
&= \frac{2(\mu_{1,1}^*)^2(\mu_{1,2}^*)^2}{9 \pi^4 \alpha^2},
\\
\boldsymbol p_1 
\begin{pmatrix}
(\varsigma_{1,1}^*)^4 & 0
\\
0 & (\varsigma_{1,2}^*)^4 
\end{pmatrix}
\boldsymbol p_2^\TT
&= \frac{\mu_{1,1}^*\mu_{1,2}^* (\mu_{1,1}^* +\mu_{1,2}^*)}{9 \pi^4 \alpha},
\\
\boldsymbol p_2 
\begin{pmatrix}
(\varsigma_{1,1}^*)^4 & 0
\\
0 & (\varsigma_{1,2}^*)^4 
\end{pmatrix}
\boldsymbol p_2^\TT
&= \frac{(\mu_{1,1}^*)^2 +(\mu_{1,2}^*)^2}{9 \pi^4},
\end{align*}
and therefore it follows from Theorem \ref{th3} and \eqref{eq-999-1} that
\begin{equation*}
\sqrt{n}
\begin{pmatrix}
\breve \mu_0 - \mu_0^*
\\
\breve \theta_1 - \theta_1^*
\\
\breve \eta_1 - \eta_1^*
\\
\breve \theta_2 - \theta_2^*
\\
\breve \sigma^2 - (\sigma^*)^2
\end{pmatrix}
= 
\left(
\begin{pmatrix}
1
\\
0
\\
0
\\
0
\\
0
\end{pmatrix}
\boldsymbol p_1 +
\begin{pmatrix}
0
\\
\theta_1^*
\\
\eta_1^*
\\
\theta_2^*
\\
(\sigma^*)^2
\end{pmatrix}
\boldsymbol p_2
\right)
\sqrt{n}
\begin{pmatrix}
\tilde \sigma_{1,1}^2-(\varsigma_{1,1}^*)^2
\\
\tilde \sigma_{1,2}^2-(\varsigma_{1,2}^*)^2
\end{pmatrix}
+o_p(1)
\dto N(0,\mathcal L).
\end{equation*}
\end{proof}

\subsection{Auxiliary results for Riemann summation and Fourier series}
We prepare several lemmas on Riemann summation and Fourier series.

Let $\alpha \in (0,2)$, $\beta \in \mathbb R$, $\gamma,\gamma_1,\gamma_2,J \ge0$ 
and $p \in \mathbb N \cup \{0\}$. 
We set some spaces of functions.

\begin{enumerate}
\item[1.]
Let $\boldsymbol F_{J,\beta}(\mathbb R_{+})$ be 
the subspace of functions in $\boldsymbol C^2(\mathbb R_{+})$, 
whose element $f$ satisfies
\begin{itemize}
\item[(i)]
$x^{\beta} f(x)$, $x^{\beta+1} f'(x)$, 
$x^{\beta+2} f''(x)\lesssim \ee^{-J x}$ ($x \to 0$), 

\item[(ii)]
$x^{\alpha+1} f(x)$, $x^{\alpha+2} f'(x)$,  
$x^{\alpha+3} f''(x) \lesssim \ee^{-J x}$ ($x \to \infty$).
\end{itemize}

\item[2.]
Let $\boldsymbol G^p_\gamma(\mathbb R^2)$ be 
the subspace of functions in $\boldsymbol C^p(\mathbb R^2)$,
whose element $g$ satisfies 
\begin{itemize}
\item[(i)]
$|\pd^j g(x,y)| \lesssim (x^2+y^2)^{\gamma+j/2}$ ($x, y \to 0$) for $j=0,\ldots, p$.

\item[(ii)]
$|\pd^j g(x,y)| \lesssim 1$ ($x, y \to \infty$) for $j=0,\ldots, p$.
\end{itemize}
In particular, 
we denote $\boldsymbol G_\gamma(\mathbb R^2) = \boldsymbol G^0_\gamma(\mathbb R^2)$.

\item[3.]
Let $\widetilde {\boldsymbol G}_{\gamma_1,\gamma_2}(\mathbb R^2)$ 
be the subspace of functions in $\boldsymbol C(\mathbb R^2)$,
whose element $g$ satisfies 
\begin{itemize}
\item[(i)]
there exist $g_1,g_2 \in \boldsymbol C(\mathbb R)$ such that $g(x,y) = g_1(x)g_2(y)$,

\item[(ii)]
$g_j(s) \lesssim s^{2\gamma_j}$ ($s \to 0$) for $j=1,2$,

\item[(iii)]
$g_j(s) \lesssim 1$ ($s \to \infty$) for $j=1,2$. 
\end{itemize}

\item[4.]
Let $\boldsymbol B_b(\mathbb R^2)$
be the space of bounded and measurable functions on $\mathbb R^2$.

\end{enumerate}

We will often use the following properties.
\begin{equation*}
\sum_{1\le \ell \le L} \frac{1}{\ell^\beta} =
\begin{cases}
O(L^{1-\beta}), & \beta < 1,
\\
O(\log L), & \beta = 1,
\\
O(1), & \beta >1,
\end{cases}
\quad
\sum_{\ell > L}\frac{1}{\ell^\beta}
=O\Bigl(\frac{1}{L^{\beta-1}}\Bigr), \quad \beta>1.
\end{equation*}
\begin{equation*}
\sum_{1\le \ell_1^2 + \ell_2^2 \le L^2} 
\frac{1}{(\ell_1^2+\ell_2^2)^\beta} =
\begin{cases}
O(L^{2(1-\beta)}), & \beta < 1,
\\
O(\log L), & \beta = 1,
\\
O(1), & \beta >1,
\end{cases}
\quad
\sum_{\ell_1^2+\ell_2^2 > L^2} 
\frac{1}{(\ell_1^2+\ell_2^2)^\beta}
=O\Bigl(\frac{1}{L^{2(\beta-1)}}\Bigr), \quad \beta>1.
\end{equation*}

\begin{lem}\label{lem1}
Let $\alpha \in (0,2)$ and $f \in \boldsymbol F_{0,\alpha}(\mathbb R_{+})$.
\begin{itemize}
\item[(1)]
For $\gamma \ge 0$, $g \in \boldsymbol G_\gamma(\mathbb R^2)$ 
and $h \in \boldsymbol B_b(\mathbb R^2)$, 
it holds that as $\epsilon \to 0$, 
\begin{align}
&\sum_{\ell_1,\ell_2\ge1} 
f(\lambda_{\ell_1,\ell_2}\epsilon^2) 
g(\ell_1 \epsilon,\ell_2 \epsilon) h(\ell_1 y, \ell_2 z)
\nonumber
\\
&=
\sum_{\ell_1,\ell_2\ge1} 
f(\theta_2\pi^2(\ell_1^2+\ell_2^2)\epsilon^2) 
g(\ell_1 \epsilon,\ell_2 \epsilon) h(\ell_1 y, \ell_2 z)
+
\begin{cases}
O(1), & \alpha < \gamma,
\\
O(\log\frac{1}{\epsilon}), & \alpha = \gamma,
\\
O(\epsilon^{2(\gamma-\alpha)}), & \alpha > \gamma
\end{cases}
\label{eq-lem1-1}
\end{align}
uniformly in $y,z\in\mathbb R$. In particular, as $\epsilon \to 0$,
\begin{equation}
\sum_{\ell_1,\ell_2\ge1} 
f(\lambda_{\ell_1,\ell_2}\epsilon^2) 
g(\ell_1 \epsilon,\ell_2 \epsilon) h(\ell_1 y, \ell_2 z)
=
\begin{cases}
O(\epsilon^{-2}), & \alpha < \gamma+1,
\\
O(\epsilon^{-2}\log\frac{1}{\epsilon}), & \alpha = \gamma+1,
\\
O(\epsilon^{2(\gamma-\alpha)}), & \alpha > \gamma+1
\end{cases}
\label{eq-lem1-2}
\end{equation}
uniformly in $y,z\in\mathbb R$.

\item[(2)]
For $\gamma > \alpha \lor 3/2$ and 
$g \in \widetilde {\boldsymbol G}_{\gamma-1,1}(\mathbb R^2) 
\cap \boldsymbol G_\gamma^2(\mathbb R^2)$, 
it holds that as $\epsilon \to 0$, 
\begin{equation}
\epsilon^2 \sum_{\ell_1,\ell_2\ge1}
f((\ell_1^2+\ell_2^2)\epsilon^2) g(\ell_1 \epsilon, \ell_2 \epsilon)
=\iint_{\mathbb R_{+}^2} f(x^2+y^2) g(x,y) \dd x \dd y
+O(\epsilon^2).
\label{eq-lem1-3}
\end{equation}

\end{itemize}
\end{lem}

\begin{proof}
(1) Since $f \in \boldsymbol F_{0,\alpha}(\mathbb R_{+})$, 
$g \in \boldsymbol G_\gamma(\mathbb R^2)$ 
and $h \in \boldsymbol B_b(\mathbb R^2)$,
it follows from the mean value theorem that
\begin{align*}
&\Biggl|
\sum_{\ell_1,\ell_2\ge1} 
f(\lambda_{\ell_1,\ell_2}\epsilon^2)
g(\ell_1 \epsilon,\ell_2 \epsilon) h(\ell_1 y, \ell_2 z)
-\sum_{\ell_1,\ell_2\ge1} 
f(\theta_2\pi^2(\ell_1^2+\ell_2^2)\epsilon^2)
g(\ell_1 \epsilon,\ell_2 \epsilon) h(\ell_1 y, \ell_2 z)
\Biggr|
\\
&\le
|\theta_2 \Gamma| \epsilon^2 
\sum_{\ell_1,\ell_2\ge1} 
\Biggl|
\int_0^1 f'(\theta_2(\pi^2(\ell_1^2+\ell_2^2)+u\Gamma)\epsilon^2) \dd u
g(\ell_1 \epsilon,\ell_2 \epsilon) h(\ell_1 y, \ell_2 z)
\Biggr|
\\
&\lesssim
\epsilon^2
\sum_{\ell_1,\ell_2\ge1} 
|f'(\theta_2\pi^2(\ell_1^2+\ell_2^2)\epsilon^2)|
|g(\ell_1 \epsilon,\ell_2 \epsilon)|
\\
&\lesssim
\epsilon^{2(\gamma-\alpha)} 
\sum_{1 \le \ell_1^2+\ell_2^2 \le 1/\epsilon^2} 
\frac{1}{(\ell_1^2+\ell_2^2)^{1+\alpha-\gamma}}
+\frac{1}{\epsilon^{2(1+\alpha)}} 
\sum_{\ell_1^2+\ell_2^2 > 1/\epsilon^2} 
\frac{1}{(\ell_1^2+\ell_2^2)^{2+\alpha}}
\\
&=
\begin{cases}
O(1), & \alpha<\gamma,
\\
O(-\log \epsilon), & \alpha=\gamma,
\\
O(\epsilon^{2(\gamma-\alpha)}), & \alpha>\gamma.
\end{cases}
\end{align*}

Similarly, it holds that
\begin{align}
&\sum_{\ell_1,\ell_2\ge1} 
f(\theta_2\pi^2(\ell_1^2+\ell_2^2)\epsilon^2) 
g(\ell_1 \epsilon,\ell_2 \epsilon) h(\ell_1 y, \ell_2 z)
\nonumber
\\
&\lesssim
\epsilon^{ 2(\gamma-\alpha)} 
\sum_{\ell_1^2+\ell_2^2 \le 1/\epsilon^2} 
\frac{1}{(\ell_1^2+\ell_2^2)^{\alpha-\gamma}}
+\frac{1}{\epsilon^{2(1+\alpha)}} 
\sum_{\ell_1^2+\ell_2^2 > 1/\epsilon^2} 
\frac{1}{(\ell_1^2+\ell_2^2)^{1+\alpha}}
\nonumber
\\
&=
\begin{cases}
O(\epsilon^{-2}), & \alpha<\gamma+1,
\\
O(-\epsilon^{-2}\log \epsilon), & \alpha=\gamma+1,
\\
O(\epsilon^{2(\gamma-\alpha)}), & \alpha>\gamma+1.
\end{cases}
\label{eq-4-1}
\end{align}
Therefore, \eqref{eq-lem1-1} and \eqref{eq-4-1} yield \eqref{eq-lem1-2}.

(2) For $\ell_1,\ell_2\ge1$, we set 
$E_{\ell_1,\ell_2}
=((\ell_1-1/2)\epsilon, (\ell_1+1/2)\epsilon]
\times ((\ell_2-1/2)\epsilon,(\ell_2+1/2)\epsilon]\subset \mathbb R_{+}^2$
and $E=\bigcup_{\ell_1,\ell_2\ge1}E_{\ell_1,\ell_2}$.
Since $f \in \boldsymbol F_{0,\alpha}(\mathbb R_{+})$ and
$g \in \boldsymbol G_\gamma^2(\mathbb R^2)$,
$G(x,y) = f(x^2+y^2) g(x,y)$ can be controlled as follows. 
\begin{equation}\label{eq-4-2}
|\pd^j G(x,y)| \lesssim 
\begin{cases}
(x^2+y^2)^{\gamma-\alpha-j/2},
\quad (x, y \to 0),
\\
\frac{1}{(x^2+y^2)^{1+\alpha+j/2}},
\quad (x, y \to \infty)
\end{cases}
\end{equation}
for $j=0,1,2$. 
Notice that $|\pd^2 G| \in L^1(\mathbb R_{+}^2)$ if $\gamma > \alpha$.
Therefore, it holds from the Taylor expansion that
\begin{align}
\Biggl|
\epsilon^2 \sum_{\ell_1,\ell_2\ge1}
G(\ell_1 \epsilon, \ell_2 \epsilon) -\iint_{E} G(x,y) \dd x \dd y
\Biggr|
&=
\Biggl|
\sum_{\ell_1,\ell_2\ge1}
\iint_{E_{\ell_1,\ell_2}}
(G(\ell_1 \epsilon,\ell_2\epsilon)-G(x,y)) \dd x \dd y
\Biggr|
\nonumber
\\
&\lesssim
\epsilon^2 \iint_E |\partial^2 G(x,y)|\dd x \dd y
\nonumber
\\
&\lesssim \epsilon^2.
\label{eq-4-3}
\end{align}
Since $g \in \widetilde {\boldsymbol G}_{\gamma-1,1}(\mathbb R^2)$,
if $\gamma > \alpha \lor 3/2$, then there exist $g_1,g_2 \in \boldsymbol C(\mathbb R)$
such that $g(x,y) = g_1(x)g_2(y)$, and it holds that
\begin{align*}
\int_0^{\epsilon/2}\int_0^{\epsilon/2}
G(x,y)\dd x \dd y
&=O\Biggl(
\int_0^{\epsilon/2}\int_0^{\epsilon/2}
(x^2+y^2)^{\gamma-\alpha} \dd x \dd y
\Biggr)
\\
&=O\Biggl(
\int_0^{\epsilon} r^{2\gamma-2\alpha+1} \dd r
\Biggr)
=O(\epsilon^{2(\gamma-\alpha+1)})
=O(\epsilon^2),
\\
\int_{\epsilon/2}^\infty \int_0^{\epsilon/2}
G(x,y)\dd x \dd y
&\lesssim 
\int_0^{\epsilon/2} x^{2(\gamma-1)} \dd x
\int_{\epsilon/2}^\infty |f(\theta_2 \pi^2 y^2)| |g_2(y)| \dd y
\\
&=O\Biggl(
\epsilon^{2\gamma-1}
\biggl(
\int_{\epsilon}^1 
y^{2-2\alpha}\dd y
+\int_1^\infty 
\frac{\dd y}{y^{2+2\alpha}}
\biggr)
\Biggr)
=O(\epsilon^2),
\\
\int_0^{\epsilon/2}\int_{\epsilon/2}^\infty 
G(x,y)\dd x \dd y
&\lesssim 
\int_{\epsilon/2}^\infty |f(\theta_2 \pi^2 x^2)| |g_1(x)| \dd x
\int_0^{\epsilon/2} y^2 \dd y
\\
&=O\Biggl(
\epsilon^3
\biggl(
\int_{\epsilon}^1 
x^{2(\gamma-\alpha-1)}\dd x
+\int_1^\infty 
\frac{\dd x}{x^{2+2\alpha}}
\biggr)
\Biggr)
=O(\epsilon^2),
\end{align*}
which yield
\begin{equation}\label{eq-4-4}
\iint_{\mathbb R_{+}^2\setminus E} G(x,y) \dd x \dd y = O(\epsilon^2).
\end{equation}
Therefore, \eqref{eq-4-3} and \eqref{eq-4-4} imply \eqref{eq-lem1-3}.
\end{proof}

The following lemma is useful for proving the theorems in Subsection \ref{sec2-3}.
\begin{lem}\label{lem13}
Let $\alpha \in (0,2)$ and $f \in \boldsymbol C^2(\mathbb R_{+})$ such that
$f(x) = f_1(x) f_2(x)$, 
\begin{equation*}
f_1(x) \lesssim 1 \ (x \to 0), 
\quad
x f_1(x) \lesssim 1 \ (x \to \infty), 
\quad
f_2(x) = x^{-\alpha}.
\end{equation*}
For $\gamma \ge 0$, $g \in \boldsymbol G_\gamma(\mathbb R^2)$ 
and $h \in \boldsymbol B_b(\mathbb R^2)$, 
it holds that as $\epsilon \to 0$, 
\begin{align*}
&\sum_{\ell_1,\ell_2\ge1} 
f_1(\lambda_{\ell_1,\ell_2}\epsilon^2) f_2(\theta_2 \mu_{\ell_1,\ell_2}\epsilon^2) 
g(\ell_1 \epsilon,\ell_2 \epsilon) h(\ell_1 y, \ell_2 z)
\nonumber
\\
&=
\sum_{\ell_1,\ell_2\ge1} 
f(\lambda_{\ell_1,\ell_2}\epsilon^2)
g(\ell_1 \epsilon,\ell_2 \epsilon) h(\ell_1 y, \ell_2 z)
+
\begin{cases}
O(1), & \alpha < \gamma,
\\
O(-\log \epsilon), & \alpha = \gamma,
\\
O(\epsilon^{2(\gamma-\alpha)}), & \alpha > \gamma
\end{cases}
\end{align*}
uniformly in $y,z\in\mathbb R$. 
\end{lem}
\begin{proof}
Let $\widetilde \Gamma = \mu_0 -\Gamma$.
Since 
\begin{equation*}
x^{\alpha +1} f_1(x) f_2'(x) \lesssim 1 \ (x \to 0), 
\quad
x^{\alpha +2} f_1(x) f_2'(x) \lesssim 1 \ (x \to \infty), 
\end{equation*}
$g \in \boldsymbol G_\gamma(\mathbb R^2)$ and $h \in \boldsymbol B_b(\mathbb R^2)$,
we see from the mean value theorem that
\begin{align*}
&\Biggl|
\sum_{\ell_1,\ell_2\ge1} 
f_1(\lambda_{\ell_1,\ell_2}\epsilon^2) f_2(\theta_2 \mu_{\ell_1,\ell_2}\epsilon^2) 
g(\ell_1 \epsilon,\ell_2 \epsilon) h(\ell_1 y, \ell_2 z)
-\sum_{\ell_1,\ell_2\ge1} 
f(\lambda_{\ell_1,\ell_2}\epsilon^2)
g(\ell_1 \epsilon,\ell_2 \epsilon) h(\ell_1 y, \ell_2 z)
\Biggr|
\\
&\le
|\theta_2 \widetilde \Gamma| \epsilon^2 
\sum_{\ell_1,\ell_2\ge1} 
\Biggl|
f_1(\lambda_{\ell_1,\ell_2}\epsilon^2) 
\int_0^1 f_2'(( \lambda_{\ell_1, \ell_2} +u \theta_2 \widetilde \Gamma)\epsilon^2) 
\dd u g(\ell_1 \epsilon,\ell_2 \epsilon) h(\ell_1 y, \ell_2 z)
\Biggr|
\\
&\lesssim
\epsilon^2
\sum_{\ell_1,\ell_2\ge1} 
|f_1(\theta_2\pi^2(\ell_1^2+\ell_2^2)\epsilon^2)
f_2'(\theta_2\pi^2(\ell_1^2+\ell_2^2)\epsilon^2)|
|g(\ell_1 \epsilon,\ell_2 \epsilon)|
\\
&\lesssim
\epsilon^{2(\gamma-\alpha)} 
\sum_{1 \le \ell_1^2+\ell_2^2 \le 1/\epsilon^2} 
\frac{1}{(\ell_1^2+\ell_2^2)^{1+\alpha-\gamma}}
+\frac{1}{\epsilon^{2(1+\alpha)}} 
\sum_{\ell_1^2+\ell_2^2 > 1/\epsilon^2} 
\frac{1}{(\ell_1^2+\ell_2^2)^{2+\alpha}}
\\
&=
\begin{cases}
O(1), & \alpha<\gamma,
\\
O(-\log \epsilon), & \alpha=\gamma,
\\
O(\epsilon^{2(\gamma-\alpha)}), & \alpha>\gamma.
\end{cases}
\end{align*}
\end{proof}

The following Lemmas \ref{lem2} and \ref{lem3} are assertions 
under more relaxed conditions than those in Lemmas A.7 and A.8 
in Hildebrandt and Trabs \cite{Hildebrandt_Trabs2021},
and Lemma \ref{lem4} is a version of two variables of Lemma \ref{lem2}.

\begin{lem}\label{lem2}
Let $\{a_{\ell}\}_{\ell=1}^L$ be a real sequence and $\tau \in \{\sin, \cos\}$.
Then, it holds that
\begin{equation*}
\Biggl|
\sum_{\ell=1}^L a_{\ell}\tau(\pi \ell y)
\Biggr|
\le \frac{K_L}{y \land (2-y)}
\end{equation*}
for any $y \in (0,2)$, where 
$K_L = |a_1|+\sum_{\ell=2}^L |a_\ell-a_{\ell-1}|$.
\end{lem}
\begin{proof}
See the proof of Lemma A.7 in Hildebrandt and Trabs \cite{Hildebrandt_Trabs2021}.
\end{proof}

\begin{lem}\label{lem3}
Let $G:\mathbb R \to \mathbb R$ be a function and define
\begin{equation*}
K(\epsilon)=
|D_{\epsilon} G(\epsilon)|
+\sum_{\ell \ge 1} |D_{\epsilon}^2 G(\ell \epsilon)|.
\end{equation*}
Then, for $y\in(0,2)$, as $\epsilon \to 0$, 
\begin{align*}
\sum_{\ell \ge 1} G(\ell \epsilon) \cos(\pi \ell y)
&=-\frac{G(\epsilon)}{2} 
+O\biggl( \frac{K(\epsilon)}{(y \land (2-y))^2} \biggr),
\\
\sum_{\ell \ge 1} G(\ell \epsilon) \sin(\pi \ell y)
&=\frac{G(\epsilon)}{2\tan(\frac{\pi y}{2})}
+O\biggl( \frac{K(\epsilon)}{(y \land (2-y))^2} \biggr).
\end{align*}
\end{lem}

\begin{proof}
It can be calculated that 
\begin{equation*}
\sum_{\ell \ge 1} G(\ell \epsilon) \cos(\pi \ell y)
=-\frac{G(\epsilon)}{2}
-\frac{1}{2\sin(\frac{\pi y}{2})}
\sum_{\ell \ge 1}
D_\epsilon G(\ell \epsilon)
\sin\Bigl(\pi \Bigl(\ell+\frac{1}{2}\Bigr)y\Bigr)
\end{equation*}
in the same way as the proof of Lemma A.8 in 
Hildebrandt and Trabs \cite{Hildebrandt_Trabs2021}.
By applying Lemma \ref{lem2}, it holds that
\begin{equation*}
\Biggl|
\sum_{\ell \ge 1}
D_\epsilon G(\ell \epsilon)
\sin\Bigl(\pi \Bigl(\ell+\frac{1}{2}\Bigr)y\Bigr)
\Biggr|
\le \frac{K(\epsilon)}{y \land (2-y)},
\end{equation*}
and thus the first statement is obtained.
Similarly, the second statement can be proved.
\end{proof}

\begin{lem}\label{lem4}
Let $\{a_{\ell_1,\ell_2}\}_{\ell_1,\ell_2=1}^L$ be a real double sequence and 
$\tau_1, \tau_2 \in \{\sin, \cos\}$. Then,
it holds that
\begin{equation*}
\Biggl|
\sum_{\ell_1,\ell_2=1}^L a_{\ell_1,\ell_2}\tau_1(\pi \ell_1 y)\tau_2(\pi \ell_2 z)
\Biggr|
\le
\frac{K_L}{(y \land (2-y))(z \land (2-z))}
\end{equation*}
for any $y,z\in(0,2)$, where 
\begin{align*}
K_L &= |a_{1,1}| + \sum_{\ell_1=2}^L |a_{\ell_1,1}-a_{\ell_1-1,1}|
+\sum_{\ell_2=2}^L |a_{1,\ell_2}-a_{1,\ell_2-1}|
\\
&\qquad
+\sum_{\ell_1,\ell_2=2}^L |a_{\ell_1,\ell_2}
-a_{\ell_1-1,\ell_2}-a_{\ell_1,\ell_2-1}+a_{\ell_1-1,\ell_2-1}|.
\end{align*}
\end{lem}
\begin{proof}
Since it can be decomposed as
\begin{equation*}
\sum_{\ell=1}^L b_\ell \tau_1(\pi \ell y)
=b_1 \sum_{\ell=1}^L \tau_1(\pi \ell y)
+\sum_{k=2}^L (b_k-b_{k-1})\sum_{\ell=k}^L \tau_1(\pi \ell y),
\end{equation*}
it follows that
\begin{align*}
&\sum_{\ell_1,\ell_2=1}^L a_{\ell_1,\ell_2}\tau_1(\pi \ell_1 y)\tau_2(\pi \ell_2 z)
\\
&= a_{1,1}\sum_{\ell_1,\ell_2=1}^L \tau_1(\pi \ell_1 y)\tau_2(\pi \ell_2 z)
+\sum_{m=2}^L (a_{m,1}-a_{m-1,1})
\sum_{\ell_1=m}^L \sum_{\ell_2=1}^L \tau_1(\pi \ell_1 y)\tau_2(\pi \ell_2 z)
\\
&\qquad
+\sum_{k=2}^L (a_{1,k}-a_{1,k-1})
\sum_{\ell_1=1}^L\sum_{\ell_2=k}^L  \tau_1(\pi \ell_1 y)\tau_2(\pi \ell_2 z)
\\
&\qquad
+\sum_{k,m=2}^L (a_{m,k}-a_{m-1,k}-a_{m,k-1}+a_{m-1,k-1})
\sum_{\ell_1,\ell_2=m}^L \tau_1(\pi \ell_1 y)\tau_2(\pi \ell_2 z).
\end{align*}
It also holds from Lagrange's trigonometric identities that
\begin{equation*}
\Biggl|
\sum_{\ell_1=L_1}^{L_2} \tau_1(\pi \ell_1 y)
\Biggr|
\le \frac{1}{\sin(\frac{\pi y}{2})}
\le \frac{1}{y \land (2-y)}
\end{equation*}
uniformly in $L_1 \le L_2$, and thus we obtain
\begin{align*}
&\Biggl|
\sum_{\ell_1,\ell_2=1}^L a_{\ell_1,\ell_2}\tau_1(\pi \ell_1 y)\tau_2(\pi \ell_2 z)
\Biggr|
\\
&\le
|a_{1,1}| \Biggl| \sum_{\ell_1,\ell_2=1}^L 
\tau_1(\pi \ell_1 y)\tau_2(\pi \ell_2 z) \Biggr|
+\sum_{m=2}^L |a_{m,1}-a_{m-1,1}|
\Biggl| \sum_{\ell_1=m}^L \sum_{\ell_2=1}^L 
\tau_1(\pi \ell_1 y)\tau_2(\pi \ell_2 z) \Biggr|
\\
&\qquad
+\sum_{k=2}^L |a_{1,k}-a_{1,k-1}|
\Biggl| \sum_{\ell_1=1}^L\sum_{\ell_2=k}^L 
\tau_1(\pi \ell_1 y)\tau_2(\pi \ell_2 z) \Biggr|
\\
&\qquad
+\sum_{k,m=2}^L |a_{m,k}-a_{m-1,k}-a_{m,k-1}+a_{m-1,k-1}|
\Biggl| \sum_{\ell_1,\ell_2=m}^L \tau_1(\pi \ell_1 y)\tau_2(\pi \ell_2 z) \Biggr|
\\
&\le
\frac{K_L}{(y \land (2-y))(z \land (2-z))}.
\end{align*}
\end{proof}

Using the above three lemmas, we get the following results.
\begin{lem}\label{lem5}
Let $\alpha \in (0,2)$, $\gamma \ge 1$ and $\tau_1, \tau_2 \in \{\sin,\cos\}$. 
For $f \in \boldsymbol F_{0,\alpha}(\mathbb R_{+})$ and 
$g \in \widetilde {\boldsymbol G}_{\gamma-1,1}(\mathbb R^2) 
\cap \boldsymbol G_\gamma^2(\mathbb R^2)$, 
define $G(x,y)=f(x^2+y^2)g(x,y)$.

\begin{itemize}
\item[(1)]
For $y\in(0,2)$, as $\epsilon\to0$,
\begin{equation}\label{eq-lem5-1}
\sum_{\ell_1,\ell_2 \ge 1}
G(\ell_1\epsilon,\ell_2\epsilon) \tau_1(\pi \ell_1 y) =
\begin{cases}
O \Bigl( \frac{\epsilon^{-1}}{(y \land (2-y))^2} \Bigr),  
& \gamma>(\alpha-1/2) \lor 1,
\\
O \Bigl( \frac{1}{(y \land (2-y))^2} \Bigr), & \gamma>\alpha \lor 3/2.
\end{cases}
\end{equation}

\item[(2)]
For $y,z \in(0,2)$, as $\epsilon\to0$,
\begin{equation}\label{eq-lem5-2}
\sum_{\ell_1,\ell_2 \ge 1} 
G(\ell_1\epsilon,\ell_2\epsilon) \tau_1(\pi \ell_1 y) \tau_2(\pi \ell_2 z) =
\begin{cases}
O \Bigl( \frac{\epsilon^{-1}}{(y \land (2-y))(z \land (2-z))} \Bigr),  
& \gamma \ge 3/2,
\\
O \Bigl( \frac{1}{(y \land (2-y))(z \land (2-z))} \Bigr), & \gamma \ge 2.
\end{cases}
\end{equation}
\end{itemize}
\end{lem}

\begin{proof}
(1) According to \eqref{eq-4-2}, it follows that 
$|\pd^j G| \in L^1(\mathbb R_{+}^2)$ for $\gamma>\alpha-1+j/2$.
It also holds from
$f \in \boldsymbol F_{0,\alpha}(\mathbb R_{+})$ and 
$g \in \widetilde {\boldsymbol G}_{\gamma-1,1}(\mathbb R^2)$ that 
\begin{align}
\sum_{\ell_2\ge1}|G(\epsilon,\ell_2\epsilon)|
&\lesssim
\epsilon^{2(\gamma-1)}
\sum_{\ell_2\ge1}|f((\ell_2^2+1)\epsilon^2)| ((\ell_2 \epsilon)^2 \land 1)
\nonumber
\\
&\lesssim
\epsilon^{2(\gamma-\alpha)}
\sum_{1\le \ell_2 \le 1/\epsilon}
\frac{1}{\ell_2^{2(\alpha-1)}}
+\epsilon^{2(\gamma-\alpha-2)}
\sum_{\ell_2 > 1/\epsilon}
\frac{1}{\ell_2^{2\alpha+2}}
\nonumber
\\
& =
\begin{cases}
O(\epsilon^{2\gamma-3}), & \alpha < 3/2,
\\
O(-\epsilon^{2\gamma-3} \log \epsilon), & \alpha = 3/2,
\\
O(\epsilon^{2(\gamma-\alpha)}), & \alpha > 3/2.
\end{cases}
\label{eq-5-1}
\end{align}
Therefore, one has
\begin{itemize}
\item[(i)]
$|\pd G| \in L^1(\mathbb R_{+}^2)$ and 
$\epsilon \sum_{\ell_2\ge1}|G(\epsilon,\ell_2\epsilon)| =O(1)$
if $\gamma>(\alpha-1/2) \lor 1$,

\item[(ii)]
$|\pd^2 G| \in L^1(\mathbb R_{+}^2)$ and 
$\sum_{\ell_2\ge1}|G(\epsilon,\ell_2\epsilon)| =O(1)$
if $\gamma>\alpha \lor 3/2$.
\end{itemize}
Since for $G \in \boldsymbol C^2(\mathbb R^2)$, 
\begin{equation*}
|D_{1,\epsilon}^k D_{2,\epsilon}^m G(\epsilon x,\epsilon y)| 
\lesssim 
\min_{j=0,\ldots,(k+m) \land 2} \epsilon^{j} 
\sup_{u,v\in[0,1]} |\pd^{j} G(\epsilon(x+u), \epsilon(y+v))|,
\end{equation*}
we have
\begin{align*}
K(\epsilon) &= \Biggl|
\sum_{\ell_2 \ge 1} 
D_{1,\epsilon} G(\epsilon,\ell_2 \epsilon)
\Biggr|
+\sum_{\ell_1 \ge 1} \Biggl|
\sum_{\ell_2 \ge 1}
D_{1,\epsilon}^2 G(\ell_1 \epsilon,\ell_2 \epsilon)
\Biggr|
\\
&\lesssim
\sum_{\ell_2 \ge 1} 
\sup_{u\in[0,1]} |G(\epsilon(1+u),\ell_2 \epsilon)|
+
\begin{cases}
\displaystyle
\epsilon
\sum_{\ell_1,\ell_2 \ge 1}
\sup_{u,v\in[0,1]}
|\pd_x G((\ell_1+u) \epsilon,(\ell_2+v) \epsilon)|
\\
\displaystyle
\epsilon^2 
\sum_{\ell_1,\ell_2 \ge 1}
\sup_{u,v\in[0,1]}
|\pd_x^2 G((\ell_1+u) \epsilon,(\ell_2+v) \epsilon)|
\end{cases}
\\
&\lesssim
\sum_{\ell_2 \ge 1} 
|G(\epsilon,\ell_2 \epsilon)|
+
\begin{cases}
\displaystyle
\epsilon^{-1}\iint_{\mathbb R_{+}^2} |\pd_x G(x,y)| \dd x \dd y
\\
\displaystyle
\iint_{\mathbb R_{+}^2} |\pd_x^2 G(x,y)| \dd x \dd y
\end{cases}
\\
&\lesssim
\begin{cases}
\epsilon^{-1}, & \gamma>(\alpha-1/2) \lor 1,
\\
1, & \gamma>\alpha \lor 3/2.
\end{cases}
\end{align*}
Therefore, Lemma \ref{lem3} yields \eqref{eq-lem5-1}. 

(2) Note that
\begin{align*}
K(\epsilon) 
&= |G(\epsilon, \epsilon)|
+\sum_{\ell_1 \ge 1} |G(\ell_1\epsilon, \epsilon)|
+\sum_{\ell_2 \ge 1} |G(\epsilon, \ell_2 \epsilon)|
+\sum_{\ell_1,\ell_2 \ge 2}
|D_{1,\epsilon} D_{2,\epsilon} G((\ell_1-1)\epsilon,(\ell_2-1)\epsilon)|
\\
&\lesssim
|G(\epsilon, \epsilon)|
+\sum_{\ell_1 \ge 1} |G(\ell_1\epsilon, \epsilon)|
+\sum_{\ell_2 \ge 1} |G(\epsilon, \ell_2 \epsilon)|
+\epsilon^2 \sum_{\ell_1,\ell_2 \ge 1}
|\pd^2 G(\ell_1 \epsilon, \ell_2 \epsilon)|.
\end{align*}
In the same way as \eqref{eq-5-1}, it follows that 
$|G(\epsilon, \epsilon)| \lesssim \epsilon^{2(\gamma-\alpha)}$,
\begin{align*}
\sum_{\ell_1 \ge 1} |G(\ell_1 \epsilon, \epsilon)|
&\lesssim
\epsilon^2
\sum_{\ell_1 \ge 1}
|f((\ell_1^2+1)\epsilon^2)|
((\ell_1\epsilon)^{2(\gamma-1)} \land 1)
\\
&\lesssim
\epsilon^{2(\gamma-\alpha)}
\sum_{1\le \ell_1 \le 1/\epsilon}
\frac{1}{\ell_1^{2(\alpha-\gamma+1)}}
+\epsilon^{-2\alpha}
\sum_{\ell_1 > 1/\epsilon}
\frac{1}{\ell_1^{2\alpha+2}}
\\
& =
\begin{cases}
O(\epsilon), & \alpha < \gamma -1/2,
\\
O(-\epsilon \log \epsilon), & \alpha = \gamma -1/2,
\\
O(\epsilon^{2(\gamma-\alpha)}), & \alpha > \gamma -1/2,
\end{cases}
\end{align*}
\begin{align*}
\epsilon^2 
\sum_{\ell_1,\ell_2 \ge 1} 
|\pd^2 G(\ell_1\epsilon, \ell_2 \epsilon)|
&\lesssim
\epsilon^{2(\gamma-\alpha)}
\sum_{1 \le \ell_1^2 + \ell_2^2 \le 1/\epsilon^2}
\frac{1}{(\ell_1^2+\ell_2^2)^{\alpha-\gamma+1}}
\\
&\qquad
+\epsilon^{-2(1+\alpha)}
\sum_{\ell_1^2 + \ell_2^2 > 1/\epsilon^2}
\frac{1}{(\ell_1^2+\ell_2^2)^{\alpha+2}}
\\
&=
\begin{cases}
O(1), & \alpha < \gamma,
\\
O(- \log \epsilon), & \alpha = \gamma,
\\
O(\epsilon^{2(\gamma-\alpha)}), & \alpha >\gamma.
\end{cases}
\end{align*}
Therefore, it holds that
$\epsilon K(\epsilon) = O(1)$ if $\gamma \ge 3/2$, 
and $K(\epsilon) =O(1)$ if $\gamma \ge 2$ 
for $\alpha \in (0,2)$.
Applying Lemma \ref{lem4}, we see that
\begin{align*}
\Biggl|
\sum_{\ell_1,\ell_2\ge1} 
G(\ell_1 \epsilon, \ell_2 \epsilon) 
\cos(\pi \ell_1 y)\cos(\pi \ell_2 z)
\Biggr|
&\lesssim
\frac{K(\epsilon)}{(y \land (2-y))(z \land (2-z))}
\\
&=
\begin{cases}
O\Bigl(\frac{\epsilon^{-1}}{(y \land (2-y))(z \land (2-z))}\Bigr), 
& \gamma \ge 3/2,
\\
O\Bigl(\frac{1}{(y \land (2-y))(z \land (2-z))}\Bigr), 
& \gamma \ge 2.
\end{cases}
\end{align*}

\end{proof}

\subsection{
Boundedness of two-dimensional Fourier cosine series and its differences}
Let $\alpha \in (0,2)$ and $r \in (0,\infty)$.
For $f \in \boldsymbol F_{0,\alpha}(\mathbb R_{+})$ and $y,z \in [0,2)$, define
\begin{equation}\label{eq-6-1}
F(y,z)=
\Delta^{1+\alpha}
\sum_{\ell_1,\ell_2\ge1}
f(\lambda_{\ell_1,\ell_2}\Delta) \cos(\pi \ell_1 y)\cos(\pi \ell_2 z).
\end{equation}
Note that $F(y,z)$ is symmetric with respect to $y$ and $z$.
By setting 
\begin{align*}
&g_1(x,y) = 1,&
&g_2(x,y) = \sin(x),&
&g_3(x,y) = 1-\cos(x),
\\
&g_4(x,y) = \sin^2(x),&
&g_5(x,y) = \sin(x)\sin(y),&
&g_6(x,y) = \sin(x)(\cos(y)-1),
\\
&g_7(x,y) = \sin(x)\sin^2(y),&
&g_8(x,y) = \sin^2(x)(\cos(y)-1),&
&g_9(x,y) = \sin^2(x)\sin^2(y),
\end{align*}
it follows that
\begin{equation}\label{eq-6-2}
g_1 \in \boldsymbol G_0(\mathbb R^2),
\quad
g_2 \in \boldsymbol G_{1/2}(\mathbb R^2),
\quad
g_3,g_4,g_5 \in \boldsymbol G_1(\mathbb R^2),
\end{equation}
\begin{equation}\label{eq-6-3}
g_6,g_7 \in \widetilde {\boldsymbol G}_{1/2,1}(\mathbb R^2) 
\subset \boldsymbol G_{3/2}(\mathbb R^2),
\quad
g_8,g_9 \in \widetilde {\boldsymbol G}_{1,1}(\mathbb R^2) 
\subset \boldsymbol G_2(\mathbb R^2).
\end{equation}

The following three lemmas are used to control 
the Fourier cosine series such as \eqref{eq-6-1}.

\begin{lem}\label{lem6}
For $\alpha\in(0,2)$ and $\delta=r\sqrt{\Delta}$, it follows that
\begin{itemize}
\item[(1)]
$F(y,z) =O(\Delta^{\alpha-1})$ uniformly in $y,z \in [0,2)$,

\item[(2)]
$D_{1,\delta} F(y,z) = O(\Delta^{\alpha-1/2})$ 
uniformly in $y \in (0,2)$ and $z \in [0,2)$,

\item[(3)]
$D_{1,\delta} F(0,z) = O(\Delta^{\alpha})$ uniformly in $z \in [0,2)$,

\item[(4)]
$D_{1,\delta}^2 F(y,z)=O(\Delta^{\alpha})$
uniformly in $y, z \in [0,2)$,

\item[(5)]
$D_{1,\delta}D_{2,\delta} F(y,z)=O(\Delta^{\alpha})$
uniformly in $y,z \in (0,2)$.

\end{itemize}
\end{lem}

\begin{proof}
Notice that
the trigonometric identities
\begin{equation}\label{eq-6-4}
\cos(a+b)-\cos(a)= -2 \sin\Bigl(\frac{b}{2}\Bigr)\sin\Bigl(a+\frac{b}{2}\Bigr),
\end{equation}
\begin{equation}\label{eq-6-5}
\cos(a+2b)-2\cos(a+b)+\cos(a)=-4\sin^2\Bigl(\frac{b}{2}\Bigr)\cos(a+b).
\end{equation}

(1) Noting that
\begin{equation*}
|F(y,z)| \le 
\Delta^{\alpha-1}
\Biggl(\Delta^2 \sum_{\ell_1,\ell_2\ge1} |f(\lambda_{\ell_1,\ell_2}\Delta)| \Biggr)
\end{equation*}
and \eqref{eq-6-2}, we see from \eqref{eq-lem1-2} with $\gamma=0$ that
\begin{align*}
\Delta^2 \sum_{\ell_1,\ell_2\ge1} |f(\lambda_{\ell_1,\ell_2}\Delta)|
&=
\begin{cases}
O(\Delta), & \alpha < 1,
\\
O(-\Delta \log \Delta), & \alpha = 1,
\\
O(\Delta^{2-\alpha}), & \alpha >1,
\end{cases}
\end{align*}
and the desired result is obtained for $\alpha \in (0,2)$. 

(2) The trigonometric identity \eqref{eq-6-4} and $\delta=r\sqrt{\Delta}$ yield
\begin{equation*}
|D_{1,\delta} F(y,z)| 
\lesssim \Delta^{\alpha-1/2}
\Biggl(
\Delta^{3/2}\sum_{\ell_1,\ell_2\ge1} 
|f (\lambda_{\ell_1,\ell_2}\Delta)| 
\Bigl|\sin\Bigl(\frac{\pi r\ell_1\sqrt{\Delta}}{2}\Bigr)\Bigr|
\Biggr).
\end{equation*}
Therefore, it follows from \eqref{eq-6-2} and \eqref{eq-lem1-2} with $\gamma=1/2$ that
\begin{align*}
\Delta^{3/2}\sum_{\ell_1,\ell_2\ge1} 
|f(\lambda_{\ell_1,\ell_2}\Delta)| 
\biggl|\sin\Bigl(\frac{\pi r\ell_1\sqrt{\Delta}}{2}\Bigr)\biggr|
&=
\begin{cases}
O(\Delta^{1/2}), & \alpha < 3/2,
\\
O(-\Delta^{1/2} \log \Delta), & \alpha = 3/2,
\\
O(\Delta^{2-\alpha}), & \alpha > 3/2,
\end{cases}
\end{align*}
which yields the desired result for $\alpha \in (0,2)$.

(3)-(5) 
Since \eqref{eq-6-4}, \eqref{eq-6-5} and \eqref{eq-6-2},
we obtain from \eqref{eq-lem1-2} with $\gamma=1$ that
\begin{align*}
|D_{1,\delta} F(0,z)| 
&\le \Delta^{\alpha}
\Biggl(
\Delta \sum_{\ell_1,\ell_2\ge1} 
|f (\lambda_{\ell_1,\ell_2}\Delta)| (1-\cos(\pi r\ell_1\sqrt{\Delta}))
\Biggr)
=O(\Delta^\alpha),
\\
|D_{1,\delta}^2 F(y,z)|
&\lesssim 
\Delta^{\alpha}
\Biggl(
\Delta\sum_{\ell_1,\ell_2\ge1} 
|f(\lambda_{\ell_1,\ell_2}\Delta)| 
\sin^2\Bigl(\frac{\pi r\ell_1\sqrt{\Delta}}{2}\Bigr)
\Biggr)
=O(\Delta^\alpha),
\\
|D_{1,\delta}D_{2,\delta} F(y,z)|
&\lesssim 
\Delta^{\alpha}
\Biggl(
\Delta \sum_{\ell_1,\ell_2\ge1} 
|f(\lambda_{\ell_1,\ell_2}\Delta)| 
\biggl|\sin\Bigl(\frac{\pi r\ell_1\sqrt{\Delta}}{2}\Bigr)\biggr|
\biggl|\sin\Bigl(\frac{\pi r\ell_2\sqrt{\Delta}}{2}\Bigr)\biggr|
\Biggr)
=O(\Delta^\alpha).
\end{align*}
\end{proof}

\begin{lem}\label{lem7}
For $\alpha\in(0,2)$, $\delta = r \sqrt{\Delta}$ and $y,z\in(0,2-\delta)$,
it follows that
\begin{itemize}
\item[(1)]
$D_{1,\delta} D_{2,\delta} F(y,0)
= O\Bigl(\frac{\Delta^{\alpha+1/2}}{(y \land (2-\delta-y))^2}\Bigr)$,

\item[(2)]
$D_{1,\delta}^2 D_{2,\delta} F(y,0)
= O\Bigl(
\frac{\Delta^{\alpha+1}}{(y \land (2-\delta-y))^2}
\Bigr)$,

\item[(3)]
$D_{1,\delta}^2 D_{2,\delta} F(y,z)
= O\Bigl(
\frac{\Delta^{\alpha+1/2}}{(y \land (2-\delta-y))(z \land (2-\delta-z))}
\Bigr)$,

\item[(4)]
$D_{1,\delta}^2 D_{2,\delta}^2 F(y,z) 
= O\Bigl(
\frac{\Delta^{\alpha+1}}{(y \land (2-\delta-y))(z \land (2-\delta-z))}
\Bigr)$.

\end{itemize}
\end{lem}

\begin{proof}
(1), (2) Note that \eqref{eq-6-3}-\eqref{eq-6-5},
\eqref{eq-lem1-1} and \eqref{eq-lem5-1} with $\gamma=3/2$ or $\gamma = 2$ yield that 
\begin{align*}
&|D_{1,\delta} D_{2,\delta} F(y,0)|
\\
&\lesssim
\Delta^{1/2+\alpha}
\Biggl|
\Delta^{1/2}
\sum_{\ell_1,\ell_2\ge1} 
f(\lambda_{\ell_1,\ell_2}\Delta) 
\sin\Bigl(\frac{\pi r\ell_1\sqrt{\Delta}}{2}\Bigr)
(\cos(\pi r \ell_2\sqrt{\Delta})-1)
\sin\Bigl(\pi \ell_1 \Bigl(y+\frac{\delta}{2}\Bigr)\Bigr)
\Biggr|
\\
&=
O\biggl( \frac{\Delta^{1/2+\alpha}}{(y \land (2-\delta-y))^2} \biggr) 
+
\begin{cases}
O(\Delta^{1/2+\alpha} \cdot \Delta^{1/2}), & \alpha < 3/2,
\\
O(\Delta^{1/2+\alpha} \cdot (-\Delta^{1/2} \log \Delta)), & \alpha = 3/2,
\\
O(\Delta^{1/2+\alpha} \cdot \Delta^{2-\alpha}), & \alpha > 3/2,
\end{cases}
\end{align*}
\begin{align*}
|D_{1,\delta}^2 D_{2,\delta} F(y,0)|
&\lesssim
\Delta^{1+\alpha}
\Biggl|
\sum_{\ell_1,\ell_2\ge1} 
f(\lambda_{\ell_1,\ell_2}\Delta) 
\sin^2 \Bigl(\frac{\pi r\ell_1\sqrt{\Delta}}{2}\Bigr)
(\cos(\pi r \ell_2\sqrt{\Delta})-1)
\cos(\pi \ell_1 (y+\delta))
\Biggr|
\\
&=O\biggl(\frac{\Delta^{1+\alpha}}{(y \land (2-\delta-y))^2}\biggr) 
+ O(\Delta^{1+\alpha}).
\end{align*}

(3), (4) 
It follows from \eqref{eq-6-3}-\eqref{eq-6-5}, \eqref{eq-lem1-1} and \eqref{eq-lem5-2} 
with $\gamma=3/2$ or $\gamma = 2$ that
\begin{align*}
&|D_{1,\delta}^2 D_{2,\delta} F(y,z)|
\\
&\lesssim
\Delta^{1/2+\alpha}
\Biggl|
\Delta^{1/2}
\sum_{\ell_1,\ell_2\ge1} 
f(\lambda_{\ell_1,\ell_2}\Delta) 
\sin^2 \Bigl(\frac{\pi r\ell_1\sqrt{\Delta}}{2}\Bigr)
\sin\Bigl(\frac{\pi r\ell_2\sqrt{\Delta}}{2}\Bigr)
\cos(\pi \ell_1 (y+\delta))
\sin\Bigl(\pi \ell_2 \Bigl(z+\frac{\delta}{2}\Bigr)\Bigr)
\Biggr|
\\
&= O\biggl(
\frac{\Delta^{1/2+\alpha}}{(y \land (2-\delta-y))(z \land (2-\delta-z))}
\biggr)
+
\begin{cases}
O(\Delta^{1/2+\alpha} \cdot \Delta^{1/2}), & \alpha < 3/2,
\\
O(\Delta^{1/2+\alpha} \cdot (-\Delta^{1/2} \log \Delta)), & \alpha = 3/2,
\\
O(\Delta^{1/2+\alpha} \cdot \Delta^{2-\alpha}), & \alpha > 3/2,
\end{cases}
\end{align*}
\begin{align*}
&|D_{1,\delta}^2 D_{2,\delta}^2 F(y,z)|
\\
&\lesssim
\Delta^{1+\alpha}
\Biggl|
\sum_{\ell_1,\ell_2\ge1} 
f(\lambda_{\ell_1,\ell_2}\Delta) 
\sin^2 \Bigl(\frac{\pi r\ell_1\sqrt{\Delta}}{2}\Bigr)
\sin^2 \Bigl(\frac{\pi r\ell_2\sqrt{\Delta}}{2}\Bigr)
\cos(\pi \ell_1 (y+\delta))
\cos(\pi \ell_2 (z+\delta))
\Biggr|
\\
&=
O\biggl(\frac{\Delta^{1+\alpha}}{(y \land (2-\delta-y))(z \land (2-\delta-z))}\biggr)
+O(\Delta^{1+\alpha}).
\end{align*}

\end{proof}

\begin{lem}\label{lem8}
For $\alpha \in (0,2)$, $\delta=r\sqrt{\Delta}$, $y\in(0,2-\delta)$ 
and $0 \le K_1,K_2 \le 1/\delta$, it follows that
\begin{itemize}
\item[(1)]
$D_{1,\delta} D_{2,\delta}^2 F(y, K_2\delta)
=O(\Delta^{1/2+\alpha})+
O\Bigl(\frac{\Delta^\alpha}{(y \land (2-y-\delta))(K_2+1)}\Bigr)$,

\item[(2)]
$D_{1,\delta}^2 D_{2,\delta}^2 F(y, K_2\delta)
= O(\Delta^{1+\alpha}) 
+ O\Bigl(\frac{\Delta^{1/2+\alpha}}{(y \land (2-y-\delta))(K_2+1)}\Bigr)$,

\item[(3)]
$D_{1,\delta}^2 D_{2,\delta}^2 F(K_1\delta, K_2\delta)
= O(\Delta^{1+\alpha}) + O\Bigl(\frac{\Delta^\alpha}{(K_1+1)(K_2+1)}\Bigr)$.
\end{itemize}

\end{lem}
\begin{proof}
In the same way as the proof of (3) in Lemma \ref{lem7}, 
we obtain the first statement:
\begin{align*}
|D_{1,\delta} D_{2,\delta}^2 F(y,K_2 \delta)|
&=
O\biggl(\frac{\Delta^{1/2+\alpha}}{(y \land (2-\delta-y)) (K_2+1)\delta}\biggr)
+O(\Delta^{1/2+\alpha})
\\
&=
O\biggl(\frac{\Delta^{\alpha}}{(y \land (2-\delta-y)) (K_2+1)}\biggr)
+O(\Delta^{1/2+\alpha}).
\end{align*}
Similarly, the rest of the statements can be shown by using (4) in Lemma \ref{lem7}.
\end{proof}

\subsection{Covariance of triple increments}
In this subsection, we consider the covariance of the triple increments $T_{i,j,k}X$
of SPDE \eqref{2d_spde} driven by the $Q_1$-Wiener process.
For $\alpha \in (0,2)$ and $J \ge 0$, define 
\begin{equation*}
f_\alpha(x)=\frac{1-\ee^{-x}}{x^{1+\alpha}},
\quad 
f_{J,\alpha}(x) = \frac{(1-\ee^{-x})^2}{x^{1+\alpha}}\ee^{-J x}, \quad x > 0.
\end{equation*}
We can easily verify that
$f_\alpha \in \boldsymbol F_{0,\alpha}(\mathbb R_{+})$ and
$f_{J,\alpha} \in \boldsymbol F_{J/2,\alpha-1}(\mathbb R_{+})$.
Moreover, by using $f_\alpha$, 
$F_\alpha$ given in \eqref{eq-1-1} can be expressed as
\begin{equation*}
F_\alpha(y,z)=
\Delta^{1+\alpha}
\sum_{\ell_1,\ell_2\ge1}
f_\alpha(\lambda_{\ell_1,\ell_2}\Delta) \cos(\pi \ell_1 y)\cos(\pi \ell_2 z).
\end{equation*}
We define
\begin{align*}
F_{j',k'}^{j,k} &=
\sum_{\ell_{1},\ell_{2}\ge1}
\frac{1-\ee^{-\lambda_{\ell_1,\ell_2}\Delta}}{\lambda_{\ell_1,\ell_2}^{1+\alpha}}
(e_{\ell_1}^{(1)}(\widetilde y_j)-e_{\ell_1}^{(1)}(\widetilde y_{j-1}))
(e_{\ell_2}^{(2)}(\widetilde z_k)-e_{\ell_2}^{(2)}(\widetilde z_{k-1}))
\\
&\qquad\times
(e_{\ell_1}^{(1)}(\widetilde y_{j'})-e_{\ell_1}^{(1)}(\widetilde y_{j'-1}))
(e_{\ell_2}^{(2)}(\widetilde z_{k'})-e_{\ell_2}^{(2)}(\widetilde z_{k'-1})),
\\
F_{J,j',k'}^{j,k} &=
\sum_{\ell_{1},\ell_{2}\ge1}
\frac{(1-\ee^{-\lambda_{\ell_1,\ell_2}\Delta})^2}{\lambda_{\ell_1,\ell_2}^{1+\alpha}}
\ee^{-\lambda_{\ell_1,\ell_2}J \Delta}
(e_{\ell_1}^{(1)}(\widetilde y_j)-e_{\ell_1}^{(1)}(\widetilde y_{j-1}))
(e_{\ell_2}^{(2)}(\widetilde z_k)-e_{\ell_2}^{(2)}(\widetilde z_{k-1}))
\\
&\qquad\times
(e_{\ell_1}^{(1)}(\widetilde y_{j'})-e_{\ell_1}^{(1)}(\widetilde y_{j'-1}))
(e_{\ell_2}^{(2)}(\widetilde z_{k'})-e_{\ell_2}^{(2)}(\widetilde z_{k'-1})).
\end{align*}

Let $r \in (0,\infty)$.
By the following two lemmas, 
$F_{j,k}$ in \eqref{eq-1-2} and the covariance of $T_{i,j,k}X$ are calculated.
\begin{lem}\label{lem9}
For $\alpha \in (0,2)$ and $\delta = r\sqrt{\Delta}$, it follows that
\begin{align}
F_{j',k'}^{j,k}
&= \ee^{-\kappa(\widetilde y_{j-1}+\widetilde y_{j'})/2}
\ee^{-\eta(\widetilde z_{k-1}+\widetilde z_{k'})/2}
D_{1,\delta}^2 D_{2,\delta}^2 
F_\alpha(\widetilde y_{j'-1}-\widetilde y_{j},\widetilde z_{k'-1}-\widetilde z_{k})
\nonumber
\\
&\qquad
+O(\Delta^{1+\alpha})
+O \Biggl(
\Delta^{1/2+\alpha}
\biggl(
\ind_{\{j \neq j'\}} \frac{1}{|j-j'|+1}
+ \ind_{\{k \neq k'\}} \frac{1}{|k-k'|+1}
\biggr) \Biggr).
\label{C1}
\end{align}
In particular, it holds that
\begin{itemize}
\item[(1)]
$F_{j,k}
= 4\ee^{-\kappa(\widetilde y_{j-1}+\widetilde y_j)/2}
\ee^{-\eta(\widetilde z_{k-1}+\widetilde z_k)/2}
D_{1,\delta} D_{2,\delta} F_\alpha (0,0) + O (\Delta^{1+\alpha})$,

\item[(2)]
$F_{j',k'}^{j,k}
=
O \Bigl(
\Delta^{\alpha}
\Bigl(\Delta + \frac{1}{(|j-j'|+1)(|k-k'|+1)} \Bigr) \Bigr)
+ O \Bigl(
\Delta^{1/2+\alpha}
\Bigl(
\ind_{\{j \neq j'\}} \frac{1}{|j-j'|+1}
+ \ind_{\{k \neq k'\}} \frac{1}{|k-k'|+1}
\Bigr) \Bigr)$,

\item[(3)]
$F_{J,j',k'}^{j,k}
=O \Bigl(
\frac{\Delta^{\alpha}}{J+1}
\Bigl(\Delta + \frac{1}{(|j-j'|+1)(|k-k'|+1)} \Bigr) \Bigr)
+ O \Bigl(
\frac{\Delta^{1/2+\alpha}}{J+1}
\Bigl(
\ind_{\{j \neq j'\}} \frac{1}{|j-j'|+1}
+ \ind_{\{k \neq k'\}} \frac{1}{|k-k'|+1}
\Bigr) \Bigr)$.
\end{itemize}
\end{lem}

\begin{proof}
Let $g_1(x)=\ee^{-\kappa x/2}$ and $g_2(x)=\ee^{-\eta x/2}$.
It follows from (22) in Hildebrandt and Trabs \cite{Hildebrandt_Trabs2021} that
\begin{align}
&\ee^{\kappa y/2}(e_{\ell_1}^{(1)}(y+\delta)-e_{\ell_1}^{(1)}(y))
\ee^{\kappa y'/2}(e_{\ell_1}^{(1)}(y'+\delta)-e_{\ell_1}^{(1)}(y'))
\nonumber
\\
&=
D_\delta^2g_1(0)\cos(\pi\ell_1(y'-y))
-D_\delta^2 \bigl[g_1(\cdot)\cos(\pi \ell_1 (y+y'+\cdot))\bigr](0)
\nonumber
\\
&\qquad
-g_1(\delta) D_\delta^2\bigl[\cos(\pi \ell_1 (y'-y-\delta+\cdot))\bigr](0).
\label{eq-7-1}
\end{align}
For $y',z'\in [0,2)$, define
\begin{align*}
G_{0}(y,z:y',z') &= g_1(y)g_2(z)F_\alpha(y+y',z+z'),
\\
G_{1}(y,z:y',z') &= g_1(y)F_\alpha(y+y',z+z'),
\\
G_{2}(y,z:y',z') &= g_2(z)F_\alpha(y+y',z+z').
\end{align*}
Since
\begin{align*}
D_{1,\delta} D_{2,\delta}^j F_\alpha(0,0) &= 
\sum_{\ell_1,\ell_2\ge1}
\frac{1-\ee^{-\lambda_{\ell_1,\ell_2}\Delta}}{\lambda_{\ell_1,\ell_2}^{1+\alpha}}
D_\delta[\cos(\pi \ell_1 \cdot)](0) D_\delta^j [\cos(\pi \ell_2 \cdot)](0),
\\
D_{1,\delta}^2 D_{2,\delta}^j G_1(0,0:y',0) &= 
\sum_{\ell_1,\ell_2\ge1}
\frac{1-\ee^{-\lambda_{\ell_1,\ell_2}\Delta}}{\lambda_{\ell_1,\ell_2}^{1+\alpha}}
D_\delta^2[g_1(\cdot)\cos(\pi \ell_1 (y'+\cdot))](0)
D_\delta^j [\cos(\pi \ell_2 \cdot)](0),
\\
D_{1,\delta}^2 D_{2,\delta}^2 G_0(0,0:y',z') &= 
\sum_{\ell_1,\ell_2\ge1}
\frac{1-\ee^{-\lambda_{\ell_1,\ell_2}\Delta}}{\lambda_{\ell_1,\ell_2}^{1+\alpha}}
D_\delta^2[g_1(\cdot)\cos(\pi \ell_1 (y'+\cdot))](0)
D_\delta^2[g_2(\cdot)\cos(\pi \ell_2 (z'+\cdot))](0),
\end{align*}
it follows from \eqref{eq-7-1} that
\begin{align*}
&\ee^{\kappa (\widetilde y_{j-1}+\widetilde y_{j'-1})/2}
\ee^{\eta (\widetilde z_{k-1}+\widetilde z_{k'-1})/2}
F^{j,k}_{j',k'}
\nonumber
\\
&=
\sum_{\ell_1, \ell_2 \ge 1}
\frac{1-\ee^{-\lambda_{\ell_1,\ell_2}\Delta}}
{\lambda_{\ell_1,\ell_2}^{1+\alpha}}
\nonumber
\\
&\qquad\qquad\times
\ee^{\kappa \widetilde y_{j-1}/2}
(e_{\ell_1}^{(1)}(\widetilde y_j)-e_{\ell_1}^{(1)}(\widetilde y_{j-1}))
\ee^{\kappa \widetilde y_{j'-1}/2}
(e_{\ell_1}^{(1)}(\widetilde y_{j'})-e_{\ell_1}^{(1)}(\widetilde y_{j'-1}))
\nonumber
\\
&\qquad\qquad\times
\ee^{\eta \widetilde z_{k-1}/2}
(e_{\ell_2}^{(2)}(\widetilde z_k)-e_{\ell_2}^{(2)}(\widetilde z_{k-1}))
\ee^{\eta \widetilde z_{k'-1}/2}
(e_{\ell_2}^{(2)}(\widetilde z_{k'})-e_{\ell_2}^{(2)}(\widetilde z_{k'-1}))
\nonumber
\\
&= 
\sum_{\ell_1, \ell_2 \ge 1}
\frac{1-\ee^{-\lambda_{\ell_1,\ell_2}\Delta}}
{\lambda_{\ell_1,\ell_2}^{1+\alpha}}
\nonumber
\\
&\qquad\qquad
\times
\bigl(
D_\delta^2g_1(0)\cos(\pi\ell_1(\widetilde y_{j'-1}-\widetilde y_{j-1}))
\nonumber
\\
&\qquad\qquad\qquad
-D_\delta^2 \bigl[g_1(\cdot)\cos(\pi \ell_1 
(\widetilde y_{j-1}+\widetilde y_{j'-1}+\cdot))\bigr](0)
\nonumber
\\
&\qquad\qquad\qquad
-g_1(\delta) D_\delta^2\bigl[\cos(\pi \ell_1 
(\widetilde y_{j'-1}-\widetilde y_{j}+\cdot))\bigr](0)
\bigr)
\nonumber
\\
&\qquad\qquad
\times
\bigl(
D_\delta^2 g_2(0)\cos(\pi\ell_2(\widetilde z_{k'-1}-\widetilde z_{k-1}))
\nonumber
\\
&\qquad\qquad\qquad
-D_\delta^2 \bigl[g_2(\cdot)\cos(\pi \ell_2 
(\widetilde z_{k-1}+\widetilde z_{k'-1}+\cdot))\bigr](0)
\nonumber
\\
&\qquad\qquad\qquad
-g_2(\delta) D_\delta^2\bigl[\cos(\pi \ell_2 
(\widetilde z_{k'-1}-\widetilde z_{k}+\cdot))\bigr](0)
\bigr)
\nonumber
\\
&=
D_\delta^2 g_1(0) D_\delta^2 g_2(0) 
F_\alpha(\widetilde y_{j'-1}-\widetilde y_{j-1},\widetilde z_{k'-1}-\widetilde z_{k-1})
\nonumber
\\
&\qquad-
D_\delta^2 g_1(0)
\Bigl(
D_{2,\delta}^2
G_2(0,0:\widetilde y_{j'-1}-\widetilde y_{j-1},\widetilde z_{k-1}+\widetilde z_{k'-1})
\nonumber
\\
&\qquad\qquad
+g_2(\delta)D_{2,\delta}^2 
F_\alpha(\widetilde y_{j'-1}-\widetilde y_{j-1},\widetilde z_{k'-1}-\widetilde z_{k})
\Bigr)
\nonumber
\\
&\qquad-
D_\delta^2 g_2(0)
\Bigl(
D_{1,\delta}^2 
G_1(0,0:\widetilde y_{j-1}+\widetilde y_{j'-1},\widetilde z_{k'-1}-\widetilde z_{k-1})
\nonumber
\\
&\qquad\qquad
+g_1(\delta)
D_{1,\delta}^2 
F_\alpha(\widetilde y_{j'-1}-\widetilde y_{j},\widetilde z_{k'-1}-\widetilde z_{k-1})
\Bigr)
\nonumber
\\
&\qquad+
D_{1,\delta}^2D_{2,\delta}^2 
G_0(0,0:\widetilde y_{j-1}+\widetilde y_{j'-1},\widetilde z_{k-1}+\widetilde z_{k'-1})
\nonumber
\\
&\qquad
+g_2(\delta) 
D_{1,\delta}^2 D_{2,\delta}^2 
G_1(0,0:\widetilde y_{j-1}+\widetilde y_{j'-1},\widetilde z_{k'-1}-\widetilde z_{k})
\nonumber
\\
&\qquad
+g_1(\delta) D_{1,\delta}^2 D_{2,\delta}^2 
G_2(0,0:\widetilde y_{j'-1}-\widetilde y_{j},\widetilde z_{k-1}+\widetilde z_{k'-1})
\nonumber
\\
&\qquad
+g_1(\delta)g_2(\delta) D_{1,\delta}^2 D_{2,\delta}^2 
F_\alpha(\widetilde y_{j'-1}-\widetilde y_{j},\widetilde z_{k'-1}-\widetilde z_{k}).
\end{align*}
Note that $D_\delta^\ell g_1(0) = O(\Delta^{\ell/2})$,
$\widetilde y_{j}+\widetilde y_{j'}\in (b,2-b)$
and $\widetilde y_{j}-\widetilde y_{j'}=(j-j')\delta$. 
It holds from Lemma \ref{lem6} that
\begin{equation*}
F_\alpha(\widetilde y_{j'-1}-\widetilde y_{j-1},\widetilde z_{k'-1}-\widetilde z_{k-1})
= O(\Delta^{\alpha-1}),
\quad
D_{1,\delta}^2 
F_\alpha(\widetilde y_{j'-1}-\widetilde y_{j},\widetilde z_{k'-1}-\widetilde z_{k-1})
=O(\Delta^\alpha)
\end{equation*}
uniformly in $j,j',k,k'$, and we see from 
\begin{align*}
D_{1,\delta}^2 D_{2,\delta}^\ell G_1(0,0:y',z')
&=D_\delta^2 g_1(0) D_{2,\delta}^\ell F_\alpha(2\delta+y',z')
+2D_\delta g_1(0) D_{1,\delta} D_{2,\delta}^\ell F_\alpha(\delta+y',z')
\\
&\qquad
+g_1(0) D_{1,\delta}^2 D_{2,\delta}^\ell F_\alpha(y',z'),
\end{align*}
$D_{2,\delta}^2 F_\alpha(y,-\delta) = 2 D_{2,\delta} F_\alpha(y,0)$ and  
Lemmas \ref{lem6}-\ref{lem8} that
\begin{align*}
D_{1,\delta}^2 
G_1(0,0:\widetilde y_{j-1}+\widetilde y_{j'-1},\widetilde z_{k'-1}-\widetilde z_{k-1})
=O(\Delta^\alpha)
\end{align*}
uniformly in $j,j',k,k'$, and
\begin{align*}
&D_{1,\delta}^2 D_{2,\delta}^2 
G_1(0,0:\widetilde y_{j-1}+\widetilde y_{j'-1},\widetilde z_{k'-1}-\widetilde z_{k})
=O(\Delta^{1+\alpha})
+O\biggl(
\ind_{\{k \neq k'\}}\frac{\Delta^{1/2 +\alpha}}{|k-k'|+1}
\biggr)
\end{align*}
uniformly in $j,j'$.
Moreover, noting that
\begin{align*}
&D_{1,\delta}^2 D_{2,\delta}^2 G_0(0,0:y',z')
\\
&=
D_\delta^2 g_1(0)
\Bigl\{
D_\delta^2 g_2(0)F_\alpha(2\delta+y',2\delta+z')
\\
&\qquad\qquad
+2D_\delta g_2(0) D_{2,\delta} F_\alpha(2\delta+y',\delta+z')
+g_2(0) D_{2,\delta}^2 F_\alpha(2\delta+y',z')
\Bigr\}
\\
&\qquad+
2D_\delta g_1(0)
\Bigl\{
D_\delta^2 g_2(0) D_{1,\delta} F_\alpha(\delta+y',2\delta+z')
\\
&\qquad\qquad
+2D_\delta g_2(0) D_{1,\delta}D_{2,\delta} F_\alpha(\delta+y',\delta+z')
+g_2(0) 
D_{1,\delta} D_{2,\delta}^2 F_\alpha(\delta+y',z')
\Bigr\}
\\
&\qquad+
g_1(0)
\Bigl\{
D_\delta^2 g_2(0)
D_{1,\delta}^2 F_\alpha(y',2\delta+z')
\\
&\qquad\qquad
+2D_\delta g_2(0) D_{1,\delta}^2 D_{2,\delta} F_\alpha(y',\delta+z')
+g_2(0) D_{1,\delta}^2 D_{2,\delta}^2 F_\alpha(y',z')
\Bigr\},
\end{align*}
one has
\begin{equation*}
D_{1,\delta}^2D_{2,\delta}^2 
G_0(0,0:\widetilde y_{j-1}+\widetilde y_{j'-1},\widetilde z_{k-1}+\widetilde z_{k'-1})
=O(\Delta^{1+\alpha})
\end{equation*}
uniformly in $j,j',k,k'$.
Therefore, we obtain
\begin{align*}
&\ee^{\kappa (\widetilde y_{j-1}+\widetilde y_{j'-1})/2}
\ee^{\eta (\widetilde z_{k-1}+\widetilde z_{k'-1})/2}
F^{j,k}_{j',k'}
\\
&=
\ee^{-\kappa \delta/2}\ee^{-\eta \delta/2} D_{1,\delta}^2 D_{2,\delta}^2 
F_\alpha(\widetilde y_{j'-1}-\widetilde y_{j},\widetilde z_{k'-1}-\widetilde z_{k})
\\
&\qquad
+O(\Delta^{1+\alpha})
+O \Biggl(
\Delta^{1/2+\alpha}
\biggl(
\ind_{\{j \neq j'\}} \frac{1}{|j-j'|+1}
+ \ind_{\{k \neq k'\}} \frac{1}{|k-k'|+1}
\biggr) \Biggr),
\end{align*}
which together with $\delta=\widetilde y_{j'}-\widetilde y_{j'-1}
=\widetilde z_{k'}-\widetilde z_{k'-1}$ gives the desired result.

(1) 
By $F_{j,k}=F_{j,k}^{j,k}$ and 
$D_{1,\delta}^2 D_{2,\delta}^2 F_\alpha(-\delta,-\delta) 
= 4D_{1,\delta} D_{2,\delta} F_\alpha(0,0)$, we complete the proof.

(2) Since it follows from 
\eqref{eq-1-6} and Lemmas \ref{lem7} and \ref{lem8} that
\begin{align*}
D_{1,\delta}^2 D_{2,\delta}^2 
F_\alpha(\widetilde y_{j'-1}-\widetilde y_{j},\widetilde z_{k'-1}-\widetilde z_{k})
=O(\Delta^{1+\alpha})
+O\biggl(\frac{\Delta^\alpha}{(|j-j'|+1)(|k-k'|+1)}\biggr),
\end{align*}
\eqref{C1} yields the desired result.

(3) For $f \in \boldsymbol F_{J,\alpha-1}(\mathbb R_{+})$, one has 
\begin{equation*}
(J+1)x^\alpha f(x),\ (J+1)x^{1+\alpha} f'(x),\ (J+1)x^{2+\alpha} f''(x)
\lesssim (J+1)x \ee^{-J x} \lesssim 1
\quad(x \to 0),
\end{equation*}
\begin{equation*}
(J+1)x^{1+\alpha} f(x),\ (J+1)x^{2+\alpha} f'(x),\ (J+1)x^{3+\alpha} f''(x)
\lesssim (J+1) \ee^{-J x} \lesssim 1
\quad(x \to \infty)
\end{equation*}
and thus 
we see that
$(J+1) f \in \boldsymbol F_{0,\alpha}(\mathbb R_{+})$. 
Therefore, it holds from 
$f_{J,\alpha} \in \boldsymbol F_{J/2,\alpha-1}(\mathbb R_{+})$ that
\begin{align*}
(J+1)F_{J,j',k'}^{j,k} &=
\Delta^{1+\alpha} \sum_{\ell_{1},\ell_{2}\ge1}
(J+1)f_{J,\alpha}(\lambda_{\ell_1,\ell_2}\Delta)
(e_{\ell_1}^{(1)}(\widetilde y_j)-e_{\ell_1}^{(1)}(\widetilde y_{j-1}))
(e_{\ell_2}^{(2)}(\widetilde z_k)-e_{\ell_2}^{(2)}(\widetilde z_{k-1}))
\\
&\qquad\times
(e_{\ell_1}^{(1)}(\widetilde y_{j'})-e_{\ell_1}^{(1)}(\widetilde y_{j'-1}))
(e_{\ell_2}^{(2)}(\widetilde z_{k'})-e_{\ell_2}^{(2)}(\widetilde z_{k'-1}))
\\
&=
O \biggl(
\Delta^{\alpha}
\biggl(\Delta + \frac{1}{(|j-j'|+1)(|k-k'|+1)} \biggr) \biggr)
\\
&\qquad
+ O \biggl(
\Delta^{1/2+\alpha}
\biggl(
\ind_{\{j \neq j'\}} \frac{1}{|j-j'|+1}
+ \ind_{\{k \neq k'\}} \frac{1}{|k-k'|+1}
\biggr) \biggr)
\end{align*}
in the same way as (2).
\end{proof}

\begin{lem}\label{lem10}
It holds that
\begin{equation*}
\cov[T_{i,j,k}X, T_{i',j',k'}X]
= \sigma^2 \times 
\begin{cases}
F_{j',k'}^{j,k}-\frac{1}{2}F_{2(i-1),j',k'}^{j,k}, & i = i',
\\
-\frac{1}{2}(F_{|i-i'|-1,j',k'}^{j,k} + F_{i+i'-2,j',k'}^{j,k}), & i \neq i'.
\end{cases}
\end{equation*}
In particular, it follows that for $\alpha \in (0,2)$ and $\delta=r \sqrt{\Delta}$,
\begin{align}
\cov[T_{i,j,k}X, T_{i',j',k'}X]
&=
O \biggl(
\frac{\Delta^{\alpha}}{|i-i'|+1}
\biggl(
\Delta
+ \frac{1}{(|j-j'|+1)(|k-k'|+1)}
\biggr) \biggr)
\nonumber
\\
&\qquad
+O \biggl(
\frac{\Delta^{\alpha+1/2}}{|i-i'|+1}
\biggl(
\ind_{\{j \neq j'\}} \frac{1}{|j-j'|+1}
+ \ind_{\{k \neq k'\}} \frac{1}{|k-k'|+1}
\biggr) \biggr)
\label{C2}
\end{align}
and thus
\begin{equation*}
\sum_{k,k'=1}^{m_2} \sum_{j,j'=1}^{m_1} \sum_{i,i'=1}^N 
\cov[T_{i,j,k}X, T_{i',j',k'}X]^2
=O(m N \Delta^{2\alpha}).
\end{equation*}
\end{lem}
\begin{proof}
By a simple calculation, we have
\begin{align*}
& \cov\bigl[x_{\ell_1,\ell_2}(t_i)-x_{\ell_1,\ell_2}(t_{i-1}),
x_{\ell_1,\ell_2}(t_{i'})-x_{\ell_1,\ell_2}(t_{{i'}-1})
\bigr]
\\
&= \cov[A_{i,\ell_1,\ell_2}+B_{i,\ell_1,\ell_2},
A_{i',\ell_1,\ell_2}+B_{i',\ell_1,\ell_2}]
\\
&= \cov[B_{i,\ell_1,\ell_2}, B_{i',\ell_1,\ell_2}]
\\
&= 
\sigma^2
\times
\begin{cases}
\frac{1-\ee^{-\lambda_{\ell_1,\ell_2}\Delta}}{\lambda_{\ell_1,\ell_2}^{1+\alpha}}
\Bigl(
1-\frac{1-\ee^{-\lambda_{\ell_1,\ell_2}\Delta}}{2}
\ee^{-2\lambda_{\ell_1,\ell_2}(i-1)\Delta}
\Bigr),
& i = i',
\\
-\frac{(1-\ee^{-\lambda_{\ell_1,\ell_2}\Delta})^2}{2\lambda_{\ell_1,\ell_2}^{1+\alpha}}
(\ee^{-\lambda_{\ell_1,\ell_2}(|i'-i|-1)\Delta} 
+ \ee^{-\lambda_{\ell_1,\ell_2}(i'+i-2)\Delta}),
& i \neq i'
\end{cases}
\end{align*}
and hence
\begin{align*}
\cov[T_{i,j,k}X, T_{i',j',k'}X]  
&=
\sum_{\ell_1, \ell_2 \ge 1}
\cov\bigl[x_{\ell_1,\ell_2}(t_i)-x_{\ell_1,\ell_2}(t_{i-1}),
x_{\ell_1,\ell_2}(t_{i'})-x_{\ell_1,\ell_2}(t_{{i'}-1})
\bigr]
\nonumber
\\
&\qquad\times
(e_{\ell_1}^{(1)}(\widetilde y_j)-e_{\ell_1}^{(1)}(\widetilde y_{j-1}))
(e_{\ell_2}^{(2)}(\widetilde z_k)-e_{\ell_2}^{(2)}(\widetilde z_{k-1}))
\nonumber
\\
&\qquad\times
(e_{\ell_1}^{(1)}(\widetilde y_{j'})-e_{\ell_1}^{(1)}(\widetilde y_{j'-1}))
(e_{\ell_2}^{(2)}(\widetilde z_{k'})-e_{\ell_2}^{(2)}(\widetilde z_{k'-1}))
\nonumber
\\
&=
\sigma^2 \times 
\begin{cases}
F_{j',k'}^{j,k}-\frac{1}{2}F_{2(i-1),j',k'}^{j,k}, & i = i',
\\
-\frac{1}{2}(F_{|i-i'|-1,j',k'}^{j,k} + F_{i+i'-2,j',k'}^{j,k}), & i \neq i'.
\end{cases}
\end{align*}
Since it follows from Lemma \ref{lem9} that 
\begin{align*}
F_{j',k'}^{j,k}-\frac{1}{2}F_{2(i-1),j',k'}^{j,k}
&= 
O \biggl(
\Delta^{\alpha}
\biggl(\Delta + \frac{1}{(|j-j'|+1)(|k-k'|+1)} \biggr) \biggr)
\\
&\qquad 
+ O \biggl(
\Delta^{1/2+\alpha}
\biggl(
\ind_{\{j \neq j'\}} \frac{1}{|j-j'|+1}
+ \ind_{\{k \neq k'\}} \frac{1}{|k-k'|+1}
\biggr) \biggr),
\\
F_{|i-i'|-1,j',k'}^{j,k} + F_{i+i'-2,j',k'}^{j,k}
&= 
O \biggl(
\frac{\Delta^{\alpha}}{|i-i'|+1}
\biggl(\Delta + \frac{1}{(|j-j'|+1)(|k-k'|+1)} \biggr) \biggr)
\\
&\qquad 
+ O \biggl(
\frac{\Delta^{1/2+\alpha}}{|i-i'|+1}
\biggl(
\ind_{\{j \neq j'\}} \frac{1}{|j-j'|+1}
+ \ind_{\{k \neq k'\}} \frac{1}{|k-k'|+1}
\biggr) \biggr)
\end{align*}
uniformly in $j,k,j',k'$, we obtain \eqref{C2}.
Furthermore, we have
\begin{align*}
&\frac{1}{m N\Delta^{2\alpha}}
\sum_{k,k'=1}^{m_2}\sum_{j,j'=1}^{m_1}\sum_{i,i' =1}^N 
\cov[T_{i,j,k}X, T_{i',j',k'}X]^2
\\
&=
O\Biggl( 
\frac{1}{N}\sum_{i,i'=1}^N \frac{1}{(|i-i'|+1)^2}
\biggl(
m\Delta^2
+\frac{1}{m_1}\sum_{j,j'=1}^{m_1} \frac{1}{(|j-j'|+1)^2}
\frac{1}{m_2}\sum_{k,k'=1}^{m_2} \frac{1}{(|k-k'|+1)^2}
\biggr)
\Biggr)
\\
&\qquad
+O\Biggl( 
\frac{1}{N}\sum_{i,i'=1}^N \frac{1}{(|i-i'|+1)^2}
\biggl(
m_2\Delta
\times\frac{1}{m_1}\sum_{j,j'=1}^{m_1} \frac{1}{(|j-j'|+1)^2}
\\
&\qquad\qquad
+m_1\Delta
\times \frac{1}{m_2}\sum_{k,k'=1}^{m_2} \frac{1}{(|k-k'|+1)^2}
\biggr)
\Biggr)
\\
&=O(1)+O(\sqrt{\Delta})
\\
&=O(1).
\end{align*}
\end{proof}

\subsection{Auxiliary results for consistency of the estimator}
The following lemmas are used to show the consistency of the estimator.
\begin{lem}\label{lem11}
Let $\alpha \in [1,2)$ and $r \in (0,\infty)$. 
For $\psi_{r,\alpha}$ given in \eqref{psi}, 
the function
\begin{equation}\label{eq-b01}
\theta_2 \mapsto \frac{\psi_{r,\alpha}(\theta_2)}{\psi_{r/\sqrt{2},\alpha}(\theta_2)}
\end{equation}
is injective for $\theta_2>0$.
\end{lem}
\begin{proof}
Note that
\begin{equation*}
\psi_{r,\alpha}(\theta_2)
=\frac{2(r')^{2\alpha}}{\theta_2 \pi}
\int_0^\infty 
\frac{1-\ee^{-(x/r')^2}}{x^{1+2\alpha}}(J_0(\sqrt{2} x)-2J_0(x)+1) \dd x,
\end{equation*}
where $r'=r/\sqrt{\theta_2}$.
By setting 
\begin{equation*}
\psi_{\alpha}(t)
=\int_0^\infty 
\frac{1-\ee^{-t x^2}}{x^{1+2\alpha}} (J_0(\sqrt{2}x)-2J_0(x)+1) \dd x,
\end{equation*}
the function \eqref{eq-b01} can be represented as
\begin{equation*}
\theta_2 \mapsto \frac{\psi_{r,\alpha}(\theta_2)}
{\psi_{r/\sqrt{2},\alpha}(\theta_2)}
=2^{\alpha} \times
\frac{\psi_\alpha(\theta_2/r^2)}{\psi_\alpha(2\theta_2/r^2)}.
\end{equation*}
Therefore, it is enough to show that the function
\begin{equation}\label{eq-8-1}
t \mapsto \frac{\psi_{\alpha}(t)}{\psi_{\alpha}(2t)}
\end{equation}
is injective for $t>0$.

Let $\Gamma(s) = \int_0^\infty x^{s-1}\ee^{-x}\dd x$ for $s > 0$. 
Since
\begin{align*}
J_0(\sqrt{2}x)-2J_0(x)+1
&=
\Biggl(1+\sum_{k\ge1} \frac{(-1)^k}{(k!)^2} \biggl(\frac{x^2}{2}\biggr)^{k}\Biggr)
-2\Biggl(
1+\sum_{k\ge1} \frac{(-1)^k}{(k!)^2} \biggl(\frac{x^2}{4}\biggr)^{k}
\Biggr)
+1
\\
&=\sum_{k\ge2} \frac{(-1)^k}{(k!)^2} \biggl(\frac{1}{2^k}-\frac{2}{4^k}\biggr)x^{2k},
\end{align*}
it holds that for $t>0$, 
\begin{align*}
\psi_{\alpha}'(t)
&=
\int_0^\infty 
\frac{\ee^{-t x^2}}{x^{2\alpha-1}} (J_0(\sqrt{2}x)-2J_0(x)+1) \dd x
\\
&=
\int_0^\infty 
\sum_{k\ge2} \frac{(-1)^k}{(k!)^2} \biggl(\frac{1}{2^k}-\frac{2}{4^k}\biggr)
x^{2(k-\alpha)+1} \ee^{-t x^2}\dd x
\\
&=
\sum_{k\ge2} \frac{(-1)^k}{(k!)^2} \biggl(\frac{1}{2^k}-\frac{2}{4^k}\biggr)
\int_0^\infty 
x^{2(k-\alpha)+1} \ee^{-t x^2}\dd x
\\
&=
\sum_{k\ge2} \frac{(-1)^k}{(k!)^2} \biggl(\frac{1}{2^{k+1}}-\frac{1}{4^k}\biggr)
\int_0^\infty 
z^{(k-\alpha+1)-1} \ee^{-t z} \dd z
\\
&=
\sum_{k\ge2} \frac{(-1)^k}{(k!)^2} \biggl(\frac{1}{2^{k+1}}-\frac{1}{4^k}\biggr)
\frac{\Gamma(k-\alpha+1)}{t^{k-\alpha+1}}.
\end{align*}
Let 
\begin{equation*}
\varphi_\alpha(s) = 
\sum_{k\ge2} \frac{\Gamma(k-\alpha+1)}{\Gamma(k+1)} \frac{(-1)^k s^k}{k!}
\end{equation*}
for $s > 0$. Since
\begin{equation*}
\int_0^s x^{k-\alpha} (s-x)^{\alpha-1} \dd x 
= \frac{\Gamma(k-\alpha+1)\Gamma(\alpha)}{\Gamma(k+1)}s^k,
\end{equation*}
one has
\begin{align*}
\varphi_{\alpha}(s) 
&= \sum_{k\ge2} 
\frac{(-1)^k}{\Gamma(\alpha)k!}
\int_0^s x^{k-\alpha} (s-x)^{\alpha-1} \dd x 
\\
&= \frac{1}{\Gamma(\alpha)}
\int_0^s x^{-\alpha} (s-x)^{\alpha-1}
\sum_{k\ge2} 
\frac{(-x)^k}{k!} \dd x 
\\
&= \frac{1}{\Gamma(\alpha)}
\int_0^s x^{-\alpha} (s-x)^{\alpha-1}(\ee^{-x}-1+x) \dd x. 
\end{align*}
Therefore,
\begin{align*}
\psi_{\alpha}'(t) &= 
t^{\alpha-1}
\Biggl(
\frac{1}{2}\sum_{k\ge2} \frac{\Gamma(k-\alpha+1)}{\Gamma(k+1)}
\frac{(-1)^k}{k!} \Bigl(\frac{1}{2t}\Bigr)^k
-\sum_{k\ge2} \frac{\Gamma(k-\alpha+1)}{\Gamma(k+1)}
\frac{(-1)^k}{k!} \Bigl(\frac{1}{4t}\Bigr)^k
\Biggr)
\\
&=
t^{\alpha-1} 
\biggl(\frac{\varphi_\alpha(1/2t)}{2} - \varphi_\alpha(1/4t) \biggr)
\\
&= \frac{t^{\alpha-1}}{\Gamma(\alpha)} 
\Biggl\{
\frac{1}{2} \int_0^{1/2t} 
x^{-\alpha} \Bigl(\frac{1}{2t}-x\Bigr)^{\alpha-1}
(\ee^{-x}-1+x) \dd x
\\
&\qquad\qquad
-\int_0^{1/4t} 
x^{-\alpha} \Bigl(\frac{1}{4t}-x\Bigr)^{\alpha-1}
(\ee^{-x}-1+x) \dd x
\Biggr\}
\\
&= \frac{1}{\Gamma(\alpha)} 
\Biggl\{
\frac{1}{2^{\alpha}} \int_0^{1/2t} 
x^{-\alpha} (1-2t x)^{\alpha-1}(\ee^{-x}-1+x) \dd x
\\
&\qquad\qquad
-\frac{1}{2\cdot4^{\alpha-1}}\int_0^{1/2t} 
\Bigl(\frac{x}{2}\Bigr)^{-\alpha} (1-2t x)^{\alpha-1}
\Bigl(\ee^{-x/2}-1+\frac{x}{2}\Bigr) \dd x
\Biggr\}
\\
&= \frac{1}{2^{\alpha}\Gamma(\alpha)} 
\int_0^{1/2t} 
x^{-\alpha} (1-2t x)^{\alpha-1} (1-\ee^{-x/2})^2 \dd x
\\
&= \frac{1}{2\Gamma(\alpha)} 
\int_{t}^\infty 
x^{\alpha-2} \Bigl(1-\frac{t}{x}\Bigr)^{\alpha-1} 
(1-\ee^{-1/4x})^2 \dd x.
\end{align*}
We see from $\psi_\alpha(0)=0$ that
\begin{align*}
\psi_\alpha(t) &= \int_0^t \psi_\alpha'(s) \dd s
= \frac{1}{2\Gamma(\alpha)} 
\int_0^t 
\int_{s}^\infty 
x^{\alpha-2} \Bigl(1-\frac{s}{x}\Bigr)^{\alpha-1} 
(1-\ee^{-1/4x})^2 \dd x \dd s
\\
&=
\frac{1}{2\Gamma(\alpha)} 
\Biggl\{
\int_0^t 
x^{\alpha-2} (1-\ee^{-1/4x})^2 
\int_0^x 
\Bigl(1-\frac{s}{x}\Bigr)^{\alpha-1} \dd s \dd x
\\
&\qquad\qquad
+\int_t^\infty 
x^{\alpha-2} (1-\ee^{-1/4x})^2 
\int_0^t 
\Bigl(1-\frac{s}{x}\Bigr)^{\alpha-1} \dd s \dd x
\Biggr\}
\\
&=
\frac{1}{2\Gamma(\alpha+1)} 
\Biggl\{
\int_0^t 
x^{\alpha-1} (1-\ee^{-1/4x})^2 \dd x
\\
&\qquad\qquad
+\int_t^\infty 
x^{\alpha-1} (1-\ee^{-1/4x})^2 
\biggl\{1-\Bigl(1-\frac{t}{x}\Bigr)^{\alpha}\biggr\} \dd x
\Biggr\}.
\end{align*}
Hence, noting that 
\begin{align*}
\psi_\alpha(2t) &=
\frac{1}{2\Gamma(\alpha+1)} 
\Biggl\{
\int_0^{2t} 
x^{\alpha-1} (1-\ee^{-1/4x})^2 \dd x
\\
&\qquad\qquad
+\int_{2t}^\infty 
x^{\alpha-1} (1-\ee^{-1/4x})^2 
\biggl\{1-\Bigl(1-\frac{2t}{x}\Bigr)^{\alpha}\biggr\} \dd x
\Biggr\}
\\
&=
\frac{2^{\alpha-1}}{\Gamma(\alpha+1)} 
\Biggl\{
\int_0^{t} 
x^{\alpha-1} (1-\ee^{-1/8x})^2 \dd x
\\
&\qquad\qquad
+\int_{t}^\infty 
x^{\alpha-1} (1-\ee^{-1/8x})^2 
\biggl\{1-\Bigl(1-\frac{t}{x}\Bigr)^{\alpha}\biggr\} \dd x
\Biggr\}
\end{align*}
and
\begin{align*}
\psi_\alpha'(2t) &=
\frac{1}{2\Gamma(\alpha)} 
\int_{2t}^\infty 
x^{\alpha-1} \Bigl(1-\frac{2t}{x}\Bigr)^{\alpha-1} 
(1-\ee^{-1/4x})^2 \dd x
\\
&=
\frac{2^{\alpha-2}}{\Gamma(\alpha)} 
\int_{t}^\infty 
x^{\alpha-2} \Bigl(1-\frac{t}{x}\Bigr)^{\alpha-1} 
(1-\ee^{-1/8x})^2 \dd x,
\end{align*}
one has
\begin{align*}
\psi_{\alpha}'(t) \psi_{\alpha}(2t) -2\psi_{\alpha}(t) \psi_{\alpha}'(2t)
=\frac{2^{\alpha-2}}{\Gamma(\alpha)\Gamma(\alpha+1)}
(I_{\alpha,1}(t)+I_{\alpha,2}(t)),
\end{align*}
where
\begin{align*}
I_{\alpha,1}(t) &= 
\int_{t}^\infty 
x^{\alpha-2} \Bigl(1-\frac{t}{x}\Bigr)^{\alpha-1} 
(1-\ee^{-1/4x})^2 \dd x
\int_0^{t} x^{\alpha-1} (1-\ee^{-1/8x})^2 \dd x
\\
&\qquad-
\int_0^t x^{\alpha-1} (1-\ee^{-1/4x})^2 \dd x
\int_{t}^\infty 
x^{\alpha-2} \Bigl(1-\frac{t}{x}\Bigr)^{\alpha-1} 
(1-\ee^{-1/8x})^2 \dd x,
\\
I_{\alpha,2}(t) &=
\int_{t}^\infty 
x^{\alpha-2} \Bigl(1-\frac{t}{x}\Bigr)^{\alpha-1} 
(1-\ee^{-1/4x})^2 \dd x
\\
&\qquad\qquad\times
\int_{t}^\infty 
x^{\alpha-1} (1-\ee^{-1/8x})^2 
\biggl\{1-\Bigl(1-\frac{t}{x}\Bigr)^{\alpha}\biggr\} \dd x
\\
&\qquad-
\int_t^\infty 
x^{\alpha-1} (1-\ee^{-1/4x})^2 
\biggl\{1-\Bigl(1-\frac{t}{x}\Bigr)^{\alpha}\biggr\} \dd x
\\
&\qquad\qquad\times
\int_{t}^\infty 
x^{\alpha-2} \Bigl(1-\frac{t}{x}\Bigr)^{\alpha-1} 
(1-\ee^{-1/8x})^2 \dd x.
\end{align*}
Since $\ee^{-1/8x}-\ee^{-1/8y}> 0$ for $0<y < x$, it follows that
\begin{align*}
I_{\alpha,1}(t) &=
\int_t^\infty \int_0^t 
x^{\alpha-2} \Bigl(1-\frac{t}{x}\Bigr)^{\alpha-1} 
y^{\alpha-1}
\\
&\qquad
\times
\{(1-\ee^{-1/4x})^2(1-\ee^{-1/8y})^2
-(1-\ee^{-1/8x})^2(1-\ee^{-1/4y})^2\}
\dd y \dd x
\\
&=
\int_t^\infty \int_0^t 
x^{\alpha-2} \Bigl(1-\frac{t}{x}\Bigr)^{\alpha-1} 
y^{\alpha-1}
(1-\ee^{-1/8x})^2(1-\ee^{-1/8y})^2 
\\
&\qquad
\times
(2+\ee^{-1/8x}+\ee^{-1/8y})
(\ee^{-1/8x}-\ee^{-1/8y})
\dd y \dd x
\\
&  > 0.
\end{align*}
Moreover, it holds that
\begin{align*}
I_{\alpha,2}(t) &=
\int_{t}^\infty \int_{t}^\infty 
x^{\alpha-2} \Bigl(1-\frac{t}{x}\Bigr)^{\alpha-1} y^{\alpha-1} 
\biggl\{1-\Bigl(1-\frac{t}{y}\Bigr)^{\alpha}\biggr\} 
\\
&\qquad
\times
\{(1-\ee^{-1/4x})^2(1-\ee^{-1/8y})^2
-(1-\ee^{-1/8x})^2(1-\ee^{-1/4y})^2\}
\dd y \dd x
\\
&=
\int_{t}^\infty \int_{t}^\infty 
x^{\alpha-2} \Bigl(1-\frac{t}{x}\Bigr)^{\alpha-1} 
y^{\alpha-1} 
\biggl\{1-\Bigl(1-\frac{t}{y}\Bigr)^{\alpha}\biggr\} 
\\
&\qquad
\times
(1-\ee^{-1/8x})^2(1-\ee^{-1/8y})^2 
(2+\ee^{-1/8x}+\ee^{-1/8y})
(\ee^{-1/8x}-\ee^{-1/8y})
\dd y \dd x
\\
&=:
\int_{t}^\infty \int_{t}^\infty g_{\alpha,t}(x,y) \dd y \dd x
\\
&=
\iint_{D_t} (g_{\alpha,t}(x,y)+g_{\alpha,t}(y,x)) \dd x \dd y,
\end{align*}
where $D_t=\{(x,y) \in \mathbb R_{+}^2|x > y > t\}$. 
By setting
\begin{equation*}
h_{\alpha,t}(x) = x\biggl\{ \Bigl(1-\frac{t}{x}\Bigr)^{1-\alpha} -1 \biggr\} 
\end{equation*}
for $\alpha\ge1$ and $x> t$, it follows that
\begin{align*}
&g_{\alpha,t}(x,y)+g_{\alpha,t}(y,x)
\\
&=x^{\alpha-2}y^{\alpha-2}
\Biggl\{
\Bigl(1-\frac{t}{x}\Bigr)^{\alpha-1}y
\biggl\{1-\Bigl(1-\frac{t}{y}\Bigr)^{\alpha}\biggr\}
-\Bigl(1-\frac{t}{y}\Bigr)^{\alpha-1}x
\biggl\{1-\Bigl(1-\frac{t}{x}\Bigr)^{\alpha}\biggr\}
\Biggr\}
\\
&\qquad
\times
(1-\ee^{-1/8x})^2(1-\ee^{-1/8y})^2 
(2+\ee^{-1/8x}+\ee^{-1/8y})
(\ee^{-1/8x}-\ee^{-1/8y})
\\
&=x^{\alpha-2}y^{\alpha-1}
\Bigl(1-\frac{t}{x}\Bigr)^{\alpha-1}
\biggl\{
1-\Bigl(1-\frac{t}{y}\Bigr)^{\alpha-1}
\biggr\}
\biggl( 1-\frac{h_{\alpha,t}(x)}{h_{\alpha,t}(y)} \biggr)
\\
&\qquad
\times
(1-\ee^{-1/8x})^2(1-\ee^{-1/8y})^2 
(2+\ee^{-1/8x}+\ee^{-1/8y})
(\ee^{-1/8x}-\ee^{-1/8y}).
\end{align*}
Using the Bernoulli inequality:
\begin{equation*}
(1+z)^\beta \ge 1+\beta z
\end{equation*}
for $\beta \ge 1$ and $z>-1$,
we see that
$h_{\alpha,t}'(x) = (1-\frac{t}{x})^{-\alpha}(1-\frac{\alpha t}{x})-1 \le 0$ for $x>t$
and $h_{\alpha,t}(x)$ is monotonically decreasing for $x>t$ when $\alpha \ge 1$. 
Therefore, we obtain that $g_{\alpha,t}(x,y)+g_{\alpha,t}(y,x) \ge 0$, that is, 
$I_{\alpha,2}(t) \ge 0$.
Consequently, we have
$\psi_{\alpha}'(t)\psi_{\alpha}(2t)-2\psi_{\alpha}(t)\psi_{\alpha}'(2t) >0$,
which implies that the function \eqref{eq-8-1} is injective if $\alpha \in [1,2)$.
\end{proof}

\begin{lem}\label{lem12}
Let $\alpha \in (0,1)$ and $r \in (0,\infty)$. 
Then, the function \eqref{eq-b01} is injective for 
$\theta_2> -\frac{r^2}{8 \log(2^{\alpha/2}-1)}$.
\end{lem}
\begin{proof}
It suffices to show that 
$\psi_{\alpha}'(t) \psi_{\alpha}(2t) -2\psi_{\alpha}(t) \psi_{\alpha}'(2t) > 0$
for $t> -\frac{1}{8 \log(2^{\alpha/2}-1)}$.
Let 
\begin{equation*}
G_\alpha(t) = 
\int_t^\infty (x-t)^{\alpha-1} (1-\ee^{-1/4 x})^2 \dd x,
\quad
F_\alpha(t) = 
\frac{G_\alpha(0)-G_\alpha(t)}{t \psi_\alpha'(t)}.
\end{equation*}
Noting that
\begin{align*}
\psi_\alpha(t) &= 
\frac{1}{2\Gamma(\alpha+1)} 
\Biggl\{
\int_0^\infty x^{\alpha-1} (1-\ee^{-1/4x})^2 \dd x
\\
&\qquad\qquad
-\int_t^\infty 
x^{\alpha-1} \Bigl(1-\frac{t}{x}\Bigr)^{\alpha} (1-\ee^{-1/4x})^2 \dd x
\Biggr\}
\\
&=
\frac{G_\alpha(0)}{2\Gamma(\alpha+1)} -\psi_{\alpha+1}'(t),
\\
\psi_{\alpha+1}'(t) 
&=
\frac{1}{2\Gamma(\alpha+1)} 
\int_t^\infty 
x^{\alpha-1} \Bigl(1-\frac{t}{x}\Bigr)^{\alpha-1} (1-\ee^{-1/4x})^2 \dd x
\\
&\qquad
-\frac{t}{2\Gamma(\alpha+1)} 
\int_t^\infty 
x^{\alpha-2} \Bigl(1-\frac{t}{x}\Bigr)^{\alpha-1} (1-\ee^{-1/4x})^2 \dd x
\\
&= \frac{G_\alpha(t)}{2\Gamma(\alpha+1)} -\frac{t\psi_\alpha'(t)}{\alpha},
\end{align*}
we obtain
\begin{align*}
\psi_\alpha(t) &= \frac{G_\alpha(0)-G_\alpha(t)}{2\Gamma(\alpha+1)}
+\frac{t\psi_\alpha'(t)}{\alpha}
=t\psi_\alpha'(t) 
\biggl(
\frac{F_\alpha(t)}{2\Gamma(\alpha+1)}+\frac{1}{\alpha}
\biggr).
\end{align*}
Therefore, it follows that 
\begin{align*}
\psi_{\alpha}'(t) \psi_{\alpha}(2t) -2\psi_{\alpha}(t) \psi_{\alpha}'(2t)
&= 2t\psi_{\alpha}'(t)\psi_\alpha'(2t) 
\biggl(
\frac{F_\alpha(2t)}{2\Gamma(\alpha+1)}+\frac{1}{\alpha}
\biggr)
\\
&\qquad
- 2t\psi_\alpha'(t) \psi_{\alpha}'(2t)
\biggl(
\frac{F_\alpha(t)}{2\Gamma(\alpha+1)}+\frac{1}{\alpha}
\biggr)
\\
&= 
\frac{t\psi_{\alpha}'(t)\psi_\alpha'(2t)}{\Gamma(\alpha+1)}
(F_\alpha(2t)-F_\alpha(t)).
\end{align*}
Notice that $G_\alpha(t)$ is strictly monotonically decreasing for $t>0$, that is,
\begin{equation}\label{eq-9-1}
G_\alpha(0)-G_\alpha(2t) > G_\alpha(0)-G_\alpha(t) >0.
\end{equation}
Since 
\begin{align*}
\psi_\alpha'(t) &=
\frac{t^{\alpha-1}}{2\Gamma(\alpha)}
\int_0^1 x^{-\alpha}(1-x)^{\alpha-1}(1-\ee^{-x/4t})^2 \dd x,
\\
\psi_\alpha'(2t) &=
\frac{2^{\alpha-2}t^{\alpha-1}}{\Gamma(\alpha)}
\int_0^1 x^{-\alpha}(1-x)^{\alpha-1}(1-\ee^{-x/8t})^2 \dd x
\end{align*}
and $(1+\ee^{-1/8t})^2 > 2^\alpha$ for $t > -\frac{1}{8 \log(2^{\alpha/2}-1)}$,
we have
\begin{align*}
\psi_\alpha'(t) &=
\frac{t^{\alpha-1}}{2\Gamma(\alpha)}
\int_0^1 x^{-\alpha}(1-x)^{\alpha-1}(1-\ee^{-x/8t})^2(1+\ee^{-x/8t})^2 \dd x
\\
&>
\frac{t^{\alpha-1}}{2\Gamma(\alpha)}(1+\ee^{-1/8t})^2
\int_0^1 x^{-\alpha}(1-x)^{\alpha-1}(1-\ee^{-x/8t})^2 \dd x
\\
&> 2\psi_\alpha'(2t),
\end{align*}
which together with \eqref{eq-9-1} yields that
$F_\alpha(2t)-F_\alpha(t)>0$ for $t > -\frac{1}{8 \log(2^{\alpha/2}-1)}$. 
\end{proof}


\end{document}